\documentclass[11pt]{article}

\usepackage{amsmath,amssymb,amsthm,hyperref,color}
\usepackage{enumerate}
\usepackage{verbatim}
\usepackage{bookmark}

\usepackage[T1]{fontenc}
\usepackage{libertine} 
\usepackage[libertine]{newtxmath}
\usepackage{tikz}
\tikzset{overlap/.style={fill=yellow!30},
    block wave/.style={thick},
    major tick/.style={semithick},
    axis label/.style={anchor=west},
    x tick label/.style={anchor=north, minimum width=7mm},
    y tick label/.style={anchor=east},
}
\usetikzlibrary{decorations.pathreplacing}
\usepackage{subcaption}

\hypersetup{colorlinks=true,citecolor=blue, linkcolor=blue, urlcolor=blue}{\tiny }
\usepackage[margin = 1in]{geometry}

\newtheorem{lemma}{Lemma}[section]
\newtheorem{theorem}[lemma]{Theorem}

\newtheorem{claim}[lemma]{Claim}
\newtheorem{fact}[lemma]{Fact}

\newtheorem{corollary}[lemma]{Corollary}

\theoremstyle{remark}
\newtheorem{remark}{Remark}
\theoremstyle{definition}

\DeclareMathOperator*{\E}{{\rm E}}
\DeclareMathOperator*{\var}{{\rm Var}}
\DeclareMathOperator*{\cov}{{\rm Cov}}

\usepackage{bbm}

\DeclareMathOperator*{\e}{{\rm e}}

\def\la{\zeta}
\def\ls{\zeta_{\textsc{l}}}
\def\lS{\zeta_{\textsc{r}}}
\def\lc{\zeta_{\textnormal{\textsc{cr}}}}
\newcommand{\taumix}{$\tau_{\rm mix}$}
\newcommand{\taumixCM}{\taumix^{\mathrm{CM}}}
\newcommand{\taumixGD}{\taumix^{\mathrm{GD}}}
\def\gs{\mathrm{gap}}
\def\gsm{\gs^{-1}}

\def\taumix{\tau_{\rm mix}}
\def\Tcoup{T_{\rm coup}}

\def\on{\omega (n)}

\renewcommand{\epsilon}{\varepsilon}

\newcommand{\Z}{\mathbb{Z}}
\newcommand{\R}{\mathbb{R}}
\newcommand{\N}{\mathbb{N}}

\newcommand{\GD}{\mathcal{M}}

\author{
	Antonio Blanca\thanks{Pennsylvania State University.
		Email: ablanca@cse.psu.edu.
		Research supported in part by NSF grant CCF-1850443.}
	\and
	Alistair Sinclair\thanks{UC Berkeley.
		Email: sinclair@cs.berkeley.edu.
		Research supported in part by NSF grant CCF-1815328.}
	\and
	Xusheng Zhang\thanks{Pennsylvania State University.
		Email: xzz5349@psu.edu.
		Research supported in part by NSF grant CCF-1850443.}
}
\date{\today}

\begin{document}

\title{The Critical Mean-field Chayes-Machta Dynamics}
\maketitle

\begin{abstract}
	The random-cluster model is a unifying framework for studying random graphs, spin systems and electrical networks that
	plays a fundamental role in designing efficient Markov Chain Monte Carlo (MCMC) sampling algorithms for the classical ferromagnetic Ising and Potts models. 
	
	In this paper, we study a natural non-local Markov chain known as the {\it Chayes-Machta dynamics\/} for the mean-field case of the random-cluster model, where the underlying graph is the complete graph on $n$ vertices.
	The random-cluster model is parametrized by an {\it edge probability}~$p$ and a {\it cluster weight}~$q$. Our focus is on the critical regime: $p = p_c(q)$ and $q \in (1,2)$, where $p_c(q)$ is the threshold corresponding to the order-disorder phase transition of the model. We show that the mixing time of the Chayes-Machta dynamics is $O(\log n \cdot  \log \log n)$ in this parameter regime, which reveals that the dynamics does not undergo an exponential slowdown at criticality, a surprising fact that had been predicted (but not proved) by statistical physicists. 
	This also provides a nearly optimal bound (up to the $\log\log n$ factor) for the mixing time of the mean-field Chayes-Machta dynamics in the only regime of parameters where no non-trivial bound was previously known. Our proof consists of a multi-phased coupling argument that combines several key ingredients, including a new local limit theorem,
	a precise bound on the maximum of symmetric random walks with varying step sizes, and tailored estimates for critical random graphs.

	In addition, we derive an improved comparison inequality between the mixing time of the Chayes-Machta dynamics and that of the local Glauber dynamics on general graphs; this results in better mixing time bounds for the local dynamics in the mean-field setting.
\end{abstract}

\thispagestyle{empty}

\newpage

\setcounter{page}{1}

\section{Introduction}

The \emph{random-cluster model} generalizes classical random graph and spin system models, providing a unifying framework for their study~\cite{FK}.
It plays an indispensable role in the design of efficient Markov Chain Monte Carlo (MCMC) sampling algorithms for the ferromagnetic Ising/Potts model~\cite{Ullrich2,BSz2,GuoJ} and has become a fundamental tool in the study of phase transitions~\cite{BDC,DCST,DGHMT}.


The random-cluster model is defined on a finite graph~$G=(V,E)$ with an edge probability parameter
$p\in(0,1)$ and a cluster weight $q>0$.
The set of \textit{configurations} of the model is the set of all subsets of edges $A \subseteq E$. The probability of each configuration $A$ is given by the Gibbs distribution:
\begin{equation}\label{eq:rcmeasure}
\mu_{G,p,q}(A) = \frac{1}{Z} \cdot p^{|A|}(1-p)^{|E|-|A|} q^{c(A)}; 
\end{equation}
where $c(A)$ is the number of connected components in~$(V,A)$ and $Z:=Z(G,p,q)$ is the normalizing factor called the \textit{partition function.}

The special case when $q=1$ corresponds to the independent bond percolation model, where each edge of the graph $G$ appears independently with probability $p$.
Independent bond percolation is also known as the Erd\H{o}s-R\'enyi random graph model when $G$ is the complete graph. 

For integer $q \ge 2$, the random-cluster model is closely related to 
the ferromagnetic $q$-state Potts model. Configurations in the $q$-state Potts model are the assignments of spin values 
$\{1,\dots,q\}$ to the vertices of $G$; the $q=2$ case corresponds to the Ising model. 
A sample $A \subseteq E$ from the random-cluster distribution can be easily transformed into one for the Ising/Potts model by independently assigning a random spin from $\{1,\dots,q\}$ to each connected component of $(V,A)$. 
Random-cluster based sampling algorithms, which include the widely-studied Swendsen-Wang dynamics~\cite{SW}, are an attractive alternative to Ising/Potts Markov chains 
since they are often efficient at ``low-temperatures'' (large $p$).
In this parameter regime, several standard Ising/Potts Markov chains are known to converge slowly.


In this paper we investigate the \textit{Chayes-Machta (CM) dynamics}~\cite{CM}, a natural
Markov chain on random-cluster configurations 
that converges to the random-cluster measure.
The CM dynamics is a generalization to non-integer values of~$q$ of the widely studied Swendsen-Wang dynamics~\cite{SW}.
As with all applications of the MCMC method, the primary object of study is the \textit{mixing time}, i.e., 
the number of steps until the dynamics is close to
its stationary distribution, starting from the worst possible initial configuration.
We are interested in understanding how the mixing time of the CM dynamics
grows as the size of the graph $G$ increases, and in particular how it relates to the phase
transition of the model.

Given a random-cluster configuration $(V,A)$, one step of the CM dynamics is defined as follows:
\begin{enumerate}[(i)]
	\setlength{\itemsep}{0pt}
	\item \label{CMdynamics1} activate each connected component of $(V,A)$ independently with probability $1/q$; 
	\item \label{CMdynamics2} remove all edges connecting active vertices; 
	\item \label{CMdynamics3} add each edge between active vertices independently with probability $p$, leaving the rest of the configuration unchanged.	
\end{enumerate} 
We call \eqref{CMdynamics1} the {\it activation\/} sub-step, and~\eqref{CMdynamics2} and \eqref{CMdynamics3} combined the {\it percolation\/} sub-step.
It is easy to check that this dynamics is reversible with respect to the Gibbs distribution~($\ref{eq:rcmeasure}$)
and thus converges to it~\cite{CM}. 
For integer $q$, the CM dynamics may be viewed as a variant of the Swendsen-Wang dynamics. 
In the Swendsen-Wang dynamics, each connected component of $(V,A)$ receives a random color from $\{1,\dots,q\}$, and the edges are updated within each color class as in \eqref{CMdynamics2} and \eqref{CMdynamics3} above; in contrast,
the CM dynamics updates the edges of exactly \emph{one} color class.
However, note that the Swendsen-Wang dynamics is only well-defined for integer $q$, while 
the CM dynamics is feasible for any real $q > 1$. Indeed, the CM dynamics
was introduced precisely to allow this generalization.

The study of the interplay between phase transitions and the mixing time of Markov chains goes back to pioneering work in mathematical physics in the late 1980s. This connection for the specific case of the CM dynamics on the complete $n$-vertex graph, known as the \emph{mean-field model},
has received some attention in recent years (see~\cite{BSmf,GSVmf,GLP}) and is the focus of this paper. 
As we shall see, the mean-field case is already quite non-trivial, and has historically proven to be a useful starting point in understanding various types of dynamics on more general graphs.
We note that, so far, the mean-field is the only setting in which there are tight mixing time bounds for the CM dynamics; 
all other known bounds 
are deduced indirectly via comparison with other Markov chains,
thus incurring significant overhead~\cite{BSz2,BGV,GLrc,BG,Ullrich2,BSmf}.

The phase transition for the mean-field random-cluster
model is fairly well-understood~\cite{BGJ,LL}. In this setting, it is natural to re-parameterize by setting $p=\la/n$; the phase transition then
occurs at the \emph{critical} value $\la = \lc(q)$, where $\lc(q)=q$ when $q \in (0,2]$
and $\lc(q)=2(\frac{q-1}{q-2})\log(q-1)$ for $q>2$.
For $\la <\lc(q)$ all components are of size $O(\log n)$ with high probability (w.h.p.); that is, with probability tending to $1$ as $n \rightarrow \infty$. 
This regime is known as the \emph{disordered phase}.
On the other hand, for
$\la>\lc(q)$ there is a unique giant component of size $\approx \theta n$, 
where $\theta = \theta(\la,q)$;
this regime of parameters is known as the \emph{ordered phase}.
The phase transition is thus analogous to that in $G(n,p)$ corresponding to the emergence of a giant component.

The phase structure of the mean-field random-cluster model, however, is more subtle and depends crucially on the second parameter $q$. 
In particular, when $q>2$ the model exhibits \emph{phase coexistence} at the critical threshold $\la = \lc(q)$. Roughly speaking, this means that when $\la = \lc(q)$, the set of configurations with all connected components of size $O(\log n)$,
and set of configurations with a unique giant component, contribute each a constant fraction of the probability mass. For $q \le 2$, on the other hand, there is no phase coexistence.
These subtleties are illustrated in Figure~\ref{fig:pt}.

\begin{figure}
\begin{subfigure}[b]{.6\linewidth}
	\begin{tikzpicture}
		\draw[->] (-3.3, 0) -- (3.3, 0)
			node[anchor=north] {$\la$};
	
		 \foreach \x in {-3.3,...,3.3} {
			 \draw (\x, 0) -- (\x, -0.8pt);
		 }
	
		\node[x tick label] at (-1.2, 0) {$\ls$};
		\draw[major tick] (-1.2, 0) -- (-1.2, -1.25pt);
		\node[x tick label] at (0, 0) {$\lc$};
		\draw[major tick] (0, 0) -- (0, -1.25pt);
	
		\node[x tick label] at (1.2, 0) {$\lS$};
		\draw[major tick] (1.2, 0) -- (1.2, -1.25pt);
	
		\draw[thick] (1.2, 0) -- (1.2, 1);
		\draw[thick,->] (1.2, 1) -- (3.3, 1) 
			node[anchor=north,text width=2cm,align=center] {Ordered phase};
	
		\draw[thick] (-1.2, 0) -- (-1.2, 1);
		\draw[thick,->] (-1.2, 1) -- (-3.3, 1)
			node[anchor=north,text width=2cm,align=center] {Disordered phase};

		\draw[thick,->] (0,0) -- (0,1.6)
			node[anchor=south] {Phase coexistence};
	
		\draw [decorate,decoration={brace,amplitude=5pt,mirror,raise=4ex}]
			  (-1.2,0) -- (1.2,0) node[midway,yshift=-3em]{Metastability window};	
	\end{tikzpicture}
	\caption{$q > 2$}
\end{subfigure}
\begin{subfigure}[b]{.38\linewidth}
	\begin{tikzpicture}

		\draw[->] (-2, 0) -- (2, 0)
			node[anchor=north] {$\la$};
	
		 \foreach \x in {-2,...,2} {
			 \draw (\x, 0) -- (\x, -0.8pt);
		 }

		 \node[x tick label] at (0, 0) {$\lc$};
		 \draw[major tick] (0, 0) -- (0, -1.25pt);

		\draw[thick] (0, 0) -- (0, 1);
		\draw[thick,->] (0, 1) -- (2, 1) 
			node[anchor=north,text width=2cm,align=center] {Ordered phase};
	
		\draw[thick] (0, 0) -- (0, 1);
		\draw[thick,->] (0, 1) -- (-2, 1)
			node[anchor=north,text width=2cm,align=center] {Disordered phase};
		
	\end{tikzpicture}
	\caption{$1 < q \le 2$}
\end{subfigure}
\caption{(a)): phase structure when $q>2$. (b): phase structure when $q\in(1,2]$.} 
\label{fig:pt}  
\end{figure}
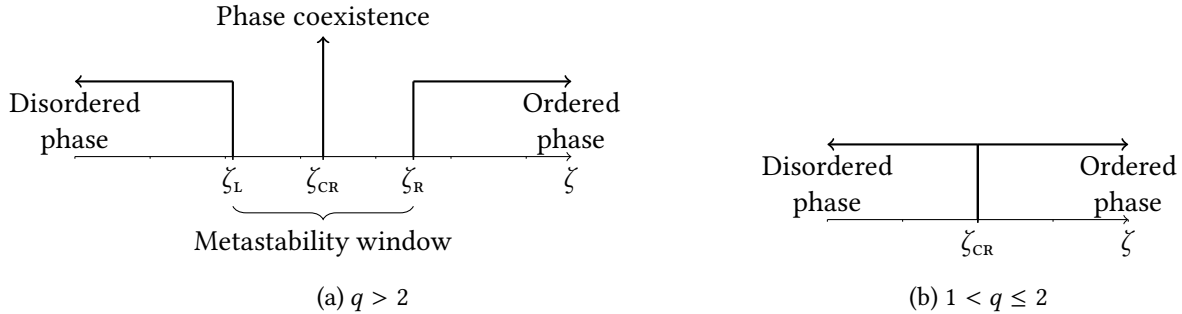

Phase coexistence at $\la =\lc(q)$ when $q > 2$ has significant implications for the speed of convergence of 
Markov chains, including the CM dynamics.
The following detailed connection between the phase structure of the model and the mixing time $\taumixCM$ of the CM dynamics was recently established in~\cite{BSmf,ABthesis,GLP}. 
When $q > 2$, we have:
\vspace{-2mm}
\begin{equation}
	\label{eq:mt-large-q}
	\taumixCM = 
	\begin{cases}
	\Theta(\log n) 			& \textrm{if}~\la \not\in [\ls,\lS); \\
	
	\Theta(n^{1/3})    			& \textrm{if}~\la = \ls; \\
	
	e^{\Omega({n})}    & \textrm{if}~\la \in (\ls,\lS),
	\end{cases}	
\end{equation}
where $(\ls,\lS)$ is the so-called \emph{metastability window}.
It is known that $\lS = q$, but $\ls$ does not have a closed form; see~\cite{BSmf,LL}; we note that $\lc(q) \in (\ls,\lS)$ for $q > 2$.

When $q \in (1,2]$, there is no 
metastability window,
and the mixing time of the mean-field CM dynamics is $\Theta(\log n)$ for all $\la \neq \lc(q)$.
In view of these results, the only case remaining open is when $q \in (1,2]$ and $\la = \lc(q)$.
Our main result shown below concerns precisely this regime, which is particularly delicate and had resisted analysis until now  for reasons we explain in our proof overview. 
\begin{theorem}
	\label{thm:critical:q<2:main}
	The mixing time of the CM dynamics 
	on the complete $n$-vertex graph
	when $\la=\lc(q) = q$ and $q \in (1,2)$ is $O(\log n \cdot \log \log n)$.
\end{theorem}

A $\Omega(\log n)$ lower bound is known for the mixing time of the mean-field CM dynamics that holds for all $p \in (0,1)$ and $q > 1$~\cite{BSmf}. Therefore, our result is tight up to the lower order $O(\log \log n)$ factor, and in fact even better as we explain in Remark~\ref{remark:betterbound}.
The conjectured tight bound when $\la=\lc(q)$ and $q \in (1,2)$ is $\Theta(\log n)$.
We mention that the $\la=\lc(q)$ and $q=2$ case, which is quite different and not covered by Theorem~\ref{thm:critical:q<2:main}, was considered earlier in~\cite{LNNP} for the closely related 
Swendsen-Wang dynamics, and a tight $\Theta(n^{1/4})$ bound was established for its mixing time. 
The same mixing time bound is expected for the CM dynamics in this regime. 

Our result establishes a striking behavior for random-cluster dynamics when $q \in (1,2)$. 
Namely, there is no slowdown (exponential or power law) in this regime at the critical threshold $\la=\lc(q)$. 
Note that for $q > 2$, as described in \eqref{eq:mt-large-q} above, the mixing time of the dynamics undergoes
an exponential slowdown, transitioning from $\Theta(\log n)$ when $\la < \ls$, 
to a power law at $\la = \ls$, and to exponential in $n$ when $\la \in (\ls,\lS)$. 
The absence of a critical slowdown for $q \in (1,2)$ was in fact predicted by the statistical physics community~\cite{Gar},
and our result provides the first rigorous proof of this phenomenon.
See Remark~\ref{remark:q=2} for further comments.

Our second result concerns the local \emph{Glauber dynamics} for the random-cluster model. In each step, the Glauber dynamics updates 
a single edge of the current configuration chosen uniformly at random; a precise definition of this Markov chain is given in Section~\ref{sec:GD}.
In~\cite{BSmf}, it was established that any upper bound on the mixing time $\taumixCM$ of the CM dynamics can be translated to one for the mixing time  $\taumixGD$ of the Glauber dynamics, 
at the expense of a $\tilde{O}(n^4)$ factor; the $\tilde{O}$ notation hides polylogarithmic factors.
In particular, it was proved in~\cite{BSmf} that
$
\taumixGD \le \taumixCM \cdot \tilde{O}(n^4).
$
We provide here an improvement of this comparison inequality.
\begin{theorem}
	\label{thm:intro:comparison}
	For all $q > 1$ and all $\la = O(1)$,
	$
	\taumixGD \le \taumixCM \cdot {O}(n^3 (\log n)^2).
	$
\end{theorem}
To prove this theorem, we establish a general comparison inequality that holds for any graph, any $q \ge 1$ and any $p \in (0,1)$;
see Theorem~\ref{thm:GDCM-comparison} for a precise statement.
When combined with the known mixing time bounds for the CM dynamics on the complete graph,
Theorem~\ref{thm:intro:comparison} yields that the random-cluster Glauber dynamics mixes in $\tilde{O}(n^3)$ steps when $q > 2$ and $\la \not\in(\ls,\lS)$, or when $q \in (1,2)$ and $\la  = O(1)$.
In these regimes,
the mixing time of the Glauber dynamics was previously known to be $\tilde{O}(n^4)$ and is conjectured to be $\tilde{O}(n^2)$; 
the improved comparison inequality in Theorem~\ref{thm:intro:comparison} gets us closer to this conjectured tight bound.
We note, however, that even if one showed the conjectured optimal bound for the mixing time of the Glauber dynamics,
the CM is faster, even if we take into account the computational cost associated to implementing its steps.

We conclude this introduction with some brief remarks about our analysis techniques,
which combine several key ingredients in a non-trivial way.
Our bound on the mixing time uses the well-known technique of {\it coupling\/}:  in order to
show that the mixing time is $O(\log n \cdot \log \log n)$, it suffices to couple
the evolutions of two copies of the dynamics, starting from two arbitrary configurations, in such a
way that they arrive at the {\it same\/} configuration after $O(\log n )$ steps with probability $\Omega(1/ \log \log n)$.
(The moves of the two copies can be correlated any way we choose, provided that
each copy, viewed in isolation, is a valid realization of the dynamics.)
Because of the delicate nature of the phase transition in the random-cluster model, combined with the fact that the percolation sub-step of the CM dynamics is critical when $\la = q$,
our coupling is somewhat elaborate and proceeds in multiple phases.
The first phase consists of a \emph{burn-in period},
where the two copies of the chain are run independently and the evolution of their
largest components is observed until they have shrunk to their ``typical'' sizes.
This part of the analysis is inspired by similar arguments in earlier work~\cite{BSmf,LNNP,GSVmf}.

In the second phase, we design a coupling of the activation of the connected components of the 
two copies which uses: 
(i) a local limit theorem, which can be thought of as a stronger version of a central limit theorem; 
(ii) a precise understanding of the distribution of the maximum of symmetric random walks on $\Z$ with varying step sizes; and
(iii) precise estimates for the component structure of random graphs.
We develop tailored versions of these probabilistic tools for our setting and 
combine them to guarantee that the same number of vertices from each copy are activated in each step w.h.p.\ for sufficiently many steps.
This phase of the coupling is the main novelty in our analysis, and allows us to quickly converge to the same configuration.
We give a more detailed overview of our proof in the following section.

\section{Proof sketch and techniques}
\label{section:overview}

We now give a detailed sketch of the multi-phased coupling argument
for proving Theorem \ref{thm:critical:q<2:main}.
We start by formally defining the notions of mixing and coupling times.
Let $\Omega_{\textsc{\tiny RC}}$ be the set of random-cluster configurations of a graph $G$; let
$\GD$ be the transition matrix of a random-cluster Markov chain with stationary distribution $\mu = \mu_{G,p,q}$, and 
let $\GD^t(X_0,\cdot)$ be the distribution of the chain after $t$ steps starting from $X_0 \in \Omega_{\textsc{\tiny RC}}$.
The \emph{$\varepsilon$-mixing time} of $\GD$ is given by
\begin{equation}\label{prelim:eq:mixing}
\taumix^{\GD}(\varepsilon) := \max\limits_{X_0 \in \Omega_{\textsc{\tiny RC}}}\min\limits_{t \ge 0}\left\{ ||\GD^t(X_0,\cdot)-\mu(\cdot)||_{\textsc{\tiny TV}} \le \varepsilon \right\}\nonumber,
\end{equation}
where $||\cdot||_{\textsc{\tiny TV}}$ denotes total variation distance. In particular, the {\it mixing time} of $\GD$ is $\taumix^{\GD} := \taumix^{\GD}(1/4)$.

A {\it (one step) coupling} of the Markov chain $\GD$ specifies, for every pair of states $(X_t, Y_t) \in \Omega_{\textsc{\tiny RC}} \times \Omega_{\textsc{\tiny RC}}$, 
a probability distribution over $(X_{t+1}, Y_{t+1})$ such that the processes $\{X_t\}$ and $\{Y_t\}$ are valid realizations of $\GD$, and if $X_t=Y_t$ then $X_{t+1}=Y_{t+1}$. The {\it coupling time}, denoted $\Tcoup$, is the minimum $T$ such that $\Pr[X_T \neq Y_T] \le 1/4$, starting from the worst possible pair of configurations in $\Omega_{\textsc{\tiny RC}}$. 
It is a standard fact that $\taumix^{\GD} \le \Tcoup$; moreover, when $\Pr[X_T = Y_T] \ge \delta$ for some coupling, then
$\taumix^{\GD} = O(T \delta^{-1})$ (see, e.g.,~\cite{levin2017markov}).

We provide first a high level description of our coupling for the CM dynamics.
For this, we require the following notation. 
For a random cluster configuration $X$, let
$L_i(X)$ denote the size of the $i$-th largest connected component in~$(V,X)$,
and let $\mathcal{R}_i(X):=\sum_{j \ge i} L_j(X)^2$; 
in particular, $\mathcal{R}_1(X)$ is the sum of the squares of the sizes of all the components of~$(V,X)$.
Our coupling has three main phases:
\begin{enumerate}
	\setlength{\itemsep}{0pt}
	\item \emph{Burn-in period:} run two copies $\{X_t\}$, $\{Y_t\}$ independently, starting from a pair of arbitrary initial configurations, until $\mathcal{R}_1(X_T)  = O(n^{4/3})$ and $\mathcal{R}_1(Y_T)  = O(n^{4/3})$.
	\item \emph{Coupling to the same component structure: }starting from $X_T$ and $Y_T$
	such that $\mathcal{R}_1(X_T)  = O(n^{4/3})$ and $\mathcal{R}_1(Y_T)  = O(n^{4/3})$, we design a two-phased coupling that reaches two configurations with the same component structure 
	as follows:			
	\begin{enumerate}
		\setlength{\itemsep}{0pt}
		\item[2a.] A two-step coupling after which the two configurations agree 
		on all ``large components'';
		\item[2b.] A coupling that after $O(\log n)$ additional steps reaches two configurations that will also have 
		the same ``small component'' structure.
	\end{enumerate}
	\item \emph{Coupling to the same configuration:} 
	starting from two configurations with the same component structure, 
	there is a straightforward coupling that couples the two configurations in $O(\log n)$ steps w.h.p.
\end{enumerate}	
We proceed to describe each of these phases in detail.

\subsection{The burn-in period}

During the initial phase,
two copies of the dynamics evolve independently. 
This is called a \emph{burn-in period} and in our case consists of three sub-phases.  

In the first sub-phase of the burn-in period the goal is to reach a configuration $X$ such that $\mathcal{R}_2(X) = O(n^{4/3})$.
For this, we use a lemma from \cite{ABthesis}, which shows that after $T = O(\log n)$ steps of the CM dynamics $\mathcal{R}_2(X_T) = O(n^{4/3})$
with at least constant probability; this holds when $\la=q$ for any initial configuration $X_0$ and any~$q > 1$.

\begin{lemma} [\cite{ABthesis}, Lemma 3.42]
	\label{lemma:critical:burn:initial-cond}
	Let $q>1$ and $\la=q$, and let $X_0$ be an arbitrary random-cluster configuration. Then, for any constant $C \geq 0$,
	after $T=O(\log n)$ steps
	$\mathcal{R}_2(X_T) = O(n^{4/3})$ and $L_1(X_T) > Cn^{2/3}$
	with probability $\Omega(1)$.
\end{lemma} 

In the second and third sub-phases of the burn-in period, we use
the fact that when $\mathcal{R}_2(X_t) = O(n^{4/3})$, the number of activated vertices is well concentrated around $n/q$ (its expectation).
This is used to show that the size of the largest component
contracts at a constant rate for $T=O(\log n)$ steps until
a configuration $X_T$ is reached
such that $\mathcal{R}_1(X_T)  = O(n^{4/3})$. 
This part of the analysis is split into two sub-phases because the contraction for $L_1(X_t)$ requires a more delicate analysis when $L_1(X_t) = o(n)$; this is captured in the following two lemmas. 

\begin{lemma}
	\label{Lemma1}
	Let $\la=q$ and $q \in (1,2)$. 
	Suppose $\mathcal{R}_2(X_0) = O(n^{4/3})$.
	Then,
	for any constant $\delta > 0$,
	there exists $T = T(\delta) =  O(1)$ such that $\mathcal{R}_2(X_T)  = O(n^{4/3})$ and ${L_1}(X_T)  \leq \delta n$ with probability $\Omega(1)$.
\end{lemma}

\begin{lemma}
	\label{lemma:burn-in-subphase3}
	Let $\la=q$ and $q \in (1,2)$.
	Suppose $\mathcal{R}_2(X_{0}) = O(n^{4/3})$ 
	and that $L_1(X_{0}) \leq \delta n$ 
	for a sufficiently small constant $\delta$. 	
	Then, with probability $\Omega(1)$, after $T=O(\log n)$ steps $\mathcal{R}_1(X_T) = O(n^{4/3})$.
\end{lemma}

Lemmas~\ref{Lemma1} and~\ref{lemma:burn-in-subphase3} are proved in Section~\ref{Appendix Burn-in}.
Combining them with Lemma~\ref{lemma:critical:burn:initial-cond} immediately yields the following theorem.

\begin{theorem}
	\label{theorem:critical:q<2:first-step}
	Let $\la=q$, $q \in (1,2)$ and let $X_0$ be an arbitrary random-cluster configuration of the complete $n$-vertex graph. 
	Then, with probability $\Omega(1)$, after $T=O(\log n)$ steps $\mathcal{R}_1(X_T) = O(n^{4/3})$.
\end{theorem}

\begin{remark}
	\label{remark:q=2}
The contraction of $L_1(X_t)$ established by Lemmas~\ref{Lemma1} and~\ref{lemma:burn-in-subphase3} only occurs
when $q \in (1,2)$; when $q > 2$ the quantity $L_1(X_t)$ may increase in expectation, 
whereas for $q=2$ we have $\E[L_1(X_{t+1}) \mid X_t] \approx L_1(X_t)$, and the contraction of the size of the largest component is due instead 
to fluctuations caused by a large second moment. 
(This is what causes the power law slowdown when $\la=q=2$.) 
\end{remark}

\begin{remark} Sub-steps \eqref{CMdynamics2} and \eqref{CMdynamics3} of the CM dynamics
are equivalent to replacing the active portion of the configuration by a $G(m,q/n)$ random graph, where $m$ is the number of active vertices.
Since $\E[m] = n/q$, one key challenge in the proofs of Lemmas~\ref{Lemma1} and~\ref{lemma:burn-in-subphase3}, and in fact in the entirety of our analysis,
is that the random graph $G(m,q/n)$ is critical or almost critical w.h.p.\ since $m \cdot q/n \approx 1$; consequently 
its structural properties are not well concentrated 
and cannot be maintained for the required $O(\log n)$ steps of the coupling.
This is one of the key reasons why the $\la = \lc(q) =q$
regime is quite delicate.
\end{remark}

\subsection{Coupling to the same component structure}
\label{subsec:couple}

For the second phase of the coupling, we assume that we start from a pair of configurations $X_0$, $Y_0$
such that $\mathcal{R}_1(X_0) = O(n^{4/3})$, $\mathcal{R}_1(Y_0) = O(n^{4/3})$.
The goal is to show that after $T = O(\log n)$ steps, with probability $\Omega(1/\log \log n)$, we reach two configurations $X_T$ and $Y_T$ with the same component structure; i.e., $L_j(X_T) = L_j(Y_T)$ for all $j \ge 1$. 
In particular, we prove the following.

\begin{theorem}
	\label{theorem:critical:q<2:second-step}
	Let $\la=q$, $q \in (1,2)$ and suppose $X_0, Y_0$ are random-cluster configurations such that
	$\mathcal{R}_1(X_0) = O(n^{4/3})$ and $\mathcal{R}_1(Y_0) = O(n^{4/3})$. Then, 
	there exists a coupling of the CM steps such that
	after $T=O(\log n)$ steps $X_T$ and $Y_T$ have the same component structure with probability $\Omega\left( (\log \log n)^{-1} \right)$.
\end{theorem}

Our coupling construction for proving Theorem \ref{theorem:critical:q<2:second-step} has two main sub-phases. The first is a two-step coupling after which the two configurations agree on all the components of size above a certain threshold $B_\omega = {n^{2/3}}/{\on}$, where $\on$ is a slowly increasing function. For convenience and definiteness we set $\on = \log \log \log \log n$.
In the second sub-phase we take care of matching the small component structures. 

We note that when the same number of vertices are activated from each copy of the chain, we can easily couple the percolation sub-step 
(with an arbitrary bijection between the activated vertices) and replace the configuration 
on the active vertices in both chains with the same random sub-graph; consequently, the component structure in the updated sub-graph would be identical. Our goal is thus to design a coupling of the activation of the components that activates the same number of vertices in both copies in every step.

In order for the initial two-step coupling to succeed,
certain (additional) properties of the configurations are required.
These properties are achieved with a continuation of the initial burn-in phase for a small number of $O(\log \on)$ steps.
For a random-cluster configuration $X$, 
let $\widetilde{\mathcal{R}}_\omega(X) = \sum_{j: L_j(X) \le B_\omega} L_j(X)^2$
and let $I(X)$ denote the number of isolated vertices of $X$.
Our extension of the burn-in period is captured by the following lemma.

\begin{lemma}
	\label{lemma:critical:q<2:second-step-first}
	Let $\la=q$, $q \in (1,2)$ and suppose $X_0$ is such that $\mathcal{R}_1(X_0) = O(n^{4/3})$. Then, there exists 
	$T=O(\log \on)$ and
	a constant $\beta > 0$
	such that 
	$ \widetilde{\mathcal{R}}_\omega(X_T) = O(n^{4/3}{\on}^{-1/2})$,
	$\mathcal{R}_1(X_T) = O(n^{4/3})$
	and $I(X_T) = \Omega(n)$ 
	with probability $\Omega(\on^{-\beta})$.
\end{lemma}

The proof of Lemma~\ref{lemma:critical:q<2:second-step-first} is provided in Section~\ref{sec:second-step-first}.

With these bounds on $\widetilde{\mathcal{R}}_\omega(X_T)$, $\widetilde{\mathcal{R}}_\omega(Y_T)$, $I(X_T)$
and $I(Y_T)$, we construct the two-step coupling for matching the \emph{large} component structure. 
The construction crucially relies on a new local limit theorem (Theorem~\ref{thm:prelim:llt-cor}).
In particular, under our assumptions, when $\on$ is small enough, there are few components with sizes above $B_\omega$. Hence,
we can condition on the event that all of them are activated simultaneously. 
The difference in the number of active vertices generated by the activation of these large components can then be ``corrected'' by a coupling of the activation of the smaller components; for this we use our new {local limit theorem}. 

Specifically, our local limit theorem applies to the random variables corresponding to the number of activated vertices from the small components of each copy. We prove it using a result of Mukhin~\cite{Muk} and the fact that, among the small components, there are (roughly speaking) many components of many different sizes. 
To establish the latter we require a refinement of known random graph estimates (see Lemma~\ref{lemma:prelim:cmpt-count}).

To formally state our result we introduce some additional notation.
Let
$\mathcal{S}_{\omega}(X)$ be the set of connected components of $X$ with sizes greater than $B_\omega$.
At step $t$,
the activation of the components of two random-cluster configurations $X_t$ and $Y_t$ is done using a maximal matching $W_t$ between the components of $X_t$ and $Y_t$,
with the restriction that only components of equal size are matched to each other.
For an increasing positive function $g$ and each integer $k \ge 0$, define $\hat{N}_k(t, g) := \hat{N}_k(X_t,Y_t, g)$ as the number of matched pairs in $W_t$ whose component sizes are in the interval 
\[\mathcal{I}_{k}(g) = \Big[\frac{\vartheta n^{2/3}}{2g(n)^{2^k}},\frac{\vartheta  n^{2/3}}{g(n)^{2^k}}\Big],\]
where $\vartheta>0$ is a fixed large constant (independent of $n$). 

\begin{lemma}
	\label{lemma:critical:q<2:second-step-second}
	Let $\la=q$, $q \in (1,2)$ and suppose $X_0, Y_0$ are random-cluster configurations such that 
	$\mathcal{R}_1(X_0) = O(n^{4/3})$, $\widetilde{\mathcal{R}}_\omega(X_0) = O(n^{4/3}{\on}^{-1/2})$, $I(X_0)=\Omega(n)$ and similarly for $Y_0$. 
	Then, 
	there exists a two-step coupling of the CM dynamics such that
	$\mathcal{S}_{\omega}(X_2) = \mathcal{S}_{\omega}(Y_2)$ with probability 
	$\exp\left(-O(\on^9)\right)$.
	
	Moreover, $L_1(X_2) = O(n^{2/3} \on)$,	 
	$\mathcal{R}_2(X_2) = O(n^{4/3})$,
	$\widetilde{\mathcal{R}}_\omega(X_2) = O(n^{4/3}{\on}^{-1/2})$, 
	$I(X_2)=\Omega(n)$, 
	$\hat{N}_k(2,\on) = \Omega({\on}^{3 \cdot 2^{k-1}})$ 
	for all $k \ge 1$ such that $n^{2/3}\on^{-2^{k-1}}\rightarrow \infty$,
	and similarly for $Y_2$.
\end{lemma}

From the first part of the lemma we obtain two configurations that agree on all of their large components, as desired, 
while the second part guarantees additional structural properties for the resulting configurations 
so that the next sub-phase of the coupling can also succeed with the required probability.
The proof of Lemma~\ref{lemma:critical:q<2:second-step-second} is given in Section~\ref{sec:coup:llt}.

In the second sub-phase,
after the large component are matched, 
we can design a coupling 
that activates exactly the same number of vertices from each copy of the chain.
To analyze this coupling we use a precise estimate on the distribution of the maximum of symmetric random walks over integers (with steps of different sizes).
We are first required to run the chains coupled for $T=O(\log \on)$ steps,
so that certain additional structural properties appear.
Let $M (X_t)$ and $M(Y_t)$ be the components in the matching $W_t$ that belong to $X_t$ and $Y_t$, respectively, and let $D(X_t)$ and $D(Y_t)$ be the complements of $M (X_t )$ and $M (Y_t)$.
Let 
$$
Z_t = \sum\nolimits_{\mathcal{C} \in D(X_t) \cup D(Y_t)} |\mathcal{C}|^2.
$$

\begin{lemma}
	\label{lemma:critical:q<2:second-step-2.5}
	Let $\la=q$, $q \in (1,2)$. 
	Suppose $X_0$ and $Y_0$ are random-cluster configurations such that 
	$\mathcal{S}_{\omega}(X_0) = \mathcal{S}_{\omega}(Y_0)$, 
	and $\hat{N}_k(0, \on) = \Omega({\on}^{3 \cdot 2^{k-1}})$ 
	for all $k \ge 1$ such that $n^{2/3}\on^{-2^{k-1}}\rightarrow \infty$.
	Suppose also that 
	$L_1(X_0) = O(n^{2/3} \on)$,	 
	$\mathcal{R}_2(X_0) = O(n^{4/3})$,
	$\widetilde{\mathcal{R}}_\omega(X_0) = O(n^{4/3}{\on}^{-1/2})$, 
	$I(X_0)=\Omega(n)$,
	and similarly for $Y_0$.
	
	Then, there exists a coupling of the CM steps
	such that with probability $\exp(-O ( (\log \on)^2))$ after $T=O(\log \on)$ steps:
	$\mathcal{S}_{\omega}(X_T) = \mathcal{S}_{\omega}(Y_T)$,
	$Z_T = O(n^{4/3}{\on}^{-1/2})$,
	$\hat{N}_k(T, \on^{1/2})= \Omega({\on}^{3 \cdot 2^{k-2}})$
	for all $k \ge 1$ such that $n^{2/3}\on^{-2^{k-1}}\rightarrow \infty$,
	$\mathcal{R}_1(X_T) = O(n^{4/3})$, 
	$I(X_T) = \Omega(n)$,
	and similarly for $Y_T$. 
\end{lemma}

\noindent The proof of Lemma~\ref{lemma:critical:q<2:second-step-2.5} also uses 
our local limit theorem (Theorem~\ref{thm:prelim:llt-cor}) and
is provided in Section~\ref{subsec:general-llt}.

The final step of our construction is a coupling of the activation of the components
of size less than $B_\omega$, so that exactly the same number of vertices
are activated from each copy in each step w.h.p.\

\begin{lemma}
	\label{lemma:critical:q<2:second-step-third}
	Let $\la=q$, $q \in (1,2)$ and suppose $X_0$ and $Y_0$ are random-cluster configurations such that
	$\mathcal{S}_{\omega}(X_0) = \mathcal{S}_{\omega}(Y_0)$,
	$Z_0 = O(n^{4/3}{\on}^{-1/2})$,
	and 	$\hat{N}_k\left(0, \on^{1/2}\right)= \Omega({\on}^{3 \cdot 2^{k-2}})$
	for all $k \ge 1$ such that $n^{2/3}\on^{-2^{k-1}}\rightarrow \infty$.
	Suppose also that 
	$\mathcal{R}_1(X_0) = O(n^{4/3})$, 
	$I(X_0) = \Omega(n)$
	and similarly for $Y_0$. 	
	Then, there exist a coupling of the CM steps and a constant $\beta > 0$ such that 
	after $T=O(\log n)$ steps, $X_T$ and $Y_T$ have the same component structure with probability 
	$\Omega \left( (\log \log \log n)^{-\beta} \right)$.
\end{lemma}	

We comment briefly on how we prove this lemma.
Our starting point is two configurations with the same ``large'' component structure;
i.e., $\mathcal{S}_{\omega}(X_0) = \mathcal{S}_{\omega}(Y_0)$.
We use the maximal matching $W_0$ to couple the activation
of the large components in $X_0$ and $Y_0$. The small components \emph{not matched} by
$W_0$, i.e., those counted in $Z_0$, are then activated independently.
This creates a discrepancy $\mathcal{D}_0$ between the number of active vertices from each copy.
Since $\E[\mathcal{D}_0] = 0$ and $\var(\mathcal{D}_0) = \Theta(Z_0) = \Theta({n^{4/3}}{\on^{-1/2}})$, it follows from Hoeffding's inequality that $\mathcal{D}_0 \le {n^{2/3}}{\on^{-1/4}}$ w.h.p.
To fix this discrepancy, we use the small components \emph{matched} by $W_0$.
Specifically, under the assumptions in Lemma~\ref{lemma:critical:q<2:second-step-third},
we can construct a coupling of the activation of the small components so that the difference in the number of activated vertices from the small components from each copy is exactly $\mathcal{D}_0$ with probability $\Omega(1)$. 
This part of the construction utilizes random walks over the integers; 
in particular, we use a lower bound for the maximum of such a random walk.

We need to repeat this process until $Z_t = 0$; this takes $O(\log n)$ steps since $Z_t \approx (1-1/q)^t Z_0$. 
However, there are a few complications. 
First, the initial assumptions on the component structure of the configurations are not preserved for this many steps w.h.p., 
so we need to relax the requirements as the process evolves.
This is in turn possible because the discrepancy $\mathcal{D}_t$ decreases with each step, which implies that 
the probability of success of the coupling increases at each step.
See Section~\ref{BigProof} for the detailed proof.

We now indicate how these lemmas lead to a proof of Theorem \ref{theorem:critical:q<2:second-step} stated earlier.

\begin{proof}[Proof of Theorem~\ref{theorem:critical:q<2:second-step}]
	Suppose $\mathcal{R}_1(X_{0}) = O(n^{4/3})$ and $\mathcal{R}_1(Y_{0}) = O(n^{4/3})$.
	It follows from Lemma~\ref{lemma:critical:q<2:second-step-first}, \ref{lemma:critical:q<2:second-step-second}, \ref{lemma:critical:q<2:second-step-2.5} and \ref{lemma:critical:q<2:second-step-third}
	that there exists a coupling of the CM steps such that
	after $T = O(\log n)$ steps, $X_{T}$ and $Y_{T}$ could have the same component structure.
	This coupling succeeds with probability at least
	$$\rho = \Omega(\on^{-\beta_1}) \cdot \exp\big(-O(\on^9)\big) \cdot \exp \big( -O \big( (\log \on)^2 \big) \big) \cdot \Omega \big( (\log \log \log n)^{-\beta_2} \big), $$
	where $\beta_1, \beta_2 > 0$ are  constants.
	Thus,
	$\rho = \Omega\big( (\log \log n)^{-1} \big)$, since $\on = \log \log \log \log n$.
\end{proof}

\begin{remark}
	We pause to mention that this delicate coupling for the activation of the components is not required
	when $\la=q$ and $q > 2$. In that regime, the random-cluster model is super-critical, 
	so after the first $O(\log n)$ steps, the component structure is much simpler, with exactly one large component. On the other hand, when $\la=q$ and $q \in (1,2]$ the model is critical, which, combined with the fact mentioned earlier that the percolation sub-step of the dynamics is also critical when $\la=q$, 
	makes the analysis of the CM dynamics in this regime quite subtle.
\end{remark}

\subsection{Coupling to the same configuration}
In the last phase of the coupling, suppose we start with two configurations $X_0$, $Y_0$ with the same component structure. We are still required to bound the number of steps until the same configuration is reached.
The following lemma from~\cite{BSmf} supplies the desired bound.

\begin{lemma} [\cite{BSmf}, Lemma 24]
	\label{lemma:316}
	Let $q>1$, $\la >0$ and let $X_0$, $Y_0$ be two random-cluster configurations 
	with the same component structure. Then, there exists a coupling of the CM steps such that 
	after $T=O(\log n)$ steps, $X_T = Y_T$ w.h.p.
\end{lemma}

Combining the results for each of the phases of the coupling, we now prove Theorem \ref{thm:critical:q<2:main}.

\begin{proof}[Proof of Theorem \ref{thm:critical:q<2:main}]
	By Theorem~\ref{theorem:critical:q<2:first-step},
	after $t_0 = O(\log n)$  steps, with probability $\Omega(1)$, 
	we have
	$\mathcal{R}_1(X_{t_0}) = O(n^{4/3})$ and $\mathcal{R}_1(Y_{t_0}) = O(n^{4/3})$. 
	If this is the case, Theorem~\ref{theorem:critical:q<2:second-step} and Lemma~\ref{lemma:316}
	imply that there exists	a coupling of the CM steps such that with probability 
	$\Omega\left( (\log \log n)^{-1} \right)$ 
	after an additional $t_1 = O(\log n)$ steps, $X_{t_0 + t_1} = Y_{t_0 + t_1}$.
	Consequently, we obtain that $\taumixCM = O(\log n \cdot \log \log n)$ as claimed.
\end{proof}

\begin{remark}
	\label{remark:betterbound}
	The probability of success in Theorem~\ref{theorem:critical:q<2:second-step}, which governs the 
	lower order term $O(\log \log n)$ in our mixing time bound, 
	is controlled by our choice of the function $\on$ for the definition of ``large components''. 
	By choosing $\on $ that goes to $ \infty$ more slowly, 
	we could improve our mixing time bound to $O(\log n \cdot g(n))$
	where $g(n)$ is any function that tends to infinity arbitrarily slowly. 
	However, it seems that new ideas are required to obtain a bound of $O(\log n)$ (matching the known lower bound). In particular, the fact that $\on \rightarrow \infty$ is crucially used in some of our proofs. Our specific choice of $\on$ 
	yields the $O(\log n \cdot \log \log n)$ bound and makes our analysis cleaner.
\end{remark}

\section{Random graph estimates}
\label{RG}
In this section, we compile a number of standard facts about the $G(n,p)$ random graph model
which will be useful in our proofs.
We use $G \sim G(n,p)$ to denote a random graph $G$ sampled from the standard $G(n,p)$ model, in which every edge appears independently with probability $p$.
A $G(n, p)$ random graph is said to be \emph{sub-critical} when $n p < 1$. 
It is called \emph{super-critical} when $n p > 1$ and \emph{critical} when $np=1$.
For a graph~$G$, with a slight abuse of notation, let
$L_i(G)$ denote the size of the $i$-th largest connected component in~$G$,
and let $\mathcal{R}_i(G):=\sum_{j \ge i} L_j(G)^2$; note that the same notation is used for the components of a random-cluster configuration, but it will always be clear from context which case is meant.
 

\begin{fact}
\label{fact:GraphMonotone}
Given $0 < N_1 < N_2$ and $p \in [0, 1]$. 
Let $G_1 \sim G(N_1, p)$ and $G_2 \sim G(N_2, p)$.
For any $K > 0$, $\Pr[L_1(G_1) > K] \leq \Pr[L_1(G_2) > K]$.
\end{fact}

\begin{proof}
	Consider the coupling of $(G_1, G_2)$ such that $G_1$ is a subgraph of $G_2$.
	$L_1(G_1) \le L_1(G_2)$ with probability 1.
	Proposition just follows from Strassen's theorem (see, e.g., Theorem 22.6 in \cite{levin2017markov}).
\end{proof}

\begin{lemma}[\cite{LNNP}, Lemma 5.7]
	\label{lemma:rg:isolated} 
	Let $I(G)$ denote the number of isolated vertices in $G$. If $np = O (1)$, then there exists a constant $C > 0$ such that 
	\[\Pr[I (G ) > Cn] = 1 - O(n^{-	1}).\]
\end{lemma}

Consider the equation
\begin{equation}\label{BetaDef}
e^{-d x} = 1 - x
\end{equation}
and let $\beta(d)$ be defined as its unique positive root. Observe that $\beta$ is well-defined for $d > 1$.

\begin{lemma} [\cite{ABthesis}, Lemma 2.7]
	\label{L27}
	Let $G\sim G(n + m, d_n /n)$ random graph where $|m| = o(n)$ and $\lim_{n\rightarrow \infty} d_n = d$. 
	Assume $1 < d_n = O(1)$ and $d_n$ is bounded away from $1$ for all $n\in \mathbb{N} $. Then,
	For $A = o(\log n)$ and sufficiently large $n$, there exists a constant $c > 0$ such that 
		$$\Pr\left[ \lvert L_1(G) - \beta(d) n \rvert > \lvert m \rvert + A \sqrt{n}\right] \le e^{-cA^2}.$$
\end{lemma}

\begin{lemma}[\cite{ABthesis}, Lemma 2.16]
	\label{lemma:rg:critical:exp-crude-bound}
	For $np>0$, we have	$\E\left[\mathcal{R}_2(G)\right] = O\left(n^{4/3}\right)$.
\end{lemma}






Consider the near-critical random graph $G\left(n, \frac{1 + \epsilon}{n}\right)$ with $\epsilon = \epsilon(n) = o(1)$.

\begin{lemma}[\cite{LNNP}, Theorem 5.9]
	\label{T58}
	Assume $\epsilon^3n \geq 1$, then for any $A$ satisfying $2 \leq A \leq \sqrt{\epsilon^3n}/10$,
	there exists some constant $c > 0$ such that 
	$$\Pr\left[\left|L_1(G) - 2\epsilon n\right| > A \sqrt{\frac{n}{\epsilon}}\right] = O\left(e^{-cA^2}\right).$$
\end{lemma}

\begin{corollary}
	\label{lemma:rg:critical:giant-concentration} 
	Let $G\sim G\left(n, \frac{1 + \epsilon}{n}\right)$ with $\epsilon = o(1)$.
	For any positive constant $\rho \le 1/10$, there exist constants $C \ge 1$ and $c > 0$
	such that if $\varepsilon^3 n \ge C$, then 	
	\[\Pr\left[\left\lvert L_1(G)-2\varepsilon n\right\rvert > \rho \varepsilon n\right] = O(e^{-c\varepsilon^3 n}).\]
\end{corollary}

\begin{lemma}[\cite{LNNP}, Theorem 5.12]
	\label{T512}
	Let $\epsilon < 0$, then $\E[\mathcal{R}_1(G)] = O\left(n/|\epsilon|\right).$
\end{lemma}

\begin{lemma}[\cite{LNNP}, Theorem 5.13]
	\label{T513}
	Let $\epsilon > 0$ and $\epsilon^3 n \geq 1$ for large $n$, then $\E[\mathcal{R}_2(G)] = O\left(n/\epsilon \right)$.
\end{lemma}

For the next results, suppose that $G \sim G(n,\frac{1+\lambda n^{-1/3}}{n})$, where $\lambda = \lambda(n)$ may depend on~$n$. 

	\begin{lemma}
		\label{lemma:prelim:critical-var}
		If $|\lambda| = O(1)$, then
		$\E\left[\mathcal{R}_1(G)\right] = O\left(n^{4/3}\right)$.
	\end{lemma}

	\begin{proof}
		Follows from Lemmas 2.13, 2.15 and 2.16 in \cite{ABthesis}.
	\end{proof}

	All the random graph facts stated so far can be either found in the literature, 
	or follow directly from well-known results. 
	The following lemmas are slightly more refined versions of similar results in the literature.



	\begin{lemma}
		\label{lemma:prelim:small-cmpt-var}
		Suppose $|\lambda| = O(h(n))$ and let $B_h = {n^{2/3}}{h(n)^{-1}}$, 
		where $h:\N \rightarrow \R$ is a positive increasing function such that $h(n)=o(\log n)$. 
		Then, for any $\alpha \in (0,1)$ there exists a constant $C = C(\alpha) > 0$ such that, with probability at least $\alpha$, 
		\[\sum\nolimits_{j:L_j(G) \le B_h} L_j(G)^2 \le  C{n^{4/3}}{h(n)^{-1/2}}. \]
	\end{lemma}

	\begin{lemma}
		\label{lemma:prelim:cmpt-count}
		Let $S_B = \{j: B \le L_j(G) \le 2B\}$ 
		and
		suppose there exists a positive increasing function $g$ such that $g(n) \rightarrow \infty$, $g(n) = o(n^{1/3})$,
		$|\lambda| \le g(n)$ and
		$B \le \frac{n^{2/3}}{g(n)^2}$.  
		If $B \rightarrow \infty$, then there exists constants $\delta_1,\delta_2 > 0$ independent of $n$ such that
		\[\Pr\left[|S_B| \le \frac{\delta_1n}{B^{3/2}}\right] \le \frac{\delta_2  B^{3/2} }{ n}.\]
	\end{lemma}
	
	The proofs of Lemmas~\ref{lemma:prelim:small-cmpt-var} and \ref{lemma:prelim:cmpt-count} are provided in Appendix~\ref{Appendix RG}.
	Finally, the following corollary of Lemma \ref{lemma:prelim:cmpt-count} will also be useful. 
	For a graph $H$, let ${N}_k(H, g)$ be the number of components of $H$ whose sizes are in the interval $\mathcal{I}_k(g)$. We note that with a slight abuse of notation, for a random-cluster configuration $X$, we also use ${N}_k(X, g)$ for the number of connected components of $X$ in $\mathcal{I}_k(g)$.

	\begin{corollary}
		\label{lemma:critical:q<2:maintain-trees}
		Let $m \in (n/2q,n]$ and let $g$ be an increasing positive function that such that $g(n)=o(m^{1/3})$, $g(n) \rightarrow \infty$ and $|\lambda| \le g(m)$.
		If $H \sim G\big(m, \frac{1 + \lambda m^{-1/3}}{m}\big)$,
		there exists a constant $b > 0$ such that, with probability at least $1-O\left(g(n)^{-3}\right)$, 
		$N_{k}(H,g) \ge b g(n)^{3 \cdot 2^{k-1}}$
		for all $k \ge 1$ such that $n^{2/3}g(n)^{-2^k}\rightarrow \infty$.
	\end{corollary}

\section{The burn-in period: proofs}
\label{Appendix Burn-in}

In this section we will provide proofs for Lemma~\ref{Lemma1} and Lemma~\ref{lemma:burn-in-subphase3}. 

\subsection{A drift function}
\label{subsec:drift}

Consider the mean-field random-cluster model with parameters $q \ge 1$ and $p = \la/n$.
In this subsection, we introduce a drift function 
captures the rate of decay of the size of the largest component in a configuration under steps of the CM dynamics which will be helpful for proving Lemma~\ref{Lemma1}; 
this function was first studied in~\cite{BSmf}.

Given $\theta \in (0,1]$, consider the equation
\begin{equation}\label{PhiDef}
e^{-\la x} = 1 - \frac{qx}{1+(q-1)\theta}
\end{equation}
and let $\phi(\theta, \la, q)$ be defined as the largest positive root of~\eqref{PhiDef}. 
We shall see that $\phi$ is not defined for all $q$ and $\la$ since there may not be a positive root. 
When $ \la$ and $q$ are clear from the context we use $\phi(\theta)=\phi(\theta, \la, q)$. 
Note that $\beta(\la)$ defined by equation~(\ref{BetaDef}) is the special case of~\eqref{PhiDef} when $q=1$;
observe that $\beta$ is only well-defined when $\la > 1$.

We let $k(\theta, q):=(1+(q-1)\theta)/q$ so that 
$\phi(\theta, \la, q)=\beta(\la \cdot k(\theta, q)) \cdot k(\theta, q).$ 
Hence, $\phi(\theta, \la, q)$ is only defined when $\la \cdot k(\theta, q)>1$; 
that is, $\theta \in (\theta_{min}, 1]$, where $\theta_{min}=\frac{q-\la}{\la(q-1)}$.
Note that when $\la = q$, $\phi(\theta)$ is defined for every $\theta \in (0, 1]$.

For fixed $\la$ and $q$, we call  $f(\theta):=\theta - \phi(\theta)$ the \textbf{drift function}. 
which is defined on $(\max\{\theta_{min}, 0\}, 1]$.

%

\begin{lemma} 
	\label{PositivityF}
	When $q = \la < 2$, the drift function $f$ is non-negative for any $\theta \in [\xi, 1]$, 
	where $\xi$ is an arbitrarily small positive constant.
\end{lemma}

\begin{proof}
	When $\la = q < 2$, the drift function $f$ does not have a positive root,	
	it is continuous in (0,1], and $f(1)>0$; see Lemma 2.5 in~\cite{BGJ} and Fact 3.5 in~\cite{ABthesis}.
	Since $\lim_{\theta \rightarrow 0}f(\theta) = 0$, the result follows.
\end{proof}

\subsection{Shrinking a large component: proof of Lemma~\ref{Lemma1}}

The proof of Lemma~\ref{Lemma1} uses the following lemma, which follows directly from standard random graph estimates and Hoeffding's inequality.  To simplify the notation, we let $\hat{L}(X) := L_1(X)/n^{2/3}$.
We use $A(X)$ to denote the number of vertices activated by the step CM dynamics from configuration $X$.
Let $\Lambda_t$ denote the event that the largest component of the configuration is activated in step $t$.

\begin{lemma} [\cite{ABthesis}, Claim~3.45]
	\label{C345}
	Suppose $ \mathcal{R}_2(X_t) = O(n^{4/3})$ and $\hat{L}(X_t) \geq B$ for a large constant $B$, 
	and let $C$ be a fixed large constant. Then
	\begin{enumerate}
		\item  $\Pr\left[\mathcal{R}_2(X_{t+1}) < \mathcal{R}_2(X_t) + \frac{Cn^{4/3}}{\sqrt{\hat{L}(X_t)}} ~\middle|~ X_t, \Lambda_t \right] = 1 - O\left(\hat{L}(X_{t})^{-1/2} \right)$.
		\item  $\Pr\left[\mathcal{R}_2(X_{t+1}) < \mathcal{R}_2(X_t) + \frac{Cn^{4/3}}{\sqrt{\hat{L}(X_t)}} ~\middle|~ X_t, \neg\Lambda_t \right] = 1 - O\left(\hat{L}(X_{t})^{-1/2} \right)$.
	\end{enumerate}
\end{lemma}

\begin{proof}[Proof of Lemma~\ref{Lemma1}]
Let $\hat{T}$ be the first time $t$ when $\hat{L}(X_t) \leq \delta n^{1/3}$,  let $T'$ be a large constant we choose later; 
we set $T := \min\{\hat{T}, T'\}$.
Observe that with constant probability 
the largest component in the configuration is activated by the CM dynamics for every $t \le T'$;
i.e., the event $\Lambda_t$ occurs for every $t \le T'$.
Let us assume this is the case and fix $t < T$. 
Suppose  $\mathcal{R}_2(X_t) \leq \mathcal{R}_2(X_0) + t \cdot \frac{C}{\sqrt{\delta}} n^{\frac{7}{6}}$ where $C$ is the positive constant from Lemma~\ref{C345}.
We show that with high probability:
\begin{enumerate}[(i)]
	\item $\mathcal{R}_2(X_{t+1}) \leq \mathcal{R}_2(X_0) + t \cdot \frac{C}{\sqrt{\delta}} n^{\frac{7}{6}}$; and 
	\item $L_1(X_{t+1}) \le L_1(X_t) - \xi n$ where $\xi$ is a positive constant independent of $t$ and $n$.
\end{enumerate} 
In particular, it suffices to set $T' = (1-\delta)/\xi$ for the lemma to hold.

First, we show that $A(X_t)$ is concentrated around its mean. 	
Let $L_1(X_t) := \theta_t n$ and $L_1(X_{t+1}) := \theta_{t+1} n$.
Let $\E[A(X_t) \mid \Lambda_t] = \mu_t = \frac{n}{q}+(1-\frac{1}{q})\cdot\theta_t n$, $\gamma:=n^{5/6}$, and $J_t:=[\mu_t - \gamma, \mu_t + \gamma]$.
Hoeffding's inequality implies
\begin{align*}
\Pr\left[A(X_t) \in J_t \mid \Lambda_t\right] 
& \geq 1 - 2\exp\left(\frac{-2\gamma^2}{ \mathcal{R}_2(X_t)}\right) 
= 1 - e^{-\Omega(n^{1/3})}.
\end{align*}
If $A(X_t) \in J_t$, then the random graph $G(A(X_t), p)$ is super-critical since 
$$A(X_t) \cdot p \geq (\mu_t - \gamma)\cdot \frac{q}{n} 
= \left[\frac{n}{q}+ \left(1-\frac{1}{q}\right)\cdot\theta_t n - n^{5/6} \right] \cdot \frac{q}{n} 
= 1 + (q-1)\theta_t - o(1) > 1.$$

Next, for a super-critical random graph, 
Lemma~\ref{L27} provides a concentration bound for the size of largest new component, 
provided $A(X_t) \in J_t$.
To see this, we write $G(A(X_t), \la/n)$ as 
$$G \left(\mu_t + m, k(\theta_t, q)\cdot q/\mu_t \right),$$ 
where $m:=A(X_t)-\mu_t$; 
notice that $|m| \leq \gamma = o(n)$.
Let $H \sim G \left(\mu_t + m, k(\theta_t, q)\cdot q/\mu_t \right)$.
Since
$$k(\theta_t, q)\cdot q = 1 + (q-1) \theta_t > 1 + \delta (q-1) > 1 $$
holds regardless of $n$,
Lemma~\ref{L27} implies that for $\phi(\theta_t) > 0$ defined in Section~\ref{subsec:drift}, with high probability
$$L_1(H) \in \left[\phi(\theta_t)n - \sqrt{n\log n} - |m|, 
\phi(\theta_t)n+ \sqrt{n\log n} + |m|\right].$$

Note that $L_1(H) = \Omega(n)$ w.h.p.; hence, since $L_2(X_t) = O(n^{2/3})$ we have $L_1(X_{t+1}) = L_1(H)$ w.h.p.\ 
We have shown that w.h.p.\  
$$\theta_{t+1} - \theta_t 
\leq \phi(\theta_t) + \frac{\sqrt{n\log n}}{n} -\theta_t - \frac{|m|}{n} 
= -f(\theta_t) + \sqrt{\frac{\log n}{n}}  - \frac{|m|}{n},$$
where $f$ is the drift function defined in Section~\ref{subsec:drift}.
By Lemma~\ref{PositivityF}, we know $f(\theta_t) > \xi_1 > 0$ for sufficiently small constant $\xi_1$ (independent of $n$ and $t$). Hence, w.h.p.\ for sufficiently large $n$
$$L_1(X_{t+1}) - L_1(X_t) \leq -\xi_1 n + o(n) \leq \frac{-\xi_1 n}{2};$$
this establishes (ii) from above.
 
For (i), note that 	for $t < T$ we have $\hat{L}(X_t) > \delta n^{1/3}$, so Lemma~\ref{C345} implies, 
$$ \Pr\left[\mathcal{R}_2(X_{t+1}) < \mathcal{R}_2(X_0) + t \cdot \frac{C}{\sqrt{\delta}} n^{\frac{7}{6}} + \frac{Cn^{4/3}}{\sqrt{\delta} n^{1/6}}\right] = 1 - o(1).$$
A union bound implies that these two events occur simultaneously w.h.p.\ and the result follows.	
\end{proof}

\subsection{Shrinking a medium size component: proof of Lemma~\ref{lemma:burn-in-subphase3}}
\label{subsec:burn-in3}

In the third sub-phase of the burn-in period, we show that 
${L_1}(X_t)$ contracts at a constant rate;
the precise description of this phenomenon is captured in the following lemma.


\begin{lemma}
	\label{C344}
	Suppose $\mathcal{R}_2(X_{t}) = O(n^{4/3})$, $\delta n^{1/3} \geq \hat{L}(X_{t}) \geq B$ for a large constant $B := B(q)$, and a small constant $\delta(q, B)$. Then:
	\begin{enumerate}
		\item There exists a constant $\alpha :=\alpha(B, q, \delta) < 1$ such that 
		$$\Pr\left[L_1(X_{t+1}) \leq \max \{\alpha L_1(X_t) ,L_2(X_t)\} \mid X_t, \Lambda_t\right]\geq 1 - \exp\left(-\Omega\left(\hat{L}(X_{t})\right)\right);$$
		\item $\Pr\left[L_1(X_{t+1}) = L_1(X_t) | X_t, \neg\Lambda_t\right] \geq 1 - O\left(\hat{L}(X_{t})^{-3} \right).$ 
	\end{enumerate}
\end{lemma}

Since Lemmas~\ref{C345} suggests $\mathcal{R}_2(X_{t}) = O(n^{4/3})$ with reasonably high probability throughout the execution of the sub-phase,
Lemmas~\ref{C344} and~\ref{C345} can be combined to derive the following more accurate
contraction estimate which will be crucial in the proof Lemma~\ref{lemma:burn-in-subphase3}.
	
	\begin{lemma}
	\label{Lemma2}
	Suppose $g(n)$ is an arbitrary function with range in the interval $\left[B^6, \delta n^{1/3}\right]$ where  
	$B$ is a large enough constant such that for $x \ge B^6$ we have  $x \geq B (\log x)^8$,
	and $\delta := \delta(q, B)$ is a small constant.
	
	Suppose $X_0$ is such that $g(n) \geq \hat{L}(X_0) \geq B (\log g(n))^8$ and $\mathcal{R}_2(X_0) = O(n^{4/3})$, 
	then there exists a constant $D$ and $T = O(\log g(n))$ such that at time $T$, 
	$ \hat{L}(X_T) \leq \max\{ B (\log g(n))^8 , D\}$ and $\mathcal{R}_2(X_T) \leq \mathcal{R}_2(X_0) + O\left(\frac{n^{4/3}}{\log g(n)}\right)$ 
	with probability at least $1 - O\left(\log^{-1} g(n) \right)$.
	\end{lemma}
	

We first provide a proof for Lemma~\ref{lemma:burn-in-subphase3}
that recursively uses the contraction estimate of Lemma~\ref{Lemma2}. 


\begin{proof}[Proof of Lemma~\ref{lemma:burn-in-subphase3}]
	Let $B$ be a constant large enough so that $\forall x \geq B^6$, we have $x \geq (\log x)^{48}$.
	Suppose $\hat{L}(X_{0}) \leq \delta n^{1/3}$ and $\mathcal{R}_2(X_{0}) = O(n^{4/3})$ 
	for the constant $\delta = \delta(q,B)$ from Lemma~\ref{Lemma2}. 
	Suppose also $ \hat{L}(X_{0}) \geq B^6$; otherwise there is nothing to prove. 
	
	Let $g_0(n):=\delta n^{1/3}$ and $g_{i+1}(n):= B (\log g_i(n))^8$ for all $i \geq 0$.
	Let $K$ be defined as the minimum natural number such that $g_K(n) \leq B^6$.
	Note that $K = O(\log^* n)$.
	Assume at time $t \ge 0$, there exists an integer $j \ge 0$ 
	such that $X_{t}$ satisfies:
	\begin{enumerate}
	\item $g_{j+1}(n) \leq \hat{L}(X_{t}) \leq g_j(n)$, {and}
	\item  $\mathcal{R}_2(X_{t}) = O(n^{4/3}) + O\left(\sum_{k=0}^{j-1} \frac{n^{4/3}}{\log g_{k}(n)}\right)$.
	\end{enumerate}
	
	We show there exists time $t' > t$ such that properties 1 and 2 hold for
	$X_{t'}$ for a different index $j' > j$. The following 
	bounds on sums and products involving the $g_i$'s will be useful; the proof is elementary and delayed to the end of this section.
	\begin{claim}
		\label{claim1.1}
		Let $K$ be defined as above. 
		$ \forall j < K,$
		\begin{enumerate}[(i)]
			\item For any positive constant $c$, we have $\prod_{i=0}^{j} \left( 1 - \frac{c}{\log g_{i}(n)} \right) \geq 1 - \frac{1.5 c}{\log g_j(n)}$ 
			\item $\sum_{i=0}^{j} \frac{1}{\log g_i(n)} \leq \frac{1.5}{\log g_j(n)} $
		\end{enumerate}
		
	\end{claim}

By part (ii) of this claim, note that 
$$ 
O\left(\sum_{k=0}^{j-1} \frac{n^{4/3}}{\log g_{k}(n)}\right) =   
O\left(\frac{n^{4/3}}{\log g_{j-1}(n)}\right) = O(n^{4/3}).$$ 
	
	Hence,
	Lemma~\ref{Lemma2} implies  
	that with probability $1 - O\left((\log g_j(n))^{-1} \right)$
	there exist a time $t' \leq t + O( \log g_j(n) )$ and a large constant $D$ such that
	$ \hat{L}(X_{t'}) \leq \max\{B(\log g_j(n))^8, D\}$
	and
	$\mathcal{R}_2(X_{t'}) \leq \mathcal{R}_2(X_{t}) + O\left(\frac{n^{4/3}}{\log g_j(n)}\right).$
	If $\hat{L}(X_{t'}) \le \max \{D, B^6\}$ we are done.
	Hence, suppose otherwise that $\hat{L}(X_{t'}) \in (B^6, \log g_{j+1}(n)]$. 
	Since the interval $(B^6, \log g_{j+1}(n)]$ is completely covered by the union of the intervals $[g_{j+2}, g_{j+1}]$, ..., $[g_{K}, g_{K-1}]$,
	there must be an integer $j'\ge j+1$ such that $g_{j'+1}(n) \leq \hat{L}(X_{t'}) \leq g_{j'}(n)$. 
	Also, notice
	\begin{align*}
		\mathcal{R}_2(X_{t'}) & \le  \mathcal{R}_2(X_{t}) + O\left(\frac{n^{4/3}}{\log g_j(n)}\right) 
		=  O(n^{4/3}) + O\left(\sum_{k=0}^{j-1} \frac{n^{4/3}}{\log g_{k}(n)}\right)  + O\left(\frac{n^{4/3}}{\log g_j(n)}\right) \\
		&= O(n^{4/3}) + O\left(\sum_{k=0}^{j} \frac{n^{4/3}}{\log g_{k}(n)}\right)
		= O(n^{4/3}) + O\left(\sum_{k=0}^{j'-1} \frac{n^{4/3}}{\log g_{k}(n)}\right).
	\end{align*}
	By taking at most $K$ steps of induction,
	we obtain that there exist constants $C$ and $c$ such that 
	with probability at least 
	$\rho := \prod_{i=0}^{K-1} \left(1 - \frac{c}{\log g_{i}(n)}\right),$
	there exists a time
	$$t_{K} \leq \sum_{i=0}^{K-1} C \log g_i(n)$$ 
	that satisfies $ \hat{L}(X_{t_{K}}) \leq g_{K}(n) \leq B^6 $ and $\mathcal{R}_2(X_{t_{K}}) = O(n^{4/3})$. 
	Observe that $t_{K}$ is a time when our goal has been achieved, 
	so it only remains to show that $\rho = \Omega(1)$ and $t_{K} = O(\log n)$.
	The lower bound on $\rho$ follows from part (i) of Claim~\ref{claim1.1}:
	\begin{align*}
	\prod_{i=0}^{K-1} 1 - \frac{c}{\log g_{i}(n)} 
	& \geq 1 - \frac{1.5 c}{\log g_{K-1}(n)}  > 1 - \frac{1.5 c}{\log B^6} = \Omega(1).
	\end{align*}
	
	\noindent
	By noting that $K = O(\log^*n)$, we can also bound $t_{K}$ since
	$
	\sum_{i=0}^{K-1} C \log g_i(n)$
	is at most $\log g_0(n) + (K-1)\log g_1(n) 
	= O(\log n)$.	
\end{proof}

Before proving Lemma~\ref{Lemma2} we provide the proof of Lemma~\ref{C344}.

\begin{proof}[Proof of Lemma~\ref{C344}]
		
	We start with part 1.
	Let $\mu_t := \E[A(X_t) | \Lambda_t, X_t]$,
	$\gamma_t:=\sqrt{\hat{L}(X_{t})} \cdot n^{2/3}$, and $J_t:=[\mu_t - \gamma_t, \mu_t + \gamma_t]$. Hoeffding's inequality implies that
	\begin{equation*}
		\begin{split}
			\Pr\left[A(X_t) \in J_t \mid \Lambda_t, X_t\right] 
			& \geq 1 - 2\exp\left(\frac{-2\gamma_t^2}{ \mathcal{R}_2(X_t)}\right)   
			= 1 - \exp \left(- \Omega(\hat{L}(X_{t})) \right).
		\end{split}
	\end{equation*}
	
	Let $m := \mu_t +\gamma_t$, $G \sim G(m, \frac{q}{n})$ and $\hat{G} \sim G(A(X_t), p)$. Then, the monotonicity of the largest component in a random graph implies that for any $\ell > 0$
	\begin{align*}
			& \Pr\left[L_1(\hat{G}) > \ell \mid A(X_t)  \in J_t\right] \\
			& = \sum_{a \in J_t} \Pr\left[L_1(\hat{G}) > \ell\mid A(X_t) = a\right]\Pr\left[A(X_t) = a \mid A(X_t) \in J_t\right] \\
		   & \leq \sum_{a \in J_t} \Pr\left[L_1(\hat{G}) > \ell\mid A(X_t) = m\right] \Pr\left[A(X_t) = a \mid A(X_t) \in J_t\right] \\
			& = \Pr[L_1(G) > \ell] \sum_{a \in J_t} \Pr[A(X_t) = a \mid A(X_t) \in J_t] \\
			& = \Pr[L_1(G) > \ell].
	\end{align*}
	
	We bound next $\Pr[L_1(G) > \ell]$. For this,
	we rewrite
	$G(m, \frac{q}{n})$ as $G\left(m, \frac{1+\epsilon}{m}\right)$; since 
	$$\mu_t  = \hat{L}(X_{t}) \cdot n^{2/3} + \left(n- \hat{L}(X_{t}) n^{2/3}\right)q^{-1}$$
	we have
	$$\epsilon = m\cdot \frac{q}{n} - 1 = \left(q - 1 + \frac{q}{\sqrt{\hat{L}(X_{t})}}\right) \frac{\hat{L}(X_{t})}{n^{1/3}}.$$ 
	Thus, 
	\begin{align*}
			\epsilon^3\cdot m 
			&= \left(q - 1 + \frac{q}{\sqrt{\hat{L}(X_{t})}}\right)^3 \frac{ \hat{L}(X_{t})^3}{n} \left(\hat{L}(X_{t})  n^{2/3} + \frac{n- \hat{L}(X_{t}) n^{2/3}}{q} + \sqrt{\hat{L}(X_{t})}n^{2/3}\right) \\
			& \geq \left(q - 1 + \frac{q}{\sqrt{\hat{L}(X_{t})}}\right)^3 \cdot \frac{\hat{L}(X_{t})^3}{n} \cdot \frac{n}{q} \\
			& \geq \frac{1}{q}\cdot \left((q-1)^3 \hat{L}(X_{t})^3 + q^3\sqrt{\hat{L}(X_{t})}^3\right)\\
			& \geq q^2 \hat{L}(X_{t})^{3/2} \geq 100 \hat{L}(X_{t}),
	\end{align*}
	where the last inequality follows from the fact that $\hat{L}(X_{t}) > B$, where $B=B(q)$ is a sufficiently large constant.
	
	Since $\epsilon^3\cdot m\geq1$, Lemma~\ref{T58} implies 
	$$ \Pr\left[ \lvert L_1(G) - 2\epsilon m \rvert > \sqrt{\hat{L}(X_{t})} \sqrt{\frac{m}{\epsilon}}\right] 
	= e^{-\Omega\left(\hat{L}(X_{t})\right) }.$$
	Let $c_1=2\sqrt{\frac{1+(q-1)\delta}{q(q-1)}}$.
	The upper tail bound implies
	$$\Pr\left[L_1(G) \leq 2\epsilon m + c_1 n^{2/3}\right] \geq 1 - e^{- \Omega\left( \hat{L}(X_{t}) \right)}.$$ 
	We show next that $2\epsilon m + c_1 n^{2/3} \le \alpha L_1(X_t)$ for some $\alpha \in (0,1)$. 
	\begin{align*}
		&2\epsilon m + c_1 n^{2/3} \\
  = & 2\left(q - 1 + \frac{q}{\sqrt{\hat{L}(X_{t})}}\right) \frac{\hat{L}(X_{t})}{n^{1/3}} \left(\hat{L}(X_{t}) n^{2/3} + \frac{n-\hat{L}(X_{t}) n^{2/3}}{q} + \sqrt{\hat{L}(X_{t})}n^{2/3}\right) + c_1 n^{2/3}\\
  = & \frac{2}{q} \left(q - 1 + \frac{q}{\sqrt{\hat{L}(X_{t})}}\right) \frac{\hat{L}(X_{t})}{n^{1/3}} \left[ n + \left( q - 1 + \frac{q}{\sqrt{\hat{L}(X_{t})}}\right) \hat{L}(X_{t}) n^{2/3} \right] + c_1 n^{2/3} \\
 = & \frac{2}{q} \left(q - 1 + \frac{q}{\sqrt{\hat{L}(X_{t})}} + \frac{c_1q}{\hat{L}(X_{t})}\right) \hat{L}(X_{t}) n^{2/3} + \frac{2}{q} \left(q - 1 + \frac{q}{\sqrt{\hat{L}(X_{t})}}\right)^2 \hat{L}(X_{t})^2 n^{1/3} \\
 \leq & \frac{2}{q} \left[ \delta\left(q - 1 + O\left(\hat{L}(X_{t})^{-1/2}\right) \right)^2 + \left(q - 1 + O\left(\hat{L}(X_{t})^{-1/2} \right)\right)\right] \hat{L}(X_{t}) n^{2/3},
	\end{align*} 
	where in the last inequality we use the assumption that $\delta n^{1/3} \geq \hat{L}(X_{t})$.
	For sufficiently small $\delta$ and sufficiently large $B$, $\exists ~\alpha < 1$ such that 
	$$\alpha > \frac{2}{q}\left[ \delta\left(q - 1 + \frac{2q}{B^{1/2}}\right)^2 + \left(q - 1 + \frac{2q}{B^{1/2}}\right)\right].$$ 
	Consequently, $ L_1\left(G\right) \leq 2\epsilon m + c_1 n^{2/3} \leq \alpha L_1(X_t)$ 
	with probability $1 - \exp\left(-\Omega(\hat{L}(X_{t})\right)$. 
	If that is the case, $L_1(X_{t+1}) \leq \max\left\{\alpha L_1(X_t), L_2(X_t)\right\} =: L^+$.
	Therefore, 
	\begin{align*}
			& \Pr\left[L_1(X_{t+1}) \leq L^+ \mid X_t, \Lambda_t\right] \\
			& \geq
			\Pr\left[L_1(X_{t+1}) \leq L^+ \mid X_t, \Lambda_t , A(X_t) \in J_t\right] \cdot \Pr\left[A(X_t) \in J_t \mid X_t, \Lambda_t \right] \\
			& \geq 1 - \exp\left(-\Omega(\hat{L}(X_{t}))\right),
	\end{align*}
	which concludes the proof of part 1.
	
	For part 2, note first that 
	when the largest component is inactive, we have 
	$L_1(X_{t+1}) \geq L_1(X_t)$; hence, it is sufficient to show that $L_1(X_{t+1}) \leq L_1(X_t)$ with the desired probability.
	
	Let $\mu_t' := \E\left[A(X_t) \mid \neg \Lambda_t, X_t\right] = \left(n- \hat{L}(X_{t}) n^{2/3} \right)q^{-1}$, 
	$\gamma_t':=\sqrt{\hat{L}(X_{t})} \cdot n^{2/3}$, 
	and $J_t':= [\mu_t' - \gamma_t', \mu_t' + \gamma_t']$.
	By Hoeffding's inequality, 
	$$\Pr\left[A(X_t) \in J_t' \mid \neg\Lambda_t, X_t\right] \geq 1 - \exp\left(- \Omega\left( \hat{L}(X_{t}) \right)\right).$$ 

	Let $G\sim G(A(X_t), p)$, $m = \mu_t' + \gamma_t'$ and let $G^+ \sim G\left(\mu_t' + \gamma_t', p\right) $,  
	By monotonicity of the largest component in a random graph, 
	$$\Pr\left[L_1(G) > L_1(X_t) \mid A(X_t) \in J_t'\right] \leq \Pr\left[L_1(G^+) > L_1(X_t)\right].$$
	Rewrite $G\left(\mu_t' + \gamma_t', p\right)$ as $G\left(m, \frac{1 + \epsilon}{m}\right)$, where 
	$$\epsilon 
	= \left(\frac{n - \hat{L}(X_{t}) n^{2/3}}{q} + \sqrt{\hat{L}(X_{t})} n^{2/3} \right) \cdot \frac{q}{n} - 1 =
	\left(\sqrt{\hat{L}(X_{t})}q -\hat{L}(X_{t}) \right) n^{-1/3}. $$ 
	
	From this bound, applying Lemma~\ref{T512} to $G^+$, we obtain 
	$$\E\left[\mathcal{R}_1(G^+)\right]
	= O\left( \frac{m}{ \epsilon } \right) 
	= O\left( \frac{n^{4/3}}{\hat{L}(X_{t})} \right). $$ 
	Hence, $\E\left[L_1(G^+)^2\right] = O\left(n^{4/3}/\hat{L}(X_{t})\right)$ and
	by Markov's inequality
	$$\Pr\left[L_1(G^+) > \hat{L}(X_{t}) n^{2/3}\right] = \Pr\left[L_1(G^+)^2 > \hat{L}(X_{t})^2 n^{4/3}\right] 
	\leq \frac{\E[L_1(G^+)^2]}{\hat{L}(X_{t})^2 n^{4/3}} 
	= O\left(\frac{1}{\hat{L}(X_{t})^3} \right).$$
	
	To conclude, we observe that
	\begin{align*}
			& \Pr[L_1(X_{t+1}) \leq L_1(X_t) \mid X_t, \neg\Lambda_t] \\				
			& \geq \Pr\left[L_1(G) \leq L_1(X_t) \mid X_t, \neg\Lambda_t, A(X_t) \in J_t'\right] \Pr\left[A(X_t) \in J_t' \mid X_t, \neg\Lambda_t\right] \\
			& \geq \left(1 - e^{- \Omega\left( \hat{L}(X_{t}) \right) }\right) \left(1 - O\left(\frac{1}{\hat{L}(X_{t})^3} \right)\right) = 1 - O\left(\frac{1}{\hat{L}(X_{t})^3} \right),
	\end{align*}
	as desired.
\end{proof}


We are now ready to prove Lemma~\ref{Lemma2}.

\begin{proof} [Proof of Lemma~\ref{Lemma2}]
Suppose $\mathcal{R}_2(X_0) \le D_1^2n^{4/3}$ for a constant $D_1$.
Let $T':=B'\log g(n)$, where $B'$ is a constant such that 
$B' \log g(n)= 2q \log_{1/\alpha}\left(\frac{ g(n)}{B (\log g(n))^8}\right)$ and $\alpha := \alpha(B, q, \delta)$ is the constant from Lemma~\ref{C344}.
Let $\hat{T}$ be the first time 
$$\hat{L}(X_{t}) \leq \max\{ B (\log g(n))^8, D\},$$
where $D$ is a large constant we choose later. 
Let $T:=T'\wedge \hat{T}$, where the operator $\wedge$ takes the minimum of the two numbers.
Define $e(t)$ as the number of 
steps up to time $t$ in which the largest component of the configuration is activated.

To facilitate the notation, we define the following events. (The constants $C$ and $\alpha$ are those from Lemmas~\ref{C344} and \ref{C345}, respectively).
\begin{enumerate}
	\item Let $H_i $ denote $\hat{L}(X_{i}) > \max\{ B (\log g(n))^8, D\}$;
	\item Let $F_i $ denote $\mathcal{R}_2(X_{i}) \leq  \mathcal{R}_2(X_{i-1})  + C n^{4/3}\hat{L}(X_{i-1})^{-1/2}$; let us assume $F_0$ occurs;   
	\item Let $F'_i$ denote $\mathcal{R}_2(X_{i}) \leq \mathcal{R}_2(X_{i-1})  +  C n^{4/3} (\log g(n))^{-4} B^{-1/2}$; again, we assume $F'_0$ occurs;  
	\item Let $Q_i$ denote $\hat{L}(X_{i}) \leq \max\{\alpha^{e(i)}\hat{L}(X_{0}), D \}$;
	\item Let $Base_i$ be the intersection of $\{F'_0, Q_0, H_0\},  ..., \{F'_{i-1}, Q_{i-1}, H_{i-1}\},$ 
	and $ \{F'_i, Q_i$\}.
\end{enumerate}

By induction, we find a lower bound for the probability of $Base_T$.
For the base case, note that $\Pr[Base_0] $ $= 1$ by assumption. 
Next we show $$\Pr\left[Base_{i+1\wedge T} \mid Base_{i\wedge T}\right] = 1 - O( (\log g(n))^{-4} ).$$  
If $T \leq i$, then $Base_{i\wedge T} = Base_{ T} = Base_{i+1\wedge T}$,
so the induction holds.
If $T > i$, then we have $H_i$. 
By the induction hypothesis $F'_1, F'_2, ..., F'_{i-1}$, 
$$ \mathcal{R}_2(X_{i}) \leq \mathcal{R}_2(X_{0}) + i \cdot C n^{4/3} (\log g(n))^{-4} B^{-1/2}.$$ 
Moreover, since $i < T \leq T' = B' \log g(n)$ and $\mathcal{R}_2(X_0) \le D_1^2n^{4/3}$, we have 
$$ \mathcal{R}_2(X_{i}) \leq D_1^2n^{4/3}  + CB' n^{4/3} (\log g(n))^{-3} B^{-1/2}.$$ 
Given $\mathcal{R}_2(X_{i}) = O(n^{4/3})$ and $H_i$, Lemma \ref{C345} implies that $F_{i+1}$ occurs with probability 
$$1 - O(\hat{L}(X_i)^{-1/2}) =  1 - O( (\log g(n))^{-4} ).$$
In addition, note that $F_{i+1} \cup H_i $ leads to $F'_{i+1}$.
Let $\mathbbm{1}(\Lambda_t)$ be the indicator function for the event $\Lambda_t$. 
Given $H_i, Q_i$ and $\mathcal{R}_2(X_{i}) = O(n^{4/3})$, Lemma \ref{C344} implies 
\begin{equation}
	\label{eq:l1-contract}
	L_1(X_{i+1}) \leq \max \{\alpha^{\mathbbm{1}(\Lambda_t)} L_1(X_i) ,L_2(X_i) \}
\end{equation}
with probability at least $1 - O\left(\hat{L}(X_{t})^{-3}\right) = 1 - O\left( (\log g(n))^{-24} \right)$.

Dividing equation~(\ref{eq:l1-contract}) by $n^{2/3}$, we obtain $Q_{i+1}$ for large enough $D$. 
In particular, we can choose $D$ to be $D_1 + 2$. 
A union bound then implies
$$\Pr[Base_{i+1\wedge T} \mid Base_{i\wedge T}] \geq \Pr\left[Base_{i+1 \wedge T} \mid Base_i, H_i\right] = 1 - O( (\log g(n))^{-4} ).$$ 

The probability for $Base_{T}$ can then be bounded as follows:
$$\Pr[Base_{T}] \geq \prod_{i=0}^{T-1} \Pr[Base_{i+1 \wedge T} \mid Base_{i \wedge T}] 
= \prod_{i=0}^{T-1} 1 - O( (\log g(n))^{-4} )= 1 - O( (\log g(n))^{-3} ).$$

Next, let us assume $Base_T$. Then we have 
$$\mathcal{R}_2(X_{T}) \leq  \mathcal{R}_2(X_{0}) + T' \cdot C n^{4/3} (\log g(n))^{-4} B^{-1/2} = \mathcal{R}_2(X_{0}) + O\left( n^{4/3} (\log g(n))^{-3} \right).$$
Notice that if $T = \hat{T}$ then the proof is complete. 
Consequently, it suffices to show $\hat{T} \le T'$ with probability at least $1 - g(n)^{-\Omega(1)}$. 

Observe that $K := e(T')$ is a binomial random variable $Bin\left(T', 1/q\right)$,  
whose expectation is $\frac{T'}{q}  = \frac{B'}{q} \log g(n)$. 
By Chernoff bound 
$$\Pr\left[K < \frac{B'}{2q} \log g(n)\right] \leq \exp\left(-\frac{B'}{16q} \log g(n)\right) =  g(n)^{-\Omega(1)}.$$
If indeed $T = T'$ and $ K \geq \frac{B'}{2q} \log g(n)$, then the event $Q_T$ implies
$$\hat{L}(X_{T}) < \alpha^{e(T)} \hat{L}(X_{0}) \leq \alpha^{\log_{\alpha}\left(\frac{B (\log g(n))^8}{g(n)}\right)} \hat{L}(X_{0}) 
= \frac{B (\log g(n))^8}{g(n)} \hat{L}(X_{0}) \leq B (\log g(n))^8,$$
which leads to $\hat{T} \le T$.
Therefore, 
$$\Pr \left[\hat{T} > T' \mid Base_T\right] \le \Pr \left[K < \frac{B'}{2q} \log g(n) \right] =  g(n)^{-\Omega(1)},$$
as desired.
\end{proof}

\begin{proof}[Proof of Claim~\ref{claim1.1}]
	We first show the following inequality:
	\begin{equation}
		\label{eq:claim1}
		\frac{1.5}{\log g_j(n)} + \frac{1}{\log g_{j+1}(n)} \leq \frac{1.5}{\log g_{j+1}(n)}.
	\end{equation}
	Note that by direction computation  
	\begin{equation*}
		\begin{split}
			\frac{1.5}{\log g_j(n)} + \frac{1}{\log g_{j+1}(n)} 
			& =  \frac{1.5 (\log B + \log (\log g_j(n))^8) + \log g_j(n)}{\log g_j(n) \log g_{j+1}(n)}. \\
		\end{split}
	\end{equation*}
	From the definition of $K$, we know that $g_j(n) > B^6$ for all $j < K$.
	Hence,  $\log B < \log g_j(n)^{1/6}$. In addition, 
	recall that $B$ is such that $\forall\,x \geq B^6$, we have $x \geq (\log x)^{48}$;
	therefore, $g_j(n) \geq (\log g_j(n))^{48}$. Then, $\log (\log g_j(n))^8 \le \log g_j(n)^{1/6}$. Putting all these together,
	\begin{equation*}
		\begin{split}
		   \frac{1.5 (\log B + \log (\log g_j(n))^8) + \log g_j(n)}{\log g_j(n) \log g_{j+1}(n)}  & \leq \frac{1.5 (\frac{1}{6}\log g_j(n) + \frac{1}{6} \log g_j(n)) + \log g_j(n)}{\log g_j(n) \log g_{j+1}(n)} \\
			  & = \frac{ 1.5 \log g_j(n)}{\log g_j(n) \log g_{j+1}(n)} 
			  = \frac{ 1.5 }{ \log g_{j+1}(n)}
		\end{split}
	\end{equation*}
	The proof of part (i) is inductive. The base case ($i=0$) holds trivially.
	For the inductive step note that
	\begin{equation*}
		\begin{split}
			\prod_{i=0}^{j+1} \left(1 - \frac{c}{\log g_{i}(n)}\right) 
			& = \left(1 - \frac{c}{\log g_{j+1}(n)}\right) \prod_{i=0}^{j} \left(1 - \frac{c}{\log g_{i}(n)}\right)  \\				
			& \geq \left(1 - \frac{c}{\log g_{j+1}(n)}\right) \left(1 - \frac{1.5 c}{\log g_j(n)} \right)  \\
			& \geq 1 - c\left(\frac{1.5}{\log g_j(n)} + \frac{1}{\log g_{j+1}(n)}\right) \\
			& \geq 1 - \frac{ 1.5 c }{ \log g_{j+1}(n)},
		\end{split}
	\end{equation*}
	where the last inequality follows from \eqref{eq:claim1}.
	
	For part (ii) we also use induction. The base case ($i=0$) can be checked straightforwardly.
	For the inductive step, 
	\begin{equation*}
		\begin{split}
			\sum_{i=0}^{j+1} \frac{1}{\log g_i(n)} 
			&  \leq \frac{1}{\log g_{j+1}(n)} + 	\sum_{i=0}^{j} \frac{1}{\log g_i(n)}   \leq  \frac{1}{\log g_{j+1}(n)} + \frac{1.5}{\log g_j(n)} 
			\leq \frac{1.5}{\log g_i(n)},
		\end{split}
	\end{equation*}
	where the last inequality follows from (\ref{eq:claim1}).
	\end{proof}

\section{Coupling to the same component structure: proofs}

	In this section we provide the proofs of Lemma~\ref{lemma:critical:q<2:second-step-first}, Lemma~\ref{lemma:critical:q<2:second-step-second}, Lemma~\ref{lemma:critical:q<2:second-step-2.5} and Lemma~\ref{lemma:critical:q<2:second-step-third}.
	
	\subsection{Continuation of the burn-in phase: proof of Lemma~\ref{lemma:critical:q<2:second-step-first}}
	\label{sec:second-step-first}

	Recall that for a random-cluster configuration $X$, let $A(X)$ denote the random variable corresponding to the number of vertices activated by step~\eqref{CMdynamics1} of the CM dynamics from $X$. 		

	\begin{proof}[Proof of Lemma \ref{lemma:critical:q<2:second-step-first}]
		We show that there exist suitable constants $C$, $D > 0$ and $\alpha \in (0,1)$
		such that if $\mathcal{R}_1(X_t)\le C n^{4/3}$ and $\widetilde{\mathcal{R}}_\omega(X_t) > Dn^{4/3}\on^{-1/2}$, then
		\begin{align}		
		\mathcal{R}_1(X_{t+1}) &\le C n^{4/3},~\text{and} \label{eq:first-lemma:to-prov-1}\\ 
		\widetilde{\mathcal{R}}_\omega(X_{t+1}) &\le (1-\alpha) \widetilde{\mathcal{R}}_\omega(X_t) \label{eq:first-lemma:to-prov-2}
		\end{align}
		with probability $\rho = \Omega(1)$.
		This implies that 
		we can maintain \eqref{eq:first-lemma:to-prov-1}-\eqref{eq:first-lemma:to-prov-2} for $T$ steps with probability $\rho^T$.
		Precisely, if we let
		\begin{align*}
			\tau_1 &= \min \{t > 0: \mathcal{R}_1(X_{t}) > C n^{4/3}\}, \\
			\tau_2 &= \min \{t > 0:\widetilde{\mathcal{R}}_\omega(X_{t}) > (1-\alpha) \widetilde{\mathcal{R}}_\omega(X_{t-1})\}, \\
			T &= \min \{ \tau_1,\tau_2,c\log \on\},
		\end{align*}
		where the constant $c>0$ is chosen such that $(1-\alpha)^{c\log \on} = O(\on^{-1/2})$, 
		then $T = c\log \on$ with probability $\rho^{c \log \on}$. (Note that $\rho^{c 
		\log \on}= {\on}^{-\beta}$ for a suitable constant $\beta > 0$.)
		Hence, $\mathcal{R}_1(X_T) = O(n^{4/3})$ and 
		$$\widetilde{\mathcal{R}}_\omega(X_T) 
		\le \widetilde{\mathcal{R}}_\omega(X_0) \cdot O(\on^{-1/2}) 
		\le \mathcal{R}_1(X_0) \cdot O(\on^{-1/2})  
		= O(n^{4/3}\on^{-1/2}).$$
		The lemma then follows from the fact that $I(X_T) = \Omega(n)$ with probability $1-o(1)$ by Lemma \ref{lemma:rg:isolated} and a union bound.

		To establish \eqref{eq:first-lemma:to-prov-1}-\eqref{eq:first-lemma:to-prov-2},
		let $\mathcal{H}^1_t$ be the event that $A(X_t) \in \left[n/q - \delta n^{2/3},n/q + \delta n^{2/3}\right]$, where $\delta > 0$ is a constant.
		By Hoeffding's inequality, for a suitable $\delta > 0$, $\Pr[\mathcal{H}^1_t] \ge 1 - \frac{1}{8q^2}$ since $\mathcal{R}_1(X_t) = O(n^{4/3})$.
		Let $K_t$ denote the subgraph induced on the inactivated vertices at the step $t$.
		Observe that $\E\big[\widetilde{\mathcal{R}}_\omega(K_t)\big] = \left(1-\frac{1}{q}\right)\widetilde{\mathcal{R}}_\omega(X_{t})$.
		Similarly,  
		$\E\big[\mathcal{R}_1(K_t) - \widetilde{\mathcal{R}}_\omega(K_t)\big] 
		= \left(1-\frac{1}{q}\right) \big(\mathcal{R}_1(X_{t}) - \widetilde{\mathcal{R}}_\omega(X_{t+1})\big)$.
		Hence, by Markov's inequality and independence between activation of each component, with probability at least $1/4q^2$, 
		the activation sub-step is such that $G_u$ satisfies
		$$\widetilde{\mathcal{R}}_\omega(K_t) \le \left(1-\frac{1}{2q}\right)\widetilde{\mathcal{R}}_\omega(X_{t}),$$ 
		and 
		$$\mathcal{R}_1(K_t) - \widetilde{\mathcal{R}}_\omega(K_t) 
		\le \left(1-\frac{1}{2q}\right)\big( \mathcal{R}_1(X_{t}) - \widetilde{\mathcal{R}}_\omega(X_{t+1}) \big).$$
		We denote this event by $\mathcal{H}^2_t$. 
		It follows by a union bound that $\mathcal{H}^1_t$ and $\mathcal{H}^2_t$ happen simultaneously with probability at least $1/8q^2$.
		We assume that this is indeed the case and proceed to discuss the percolation sub-step.
	
		Lemma \ref{lemma:prelim:small-cmpt-var} implies that
		there exists $C_1 > 0$ such that with probability $99/100$,
		$$\widetilde{\mathcal{R}}_\omega\left(G\left(A(X_t), \frac{q}{n}\right)\right) \le C_1\frac{n^{4/3}}{{\on}^{1/2}}.$$
		Hence, 
		\begin{align*}
			\widetilde{\mathcal{R}}_\omega(X_{t+1}) 
			=\widetilde{\mathcal{R}}_\omega(K_t) + \widetilde{\mathcal{R}}_\omega\left(G\left(A(X_t), \frac{q}{n}\right)\right)
			\le \left(1-\frac{1}{2q}\right)\widetilde{\mathcal{R}}_\omega(X_{t}) + C_1\frac{n^{4/3}}{{\on}^{1/2}}
			 \le  (1-\alpha) \widetilde{\mathcal{R}}_\omega(X_{t}),
		\end{align*}
		where the last inequality holds for a suitable constant $\alpha \in (0,1)$ and a sufficiently large $D$ since $\widetilde{\mathcal{R}}_\omega(X_t) > Dn^{4/3}\on^{-1/2}$. 
		
		On the other hand, Lemma~\ref{lemma:prelim:critical-var} implies  
		$\E\left[\mathcal{R}_1\left(G\left(A(X_t), \frac{q}{n}\right)\right)\right] = O(n^{4/3})$. 
		By Markov's inequality, there exists $C_2$ such that, with probability $99/100$, 
		$$
		\mathcal{R}_1\left(G\left(A(X_t), \frac{q}{n}\right)\right) \le C_2n^{4/3}. 
		$$
		For large enough $C$, 
		\begin{align*}
			\mathcal{R}_1(X_{t+1}) &\le \mathcal{R}_1(K_t) +  \mathcal{R}_1\left(G\left(A(X_t), \frac{q}{n}\right)\right) 
			\le \big( 1- \frac{1}{2q} \big)\mathcal{R}_1(X_{t}) +   \mathcal{R}_1\left(G\left(A(X_t), \frac{q}{n}\right)\right) \\
			&\le \big( 1- \frac{1}{2q} \big) Cn^{4/3} +  C_2n^{4/3} 
			\le  Cn^{4/3} 
		\end{align*}
		 Finally, it follows from a union bound that  \eqref{eq:first-lemma:to-prov-1} and \eqref{eq:first-lemma:to-prov-2}
		hold simultaneously with probability at least $\frac{98}{100\cdot 8q^2}$.
	\end{proof}

	\subsection{Coupling to the same large component structure: proof of Lemma \ref{lemma:critical:q<2:second-step-second}}
	\label{sec:coup:llt}
	
	To prove Lemma \ref{lemma:critical:q<2:second-step-second}, we use a local limit theorem to construct a two-step coupling of the CM dynamics that reaches two configurations with the same large component structure.
	The construction of Markov chain couplings using local limit theorems is not common (see~\cite{LNNP} for another example), but it appears to be a powerful technique that may have other interesting applications.
	We provide next a brief introduction to local limit theorems.
	
	\medskip
	\noindent\textbf{Local limit theorem.} \
	Let $m$ be an integer.
	Let $c_1 \le \dots \le c_m$ be integers and for $i=1,\dots,m$, let $X_i$ be the random variable that is equal to $c_i$ with probability $r \in (0,1)$, and it is zero otherwise. 
	Let us assume that $X_1, \dots, X_m$ are independent random variables. Let
	$S_m = \sum_{i=1}^m X_i$, $\mu_m = \E[S_m]$ and $\sigma_m^2 = \var(S_m)$. 
	We say that a \textit{local limit theorem} holds for $S_m$ if for every integer $a \in \Z$:
	\begin{equation}
	\label{eq:prelim:llt}
	\Pr[S_m = a] = \frac{1}{\sqrt{2\pi} \sigma_m} \exp\Big(-\frac{(a-\mu_m)^2}{2\sigma_m^2}\Big) + o(\sigma_m^{-1}).
	\end{equation}
	
	We prove, under some conditions, a local limit theorem that applies to the random variables corresponding to the number of active vertices from small components. 
	Recall that for an increasing positive function $g$ and each integer $k \ge 0$, we defined the intervals
	\[\mathcal{I}_k = \left[\frac{\vartheta m^{2/3}}{2 g(m)^{2^k}},\frac{\vartheta m^{2/3}}{g(m)^{2^k}}\right].\]
	where $\vartheta>0$ is a fixed large constant.

	\begin{theorem}
		\label{thm:prelim:llt-cor}
		Let $m$ be an integer.
		 Let $c_1 \le \dots \le c_m$ be integers, and
		 suppose $X_1, ..., X_m$ are independent random variables such that
		 $X_i$ is equal to $c_i$ with probability $r \in (0,1)$, and $X_i$ is zero otherwise. 
		Let $g:\N \rightarrow \R$ be an increasing positive function such that $g(m) \rightarrow \infty$ and $g(m)=o(\log m)$.
		Suppose $c_m = O\left(m^{2/3}g(m)^{-1}\right)$, $\sum_{i=1}^m c_i^2 = O\left(m^{4/3}g(m)^{-1/2}\right)$ and
		$c_i = 1$ for all $i \le \rho m$, where $\rho \in (0,1)$ is independent of $m$.	
		Let $\ell = \ell(m, g) > 0$ be the smallest integer such that $m^{2/3}g(m)^{-2^\ell} = o(m^{1/4})$.
		If for all $1\le k \le \ell$, we have
		$|\{i:c_i \in \mathcal{I}_k(g )\}| = \Omega(g(m)^{3\cdot2^{k-1}})$, then a local limit theorem holds for $S_m = \sum_{i=1}^m X_i$.	
	\end{theorem}
	
	Theorem~\ref{thm:prelim:llt-cor} follows from a general local limit theorem proved in~\cite{Muk};
	a proof is given in Appendix~\ref{app:llt}.
	We provide next the proof of Lemma~\ref{lemma:critical:q<2:second-step-second}.

		\begin{proof}[Proof of Lemma \ref{lemma:critical:q<2:second-step-second}]
			First, both $\{X_t\}$, $\{Y_t\}$ perform one independent CM step from the initial configurations $X_0$, $Y_0$. 
			We start by establishing that $X_1$ and $Y_1$ preserve the structural properties assumed for $X_0$ and $Y_0$.
			
			By assumption $\mathcal{R}_1(X_{0}) = O(n^{4/3})$, so
			Hoeffding's inequality implies that the number of activated vertices from $X_0$ is such that 
			\[
			A(X_0) \in I := \left[n/q - O( n^{2/3}),n/q + O(n^{2/3})\right]
			\]
			with probability $\Omega(1)$.
			Then, the percolation step is distributed as a 
			\[G\left(A(X_0), \frac{1 + \lambda A(X_0)^{-1/3}}{A(X_0)}\right)\] 
			random graph,
			with $|\lambda| = O(1)$ with probability $\Omega(1)$.
			 Conditioning on this event, from Lemma~\ref{lemma:rg:isolated} 
			 we obtain that $I(X_1)=\Omega(n)$ w.h.p.
			 Moreover, from Lemma \ref{lemma:prelim:critical-var} and Markov's inequality we obtain that $\mathcal{R}_1(X_1) = O(n^{4/3})$ with probability at least $99/100$ and from
			 Lemma~\ref{lemma:prelim:small-cmpt-var} that $\widetilde{\mathcal{R}}_\omega(X_1) = O(n^{4/3}{\on}^{-1/2})$ also with probability at least $99/100$.
			
			We show next that $X_1$ and $Y_1$, in addition to preserving the structural properties of $X_0$ and $Y_0$, 
			also have many connected components with sizes in certain carefully chosen intervals. This fact will be crucial in the design of our coupling. When $A(X_0) \in I$, by Lemmas~\ref{lemma:prelim:cmpt-count} and \ref{lemma:critical:q<2:maintain-trees} and a union bound, for all integer $k \ge 0$ such that $n^{2/3}\on^{-2^k} \rightarrow \infty$, 
			$N_{k}(X_1, \omega) = \Omega(\on^{3 \cdot 2^{k-1}})$ w.h.p. (Recall, that
			$N_{k}(X_1, \omega)$ denotes the number of connected components of $X_1$ with sizes in the interval $\mathcal I_k(\omega)$.)
			We will also require a bound for the number of components with sizes in the interval
			\[J = \left[\frac{cn^{2/3}}{\on^{6}},\frac{2cn^{2/3}}{\on^{6}}\right], \]
			where $c > 0$ is a constant such that $J$ does not intersect any of the $\mathcal I_k(\omega)$'s intervals.
			Let $W_X$ (resp., $W_Y$) be the set of components of $X_1$ (resp., $Y_1$)
			with sizes in the interval $J$.		
			Lemma \ref{lemma:prelim:cmpt-count} then implies that for some positive constants $\delta_1, \delta_2$ independent of $n$,
			$$
			\Pr\Big[|W_X|\ge \delta_1n \big(\frac{\on^{6}}{cn^{2/3}}\big)^{3/2} \Big] 
			\ge 1- \frac{\delta_2}{n}  \Big(\frac{cn^{2/3}}{\on^{6}}\Big)^{3/2}
			= 1 - O(\on^{-9}).
			$$
			All the bounds above apply also to the analogous quantities for $Y_1$ with the same respective probabilities. 
			Therefore, by a union bound, all these properties hold simultaneously for both $X_1$ and $Y_1$ with probability $\Omega(1)$. 
			We assume that this is indeed the case and proceed to describe the second step of the coupling, in which we shall use each of the established properties for $X_1$ and $Y_1$.
			
			Recall $\mathcal{S}_{\omega}(X_1)$ and $\mathcal{S}_{\omega}(Y_1)$ denote the sets of connected components in $X_1$ and $Y_1$, respectively,
			with sizes larger than $B_\omega$. (Recall that $B_\omega = n^{2/3} \on^{-1}$, where $\on = \log \log \log \log n$.)
			Since $\mathcal{R}_1(X_1) = O(n^{4/3})$, the total number of components in $\mathcal{S}_{\omega}(X_1)$ is $O(\on^2)$; 
			moreover, it follows from the Cauchy–Schwarz inequality
			that the total number of vertices in the components in $\mathcal{S}_{\omega}(X_1)$, denoted $\|\mathcal{S}_{\omega}(X_1)\|$, is $O(n^{2/3}\on)$; 
			the same holds for $\mathcal{S}_{\omega}(Y_1)$.	
			Without loss of generality, let us assume that $\|\mathcal{S}_{\omega}(X_1)\| \ge \|\mathcal{S}_{\omega}(Y_1)\|$.		
			Let 
			$$
			\varGamma  = \{C \subset W_Y: \|\mathcal{S}_{\omega}(Y_1) \cup C\| \ge \|\mathcal{S}_{\omega}(X_1)\| \},
			$$
			and let $C_{\rm min} = \arg \min_{C \in \varGamma} \|\mathcal{S}_{\omega}(Y_1) \cup C\|$.
			In words, $C_{\rm min}$ is the smallest subset $C$ of components of $W_Y$ so that 
			the number of vertices in the union of $\mathcal{S}_{\omega}(Y_1)$ and $C$ is greater than that in $\mathcal{S}_{\omega}(X_1)$.
			Since every component in $W_Y$ has size at least $cn^{2/3}\on^{-6}$ and $|W_Y| = \Omega(\on^9)$,
			the number of vertices in $W_Y$ is $\Omega(n^{2/3}\on^3)$ and so $\varGamma \neq \emptyset$.
			In addition, the number of components in $C_{\rm min}$ is $O(\on^9)$. 
			Let $\mathcal{S}_{\omega}'(Y_1) = \mathcal{S}_{\omega}(Y_1) \cup C_{\rm min}$ and
			observe that the number of components in $\mathcal{S}_{\omega}'(Y_1)$ is also $O(\on^9)$
			and that 
			\[0 \le \|\mathcal{S}_{\omega}'(Y_1)\| - \|\mathcal{S}_{\omega}(X_1)\| \le 2cn^{2/3}\on^{-6}. \]
			Note that $\|\mathcal{S}_{\omega}(X_1)\| - \|\mathcal{S}_{\omega}(Y_1)\|$ may be $\Omega(n^{2/3}\on)$ (i.e., much larger than $ \|\mathcal{S}_{\omega}'(Y_1)\| - \|\mathcal{S}_{\omega}(X_1)\| $). 
			Hence, if all the components from $\mathcal{S}_{\omega}(Y_1)$ and $\mathcal{S}_{\omega}(X_1)$ were activated, the difference in the number of active vertices could be $\Omega(n^{2/3}\on)$. This difference cannot be corrected by our coupling for the activation of the small components.
			We shall require instead that all the components from $\mathcal{S}_{\omega}'(Y_1)$ and $\mathcal{S}_{\omega}(X_1)$ are activated so that the difference is $O(n^{2/3}\on^{-6})$ instead.
					
			We now describe a coupling of the activation sub-step for the second step of the CM dynamics. As mentioned, our goal is to design a coupling in which the same number of vertices are activated from each copy.
			If indeed $A(X_1) = A(Y_1)$, 
			then we can choose an arbitrary bijective map $\varphi$ between the activated vertices of $X_1$ and the activated vertices of $Y_1$ and use $\varphi$ to couple the percolation sub-step. Specifically,
			if $u$ and $v$ were activated in $X_1$, the state of the edges $\{u,v\}$ in $X_2$ and $\{\varphi(u),\varphi(v)\}$ in $Y_2$ would be the same. 
			This yields a coupling of the percolation sub-step
			such that $X_2$ and $Y_2$ agree on the subgraph update at time $1$.	
					
			Suppose then that in the second CM step all the components in $\mathcal{S}_{\omega}(X_1)$ and $\mathcal{S}_{\omega}'(Y_1)$ are activated simultaneously.
			If this is the case, then the difference in the number of activated vertices is $d \le 2c n^{2/3}\on^{-6}$.
			We will use a local limit theorem (i.e., Theorem \ref{thm:prelim:llt-cor})
			to argue that there is a coupling of the activation of the remaining components in $X_1$ and $Y_1$ such that the total number of active vertices in both copies is the same with probability $\Omega(1)$.
			Since all the components in $\mathcal{S}_{\omega}(X_1)$ and $\mathcal{S}_{\omega}'(Y_1)$ are activated with probability $\exp(-O(\on^9))$, the overall
			success probability of the coupling will be $\exp(-O(\on^9))$.
				
			Now, let $x_1,x_2,\dots,x_m$ be the sizes of the components of $X_1$ that are not in $\mathcal{S}_{\omega}(X_1)$ (in increasing order). 
			Let $\hat{A}(X_1)$ be the random variable corresponding to the number of active vertices from these components.
			Observe that $\hat{A}(X_1)$ is the sum of $m$ independent random variables, where the $j$-th variable in the sum is equal to $x_j$ with probability $1/q$, and it is $0$ otherwise.		
			We claim that sequence $x_1,x_2,\dots,x_m$ satisfies all the conditions
			in Theorem~\ref{thm:prelim:llt-cor}.
			
			First, note that since the number of isolated vertices in $X_1$ is $\Omega(n)$, $m = \Theta(n)$ and consequently
			$x_m = O(m^{2/3}\omega(m)^{-1})$, $\sum_{i=1}^m x_i^2 = \tilde{R}_\omega(X_1) = O(m^{4/3}\omega(m)^{-1/2})$ and $x_i=1$ for all $i \le \rho m$, 
			where $\rho \in (0,1)$ is independent of $m$. 
			Moreover,  since
			$N_{k}(X_1, \omega) = \Omega(\on^{3 \cdot 2^{k-1}})$
			for all $k \ge 1$ such that $n^{2/3}\on^{-2^k} \rightarrow \infty$, 
			\[|\{i:x_i \in \mathcal{I}_k(\omega )\}| = \Omega(\omega(m)^{3\cdot2^{k-1}}).\]
			 Since $N_0(X_1,\omega)	= \Omega( \on^{3/2}) $, we also have
			$$
			\sum\nolimits_{i=1}^m x_i^2 
			\ge  N_0(X_1,\omega) \cdot \frac{\vartheta^2 n^{4/3}}{4\on^2}
			= \Omega ({m^{4/3}\omega(m)^{-1/2}} ).
			$$		
			Let
			$\mu_X = \E[\hat{A}(X_1)] = q^{-1}\sum_{i=1}^m x_i$ and
			let 
			\[\sigma_X^2 = \var(\hat{A}(X_1)) = q^{-1}(1-q^{-1}) \sum_{i=1}^m x_i^2 = \Theta(m^{4/3}\omega(m)^{-1/2}). \]
			Hence, Theorem \ref{thm:prelim:llt-cor} implies that 
			$\Pr [\hat{A}(X_1) = a] = \Omega\left(\sigma_X^{-1}\right)$
			for any $a \in [\mu_X-\sigma_X,\mu_X+\sigma_X]$.
			Similarly, we get $\Pr [\hat{A}(Y_1) = a ] = \Omega(\sigma_Y^{-1})$ for any $a \in [\mu_Y-\sigma_Y,\mu_Y+\sigma_Y]$,		
			with $\hat{A}(Y_1)$, $\mu_Y$ and $\sigma_Y$ defined analogously for  $Y_1$.
			Note that $\mu_X - \mu_Y  = O(n^{2/3}\on^{-6})$ and $\sigma_X , \sigma_Y= \Theta(n^{2/3} \on^{-1/4})$.
			Without loss of generality, suppose $\sigma_X < \sigma_Y$.
			Then for any $a \in [\mu_X-\sigma_X /2,\mu_Y+\sigma_X /2]$ and $d = O(n^{2/3}\on^{-6})$, we have
			\[
				\min \left\{\Pr [\hat{A}(X_1) = a ], \Pr [\hat{A}(Y_1) = a - d ]\right\} 
				= \min \left\{\Omega ( \sigma_X^{-1}), \Omega ( \sigma_Y^{-1}) \right\}
				= \Omega (\sigma_Y^{-1}).
			\]
			Hence, there exists a coupling $\mathbb P$ of $\hat{A}(X_1)$ and $\hat{A}(Y_1)$ so that 
			$\mathbb P[\hat{A}(X_1) = a,  \hat{A}(Y_1) = a - d] 
			= \Omega(\sigma_Y^{-1})$ for all  $a \in \left[\mu_X-\sigma_X /2,\mu_Y+\sigma_X /2\right]$.
			Therefore, there is a coupling of $\hat{A}(X_1)$ and $\hat{A}(Y_1)$ 
			such that
			$$
			\Pr[\hat{A}(X_1) - \hat{A}(Y_1) = d ] = \Omega \left( {\sigma_X}/{\sigma_Y} \right) = \Omega(1).
			$$ 
					
			Putting all these together, we deduce that $A(X_1) = A(Y_1)$ with probability $\exp(-O(\on^{9}))$.
			If this is the case, the edge re-sampling step is coupled bijectively (as described above) so that $\mathcal{S}_{\omega}(X_2) = \mathcal{S}_{\omega}(Y_2)$.
			
			It remains for us to guarantee the additional desired structural properties of $X_2$ and $Y_2$, 
			which follow straightforwardly from the random graph estimates stated in Section~\ref{RG}.
			First note that by Hoeffding's inequality, with probability $\Omega(1)$,
			$$ \left| A(X_1) - \frac{n}{q} - \frac{(q-1)\|\mathcal{S}_{\omega}(X_1)\|}{q}  \right| = O(n^{2/3}).$$
			
			Hence, in the percolation sub-step the active subgraph is replaced by $F \sim G\left(A(X_1), \frac{1 + \lambda A(X_1)^{-1/3}}{A(X_1)}\right)$, where $|\lambda| = O(\on)$ with probability $\Omega(1)$ since $\|\mathcal{S}_{\omega}(X_1)\| = O(n^{2/3}\on)$.
			Conditioning on this event,
			since the components of $F$ contribute to both $X_2$ and $Y_2$,
			Corollary~\ref{lemma:critical:q<2:maintain-trees} implies that
			w.h.p.\ $\hat{N}_k(2, \on) $ 
			$ = \Omega( \on^{3 \cdot 2^{k-1}} )$
			for all $k \ge 1$ such that $n^{2/3}\on^{-2^k}\rightarrow \infty$.
			Moreover, from Lemma~\ref{lemma:rg:isolated} we obtain that $I(X_2)=\Omega(n)$ w.h.p.
			From Lemma~\ref{lemma:rg:critical:exp-crude-bound} and Markov's inequality, 
			we obtain that $\mathcal{R}_2(X_2) = O(n^{4/3})$ with probability at least $99/100$ and from
			Lemma~\ref{lemma:prelim:small-cmpt-var} that $\widetilde{\mathcal{R}}_\omega(X_2) = O(n^{4/3}{\on}^{-1/2})$ also with probability at least $99/100$.
			All these bounds apply also to the analogous quantities for $Y_2$ with the same respective probabilities. 
		
			Finally, we derive the bound for $L_1(X_2)$ and $L_1(Y_2)$.
			First, notice $L_1(F)$ is stochastically dominated by $L_1(F')$, where
			$F'\sim G\Big(A(X_1), \frac{1 + |\lambda| A(X_1)^{-1/3}}{A(X_1)}\Big)$.
			Under the assumption that $|\lambda| = O(\on)$, if $|\lambda| \rightarrow \infty$, then
			Corollary \ref{lemma:rg:critical:giant-concentration} implies that $L_1(F') = O(|\lambda| A(X_1)^{2/3}) = O(n^{2/3} \on)$ w.h.p.;
			otherwise, $|\lambda| = O(1)$ and by Lemma~\ref{lemma:prelim:critical-var} and Markov's inequality, 
			$L_1(F') = O(n^{2/3})$ with probability at least $99/100$. 
			Thus, $L_1(F) = O(n^{2/3}\on)$ with probability at least $99/100$.
			We also know that the largest inactivated component in $X_1$ has size less than $n^{2/3}\on^{-1}$,
			so $L_1(X_2) = O(n^{2/3} \on)$ with probability at least $99/100$.
			The same holds for $Y_2$.
			Therefore, by a union bound, all these properties hold simultaneously for both $X_2$ and $Y_2$ with probability $\Omega(1)$, as claimed.
		\end{proof}

		\subsection{Re-contracting largest component: proof of Lemma~\ref{lemma:critical:q<2:second-step-2.5}}
		\label{subsec:general-llt}
	
	In Section~\ref{sec:coup:llt}, we designed a coupling argument to ensure that
	the largest components of both configurations have the same size.
	For this, we needed to relax our constraint on the size of the largest component of the configurations.
	In this section we prove Lemma~\ref{lemma:critical:q<2:second-step-2.5}, which ensures that after $O(\log \on)$ steps the largest components of each configuration have size $O(n^{2/3})$ again.
	
	The following lemma is the core of the proof Lemma~\ref{lemma:critical:q<2:second-step-2.5}
	and it may be viewed as a generalization of the coupling from the proof of Lemma \ref{lemma:critical:q<2:second-step-second} 
	using the local limit theorem from Section~\ref{sec:coup:llt}.
	
	We recall some notation from the proof sketch.
	Given two random-cluster configurations $X_t$ and $Y_t$, $W_t$ is maximal matching  between the components of $X_t$ and $Y_t$ that only matches components of equal size to each other.
	We use $M (X_t)$, $M(Y_t)$ for the components in $W_t$ from $X_t$, $Y_t$, respectively, $D(X_t)$, $D(Y_t)$ for the complements of $M (X_t )$, $M (Y_t)$, and
	$
	Z_t = \sum_{\mathcal{C} \in D(X_t) \cup D(Y_t)} |\mathcal{C}|^2.
	$
	For an increasing positive function $g$ and each integer $k \ge 1$, define $\hat{N}_k(t, g) := \hat{N}_k(X_t,Y_t, g)$ as the number of matched pairs in $W_t$ whose component sizes are in the interval 
\[\mathcal{I}_{k}(g) = \left[\frac{\vartheta n^{2/3}}{2g(n)^{2^k}},\frac{\vartheta  n^{2/3}}{g(n)^{2^k}}\right],\]
where $\vartheta>0$ is a fixed large constant (independent of $n$). 
	
	\begin{lemma}
		\label{lemma:activation-second-step-2.5}
		There exists a coupling of the activation sub-step of the CM dynamics such that
		$A(X_t) = A(Y_t)$ with at least $\Omega \left(\frac{1}{\on} \right)$ probability,
		provided $X_t$ and $Y_t$ are random-cluster configurations satisfying
		\begin{enumerate}
			\item $\mathcal{S}_{\omega}(X_t) = \mathcal{S}_{\omega}(Y_t)$;
			\item $Z_t = O\left(\frac{n^{4/3}}{\on^{1/2}}\right)$;
			\item $\hat{N}_k(X_t, Y_t, \on) = \Omega\left(\on ^ {3 \cdot 2 ^{k-1}}\right)$ 	for all $k \ge 1$ such that $n^{2/3}\on^{-2^k}\rightarrow \infty$;
			\item $I(X_t), I(Y_t) = \Omega(n)$.
		\end{enumerate}
	\end{lemma}

	\begin{proof} 
		The activation coupling has two parts.
		First we use the maximal matching $W_t$ to couple the activation
		of a subset of the components in $M(X_t)$ and $M(Y_t)$.
		Specifically, let $\ell$ be defined as in Theorem~\ref{thm:prelim:llt-cor};
		for all $k \in [1, \ell]$, 
		we exclude $\Theta(\on ^ {3 \cdot 2 ^{k-1}})$ pairs of components of size in the interval $\mathcal{I}_k(\omega)$
		and we exclude $\Theta(n)$ pairs of matched isolated vertices.
		(These components exist by assumptions 3 and 4.) 
		All other pairs of components matched by $W_t$ are jointly activated (or not).
		Hence, the number of vertices activated from $X_t$ in this first part of the coupling is the same as that from $Y_t$.
		
		Let $\mathcal{C}(X_t)$ and $\mathcal{C}(Y_t)$ denote the sets containing 
		the components in $X_t$ and components in $Y_t$ not considered to be activated in the first step of the coupling.
		This includes all the components from $D(X_t)$ and $D(Y_t)$, and all the components from 
		$M(X_t)$ and $M(Y_t)$ excluded in the first part of the coupling.	
		Let $A'(X_t)$ and $A'(Y_t)$ denote the number of activated vertices from $\mathcal{C}(X_t)$ and $\mathcal{C}(Y_t)$ respectively.
		The second part is a coupling of the activation sub-step in a way such that 
		$$\Pr \left[ A'(X_t) = A'(Y_t) \right] = \Omega(\on^{-1}).$$
		
		Let $m_x := |\mathcal{C}(X_t)| = \Theta(n)$, and similarly for $m_y:=|\mathcal{C}(Y_t)|$.
		Let $\mathcal{C}_1 \le \dots \le \mathcal{C}_{m_x}$ (resp., $\mathcal{C'}_1 \le \dots \le \mathcal{C'}_{m_y}$) be sizes of components in $\mathcal{C}(X_t)$ (resp., $\mathcal{C}(Y_t)$) in ascending order.
		For all $i\le m_x$, let $\mathcal{X}_i$ be a random variable that
		equals to $\mathcal{C}_i$ with probability $1/q$ and $0$ otherwise,
		which corresponds to the number of activated vertices from $i$th component in $\mathcal{C}(X_t)$. 
		Note that $\mathcal{X}_1, \dots, \mathcal{X}_{m_x}$ are independent.  
		We check that $\mathcal{X}_1, \dots, \mathcal{X}_{m_x}$ satisfy all other conditions of Theorem~\ref{thm:prelim:llt-cor}. 	 
		
		Assumption $\mathcal{S}_{\omega}(X_t) = \mathcal{S}_{\omega}(Y_t)$ and the first part of the activation ensure that 
		$$\mathcal{C}_{m_x} \le B_\omega = O\left( n^{2/3}\on^{-1} \right) = O\left(m_x^{2/3}\omega(m_x)^{-1}\right).$$
		Observe also that there exists a constant $\rho$ such that $\mathcal{C}_i = 1$ for $i \le \rho m_x$
		and $\lvert\{i : \mathcal{C}_i  \in \mathcal{I}_k(\omega)\}\rvert = \Theta\left(\on ^ {3 \cdot 2 ^{k-1}}\right)$ for $ 1 \le k \le \ell$; 
		lastly, from assumption $Z_t = O\left(\frac{n^{4/3}}{\on^{1/2}}\right)$, we obtain
		\begin{align} 		
		\sum_{i = 1}^{m_x} \mathcal{C}_i^2 & \le Z_t + O(\rho m_x) + \sum_{k=1}^{\ell} \frac{\vartheta n^{4/3}}{\omega(n)^{2^{k+1}}} \cdot O\left(\on^{3\cdot 2^{k-1}} \right)  \nonumber \\
		& = O\left(\frac{m_x^{4/3}}{\sqrt{\omega(m_x)}}\right) + O\left( \sum_{k=1}^{\ell} \frac{m_x^{4/3}}{\omega(m_x)^{2^{k-1}}}  \right) \nonumber \\
		& = O\left(\frac{m_x^{4/3}}{\sqrt{\omega(m_x)}}\right) + O\left( \sum_{k=1}^{\ell} \frac{m_x^{4/3}}{\omega(m_x)^{k}}  \right) \nonumber \\
		& = O\left(\frac{m_x^{4/3}}{\sqrt{\omega(m_x)}}\right) + O\left(\frac{m_x^{4/3}}{\omega(m_x)}\right) = O\left(\frac{m_x^{4/3}}{\sqrt{\omega(m_x)}}\right) \label{eq:ub-sos} . 
		\end{align}
		Therefore, if $\mu_x = \E \left[ \sum_{i=1}^{m_x} \mathcal{X}_i \right] $ and $\sigma_x^2 = Var\left( \sum_{i=1}^{m_x} \mathcal{X}_i \right)$, 
		Theorem~\ref{thm:prelim:llt-cor} implies that for any $x \in [\mu_x - \sigma_x, \mu_x + \sigma_x]$, 
		$$\Pr \left[A'(X_t) = x\right] = \Pr \left[\sum_{i=1}^{m_x} \mathcal{X}_i = x\right] = \frac{1}{\sqrt{2\pi} \sigma_x} \exp\left(-\frac{(x-\mu_x)^2}{2\sigma_x^2}\right) + o\left(\frac{1}{\sigma_x}\right) = \Omega \left( \frac{1}{\sigma_x}\right)  .$$
		Similarly, we get that $\Pr \left[A'(Y_t) = y\right] = \Omega (\sigma_y^{-1})$ for any $y \in [\mu_y - \sigma_y, \mu_y + \sigma_y]$, with $\mu_y$ and $\sigma_y$ defined analogously. 
		Without loss of generality, suppose $\sigma_y \le \sigma_x$.
		Since $\mu_x = \mu_y$, for $x \in \left[\mu_x - \sigma_y, \mu_x + \sigma_y\right]$, we obtain
		$$
		\min \left\{\Pr \left[A'(X_t) = x \right], \Pr \left[A'(Y_t) = x \right]\right\} = \Omega \left( \frac{1}{\sigma_x}\right).
		$$
		Hence, we can couple $(A'(X_t), A'(Y_t))$ so that $\Pr [A'(X_t) = A'(Y_t) = x] = \Omega(\sigma_x^{-1})$ for all $x \in [\mu_x - \sigma_y, \mu_x + \sigma_y]$.
		Consequently, under this coupling,
		$$
		\Pr \left[ A'(X_t) = A'(Y_t)   \right] =  \Omega\left(\frac{\sigma_y}{\sigma_x}\right).
		$$
		
		Since $\mathcal{X}_1, \dots, \mathcal{X}_{m_x}$ are independent, 
		$\sigma_x^2 = \Theta \left(\sum_{i = 1}^{m_x} \mathcal{C}_i^2\right)$, 
		and similarly  $\sigma_y^2 = \Theta \left(\sum_{i = 1}^{m_y} \mathcal{C'}_i^2\right)$.
		Hence, inequality~\eqref{eq:ub-sos} gives an upper bound for $\sigma_x^2$; 
		meanwhile, a lower bound for $\sigma_y^2$ can be obtained by counting components in the largest interval:
		$$ 
		\sum_{i = 1}^{m_y} \mathcal{C'}_i^2 
		\ge \sum_{i: \mathcal{C'}_i \in \mathcal{I}_1(\omega)} \mathcal{C'}_i^2 
		\ge \frac{Bn^{4/3}}{\omega(n)^{4}} \cdot \Theta\left(\on^{3} \right) 
		= \Omega\left(\frac{n^{4/3}}{\omega(n)}  \right).
		$$
		Therefore, 
		$$
		\Pr \left[ A'(X_t) = A'(Y_t) \right] 
		= \Omega\left(\frac{n^{2/3}}{\on^{1/2}} \cdot \frac{\omega(m_x)^{1/4} }{m_x^{2/3}}\right) 
		= \Omega \left( \frac{1}{\on}\right),
		$$
		as desired.
	\end{proof}
	
	We are now ready to prove Lemma~\ref{lemma:critical:q<2:second-step-2.5}. 
	
	\begin{proof} [Proof of Lemma~\ref{lemma:critical:q<2:second-step-2.5}]
		Let $C_1$ be a suitable constant that we choose later.
		We wish to maintain the following properties for all $t \le T := C_1 \log \on$:
		\begin{enumerate}
			\item $\mathcal{S}_{\omega}(X_t) = \mathcal{S}_{\omega}(Y_t)$;
			\item $Z_t = O\left(\frac{n^{4/3}}{\on^{1/2}}\right)$;
			\item $\hat{N}_k(X_t, Y_t, \on) = \Omega\left(\on ^ {3 \cdot 2 ^{k-1}}\right)$ 	for all $k \ge 1$ such that $n^{2/3}\on^{-2^k}\rightarrow \infty$;
			\item $I(X_t), I(Y_t) = \Omega(n)$;
			\item $\mathcal{R}_2(X_t), \mathcal{R}_2(Y_t) = O(n^{4/3})$;
			\item $L_1 (X_t) \le \alpha^t L_1(X_0), L_1 (Y_t) \le \alpha^t L_1(Y_0)$ for some constant $\alpha$ independent of $t$.
		\end{enumerate}
		
		By assumption, $X_0$ and $Y_0$ satisfy these properties.
		Suppose that $X_t$ and $Y_t$ satisfy these properties at step $t \le T$. 
		We show that there exists a one-step coupling of the CM dynamics such that
		$X_{t+1}$ and $Y_{t+1}$ preserve all six properties with probability $\Omega \left(\on^{-1} \right)$.
		
		We provide the high-level ideas of the proof first. 
		We will crucially exploit the coupling from Lemma ~\ref{lemma:activation-second-step-2.5}.  
		Assuming $A(X_t) = A(Y_t)$, properties 1 and 2 hold immediately at $t + 1$, 
		and properties 3 and 4 can be shown by a ``standard'' approach used throughout the paper.
		In addition, we reuse simple arguments from previous stages to 
		guarantee properties 5 and 6.
		
		Consider first the activation sub-step. 
		By Lemma~\ref{lemma:activation-second-step-2.5}, $A(X_t) = A(Y_t)$ with probability at least $\Omega(\on^{-1})$.
		If the number of vertices in the percolation is the same in both copies, we can couple the edge re-sampling so that the updated part of the configuration is identical in both copies.
		In other words, all new components created in this step are automatically contained in the component matching $W_{t+1}$; this includes
		all new components whose sizes are greater than $B_\omega$. 
		Since none of the new components contributes to $Z_{t+1}$, we obtain $Z_{t+1} \le Z_t = O\left(\frac{n^{4/3}}{\on^{1/2}}\right)$.
		Therefore,  $A(X_t) = A(Y_t)$ immediately implies properties 1 and 2 at time $t+1$. 
	
		With probability $1/q$, the largest components of $X_t$ and $Y_t$ are activated simultaneously.
		Suppose that this is the case. 
		By Hoeffding's inequality, for constant $K>0$, we have
		\[\Pr \left[ \left\lvert A(X_t) -\E\left[A(X_t)\right] \right\rvert \ge Kn^{2/3} \right] \le \exp\left(-\frac{K^2 n^{4/3}}{\mathcal{R}_2(X_t)}\right).\]  
		Property 5 and the observation that $\E\left[A(X_t)\right] = L_1(X_t) + \frac{n - L_1(X_t)}{q}$ imply that 
		\[\Pr \left[ \left\lvert A(X_t) -L_1(X_t) - \frac{n - L_1(X_t)}{q} \right\rvert \ge Kn^{2/3} \right] = O(1).\]
		By noting that $L_1(Y_0), L_1(X_0) \le n^{2/3} \on$, property 6 implies that 
		\begin{equation}
		\label{eq:second-step-activation}
		\Pr\left[A(X_t) \cdot \frac{q}{n} \le 1 + \frac{(q-1)\on + Kq}{n^{1/3}} \right] = \Omega(1).
		\end{equation}

		We denote $A(X_t) = A(Y_t)$ by $m$. 
		By inequality (\ref{eq:second-step-activation}), with at least constant probability,
		the random graph for both chains is $H \sim G\left(m, \frac{1 + \lambda m^{-1/3}}{m}\right)$, 
		where $\lambda \le \omega(m)$.
		Let us assume that is the case.
		Corollary~\ref{lemma:critical:q<2:maintain-trees} ensures that 
		there exists a constant $b > 0$ such that, with probability at least $1- O(\on^{-3})$, 
		$N_{k}(H,\on) \ge b \on^{3 \cdot 2^{k-1}}$
		for all $k \ge 1$ such that $n^{2/3}\on^{-2^k}\rightarrow \infty$.
		Since components in $H$ are simultaneously added to both $X_{t+1}$ and $Y_{t+1}$, 
		property 3 is satisfied. 
		Moreover, Lemma~\ref{lemma:rg:isolated} implies that with high probability $\Omega(n)$
		isolated vertices are added to $X_{t+1}$ and $Y_{t+1}$, and thus
		property 4 is satisfied at time $t+1$.

		In addition, Lemma~\ref{lemma:rg:critical:exp-crude-bound} and Markov's inequality imply that 
		there exists a constant $C_2$ such that
		$$\Pr \left[\mathcal{R}_2(X_{t+1}) = C_2n^{4/3} \right] \ge \frac{99}{100};$$ 
		By Lemma~\ref{C344}, 
		there exists $\alpha < 1$ such that with 
		at least probability $99/100$ 
		$$L_1(X_{t+1}) \le \max \{\alpha L_1(X_t) ,L_2(X_t) \},$$ 
		where $\alpha$ is independent of $t$ and $n$. 
		Potentially, property 6 may not hold when $\alpha L_1(X_t) < L_1(X_{t+1}) \le  L_2(X_t) = O(n^{2/3})$, 
		but then we stop at this point. (We will argue that in this case all the desired properties are also established shortly.)
		Hence, we suppose otherwise and establish properties 5 and 6 for $X_{t+1}$. 
		Similar bounds hold for $Y_{t+1}$.
		
		By a union bound, $X_{t+1}$ and $Y_{t+1}$ have all six properties with probability at least $92/100$, 
		assuming the activation sub-step satisfies all the desired properties, 
		and thus overall with probability $\Omega \left(\on^{-1} \right)$.
		Inductively, the probability that $X_T$ and $Y_T$ satisfy the six properties is
		\[O\left(\on\right)^{-C_1 \log \on} = \exp\left( \log O(\on)^{-C_1 \log \on}\right) = \exp \left(- O\left( (\log \on)^2 \right) \right).\]
		
		Suppose $X_T$ and $Y_T$ have the six properties. By choosing $C_1 > 1 / \log \frac{1}{\alpha}$,
		properties 5 and 6 imply
		\[\mathcal{R}_1(X_T) = L_1(X_T)^2 + \mathcal{R}_2(X_T) \le \left( \alpha^{C_1 \log \on} n^{2/3} \on \right)^2 + O(n^{4/3}) = O(n^{4/3}),\]
		and $\mathcal{R}_1(Y_T) = O(n^{4/3})$. While the lemma almost follows from these properties, 
		notice that property 3 does not match the desired bounds on the components in the lemma statement. To fix this issue, we perform one additional step of the coupling.
		
		Consider the activation sub-step at $T$. Assume again $A(X_T) = A(Y_T) =: m'$.
		By Hoeffding's inequality, for some constant $K'$, we obtain
		\begin{equation}
		\label{eq:second-step-activation2}
		\Pr\left[ \left| m' \cdot \frac{q}{n} - 1 \right| > \frac{K'}{n^{1/3}} \right]
		=\Pr\left[ \left| m' - \frac{n}{q} \right| > K'n^{2/3} \right] \le \exp\left( \frac{-{K'}^2 n^{4/3}}{\mathcal{R}_1(X_T)}\right) = O(1). 
		\end{equation}

		Let $\lambda' := (m'qn^{-1} - 1)\cdot m'^{1/3}$. Inequality (\ref{eq:second-step-activation2}) implies with at least constant probability
		the random graph in the percolation step is $H' \sim G\left(m', \frac{1 + \lambda' m'^{-1/3}}{m'}\right)$, 
		where $\lambda' \le K'$ and $m' \in (n/2q,n)$.
		If so, Corollary~\ref{lemma:critical:q<2:maintain-trees} ensures with high probability
		$\hat{N}_k(X_{T+1}, Y_{T+1}, \on^{1/2})= \Omega\left({\on}^{3 \cdot 2^{k-2}}\right)$
		for all $k \ge 1$ such that $n^{2/3}\on^{-2^{k-1}}\rightarrow \infty$.
		
		By the preceding argument, with $\Omega(\on^{-1})$ probability, the six properties are still valid at step $T+1$ , so the proof is complete.
		We note that if we had to stop earlier because
		property 6 did not hold, we perform one extra step (as above) to ensure that 
		$\hat{N}_k(X_{T+1}, Y_{T+1}, \on^{1/2})= \Omega({\on}^{3 \cdot 2^{k-2}})$
		for all $k \ge 1$ such that $n^{2/3}\on^{-2^{k-1}}\rightarrow \infty$.
	\end{proof}

	\subsection{A four-phase analysis using random walks couplings: proof of Lemma~\ref{lemma:critical:q<2:second-step-third}}
	\label{BigProof}
	
	We introduce first some notation that will be useful in the proof of Lemma \ref{lemma:critical:q<2:second-step-third}. Let $S(X_0) = \emptyset$, and given $S(X_t)$, $S(X_{t+1})$ is obtained as follows:
	\begin{enumerate}[(i)]
		\item $S(X_{t+1})= S(X_{t})$;
		\item every component in $S(X_{t})$ activated by the CM dynamics at time $t$ is removed from $S(X_{t+1})$; and
		\item the largest new component (breaking ties arbitrarily) is added to $S(X_{t+1})$.
	\end{enumerate}
	
	Let $\mathcal{C}(X_t)$ denote the set of connected components of $X_t$ and note that $S(X_t)$ is a subset of $\mathcal{C}(X_t)$; we use $|S(X_t)|$ to denote the total number of vertices of the components in $S(X_t)$.
	Finally,  let 
	$$
	Q(X_t) = \sum_{\mathcal{C} \in \mathcal{C}(X_t)\setminus S(X_{t})} |\mathcal{C}|^2.
	$$
	
	In the proof of Lemma~\ref{lemma:critical:q<2:second-step-third}, we use the following lemmas. 
	
	\begin{lemma}
		\label{lemma:critical:q<2:maintain-var} 
		Let $r$ be an increasing positive function such that $r(n) = o(n^{1/15})$ and let $c > 0$ be a sufficiently large constant.
		Suppose $|S(X_t)| \le c t n^{2/3} r(n)$, $Q(X_t) \le t n^{4/3} r(n) + O(n^{4/3})$ and $t \le r(n)/\log r(n)$.
		Then, with probability at least $1 - O\left(r(n)^{-1}\right)$, 
		$|S(X_{t+1})| \le c(t+1) n^{2/3} r(n)$ and $Q(X_{t+1}) \le (t+1) n^{4/3} r(n) + O(n^{4/3})$.
	\end{lemma}	
	
	
	
	\begin{lemma}
		\label{lemma:critical:q<2:estimate-lambda}
		Let $f$ be a positive function such that $f(n)=o(n^{1/3})$.
		Suppose a configuration $X_t$ satisfies $\mathcal{R}_1(X_t) = O\left(n^{4/3} f(n)^2 (\log f(n))^{-1}\right)$.
		Let $m$ denote the number of activated vertices in this step, and $\lambda:=(mq/n - 1)\cdot m^{1/3}$.
		With probability $1-O\left(f(n)^{-1}\right)$,
		$m \in (n/2q,n)$ and $|\lambda| \le f(n)$. 
	\end{lemma}
	
	\begin{lemma}
		\label{lemma:critical:q<2:iterative-process} 
		Let $g$ and $h$ be two increasing positive functions of $n$.
		Assume $g(n) = o(n^{1/6})$.
		Let $X_t$ and $Y_t$ be two random-cluster configuration such that
		$\hat{N}_{k}(X_t, Y_t, g) \ge b g(n)^{3 \cdot 2^{k-1}}$ 
		for some fixed constant $b>0$ independent of $n$
		and for all $k \ge 1$ such that $n^{2/3}g(n)^{-2^k}\rightarrow \infty$.
		Assume also that $Z_t \le C n^{4/3}h(n)^{-1}$
		for some constant $C>0$.
		Lastly, assume $I(X_t), I(Y_t) = \Omega(n)$.
		Then for every positive function $\eta$ there exists a coupling for the activation sub-step of the components of $X_t$ and $Y_t$
		such that 
		\[\Pr[A(X_t) = A(Y_t)] \ge 1- 4e^{-2\eta(n)}- \sqrt{\frac{g(n)\eta(n)}{h(n)}} - \frac{\delta}{g(n)},\]
		for some constant $\delta > 0$ independent of $n$.
	\end{lemma}

	The proofs of these lemmas are given in Section~\ref{subsec:aux:rw}.
	In particular, as mentioned, to prove Lemma~\ref{lemma:critical:q<2:iterative-process} we use a precise estimate on the maximum of a random walk on $\Z$ with steps of different sizes (see Theorem~\ref{thm:prelim:rw-coupling}).
	
	\begin{proof}[Proof of Lemma~\ref{lemma:critical:q<2:second-step-third}]
		The coupling has four phases:
		in phase 1 we will consider $O(\log \log \log \log n)$ steps of the coupling,
		$O(\log \log \log n)$ steps in phase 2,
		$O(\log \log n)$ steps in phase 3 
		and phase 4 consists of $O(\log n)$ steps. 
		
		We will keep track of the random variables $ \mathcal{R}_1(X_t), \mathcal{R}_1(Y_t), I(X_t), I(Y_t), Z_t$ and $\hat{N}_{k}(t, g)$ for a function $g$ we shall carefully choose for each phase, 
		and use these random variables to derive bounds on the probability of various events. 
		
		\textbf{Phase 1.} 
		We set $g_1(n) =\on^{1/2}$ and $h_1(n) = K^2\on^{1/2}$ where $K>0$ is a constant we choose. 
		Let $a_1 := 1 - \frac{1}{2q}$ and
		let $T_1:=-12\log_{a_1} (\log \log \log n)$, and we fix $t < T_1$. 
		Suppose we have $\mathcal{R}_1(X_t) + \mathcal{R}_1(Y_t) \le C_1 n^{4/3}$,
		$I(X_t), I(Y_t) = \Omega(n)$, 
		and $\hat{N}_{k}(t, g_1) = \Omega(g_1(n)^{3 \cdot 2^{k-1} })$ for all $k \ge 1$ such that $n^{2/3}g_1(n)^{-2^k}\rightarrow \infty$, 
		where $C_1 > 0$ is a large constant that we choose later.
		
		By Lemma~\ref{lemma:critical:q<2:iterative-process}, for a sufficiently large constant $K>0$,
		we obtain a coupling for the activation of $X_t$ and $Y_t$ such that 
		the same number of vertices are activated in $X_t$ and $Y_t$, with probability at least
		\[1- 4e^{-2K}- \sqrt{\frac{K\on^{1/2}}{K^2\on^{1/2}}} - \frac{\delta}{\on^{1/2}} \ge 1 - \frac{1}{16q^2}.\]
		By Lemma~\ref{lemma:critical:q<2:estimate-lambda}, 
		$A(X_t) \in (n/2q,n)$ and $\lambda := (A(X_t)q/n - 1)\cdot A(X_t)^{1/3} = O(1)$ 
		with probability at least $ 1 - \frac{1}{16q^2}$.
		It follows a union bound that $A(X_t) = A(Y_t)$, $A(X_t) \in (n/2q,n)$ and $\lambda = O(1)$ 
		with probability $1 - \frac{1}{8q^2}$.
		We call this event $\mathcal{H}^1_t$. 

		Let $D_t'$ denote the inactivated components in $D(X_t) \cup D(Y_t)$ at the step $t$,
		and $M_t'$ the inactivated components in $M(X_t) \cup M(X_t)$.
		Observe that 
		$$\E\left[\sum\nolimits_{\mathcal{C} \in D_t'} |\mathcal{C}|^2\right] 
		= \left(1 -\frac{1}{q} \right) \sum\nolimits_{\mathcal{C} \in D(X_t) \cup D(Y_t)} |\mathcal{C}|^2 
		=  \left(1 -\frac{1}{q} \right) Z_t.$$
		Similarly,  
		$$
		\E\left[\sum\nolimits_{\mathcal{C} \in M_t'} |\mathcal{C}|^2\right] 
		= \left(1 -\frac{1}{q} \right) \sum\nolimits_{\mathcal{C} \in M(X_t) \cup M(Y_t)} |\mathcal{C}|^2 
		=  \left(1 -\frac{1}{q} \right) \left(\mathcal{R}_1(X_t) + \mathcal{R}_1(Y_t) - Z_t \right).
		$$
		Hence, by Markov's inequality and independence between activation of components in $D_t'$ and components in $M_t'$, 
		with probability at least $1/4q^2$, 
		the activation sub-step is such that 
		$$\sum\nolimits_{\mathcal{C} \in D_t'} |\mathcal{C}|^2 \le \left(1-\frac{1}{2q}\right) Z_t,$$ 
		and 
		$$\sum\nolimits_{\mathcal{C} \in M_t'} |\mathcal{C}|^2
		\le \left(1-\frac{1}{2q}\right)
		 \left(\mathcal{R}_1(X_t) + \mathcal{R}_1(Y_t) - Z_t \right).$$
		We denote this event by $\mathcal{H}^2_t$. 
		By a union bound, $\mathcal{H}^1_t$ and $\mathcal{H}^2_t$ happen simultaneously with probability $1/8q^2$.

		Suppose all these events indeed happen; 
		then we couple the percolation step so that the components newly generated in both copies are exactly identical, 
		and we claim that all of the following holds with at least constant probability:
		
		\begin{enumerate}
			\item $\mathcal{R}_1(X_{t+1}) + \mathcal{R}_1(Y_{t+1}) \le C_1 n^{4/3}$;
			\item $Z_{t+1} \le a_1 Z_t$;
				\item $I(X_{t+1}), I(Y_{t+1}) = \Omega(n)$;
			\item $\hat{N}_{k}(t + 1, g_1) = \Omega(g_1(n)^{3 \cdot 2^{k-1}})$ for all $k \ge 1$ such that $n^{2/3}g_1(n)^{-2^k}\rightarrow \infty$.
		\end{enumerate}
		
		First, note that $Z_{t+1}$ can not possibly increase because
		the matching $W_{t+1}$ can only grow under the coupling if indeed $A(X_t) = A(Y_t)$.
		Observe that only the inactivated components in $X_t$ and $Y_t$ would contribute to $Z_{t+1}$, 
		so 
		$$Z_{t+1} = \sum\nolimits_{\mathcal{C} \in D_t'} |\mathcal{C}|^2 \le a_1 Z_t.$$
		
		
		Next, we establish the properties 3 and 4.
		For this, notice that the percolation step is distributed as a
		$H \sim $ $G\left(A(X_t), \frac{1 + \lambda A(X_t)^{-1/3}}{A(X_t)}\right)$ random graph.
		Corollary~\ref{lemma:critical:q<2:maintain-trees} implies 
		$N_{k}(H, g_1) = \Omega( g_1(n)^{3 \cdot 2^{k-1}})$
		for all $k \ge 1$ such that $n^{2/3}g_1(n)^{-2^k}\rightarrow \infty$,
		with probability at least $1 - O\left(g_1(n)^{-3}\right) $. 
		Moreover, Lemma~\ref{lemma:rg:isolated} implies that with high probability $I(H) = \Omega(n)$.
		Since the percolation step is coupled, 
		this implies that
		both $X_{t+1}$ and $Y_{t+1}$ will have all the components in $H$, so we have
		$\hat{N}_{k}(t+1, g_1) = \Omega(g_1(n)^{3 \cdot 2^{k-1}})$ for all $k \ge 1$ such that $n^{2/3}g_1(n)^{-2^k}\rightarrow \infty$,
		and $I(X_{t+1}), I(Y_{t+1}) = \Omega(n)$, w.h.p.

		Finally, assuming that $|\lambda| = O(1)$, 
		by Lemma~\ref{lemma:prelim:critical-var} and Markov's inequality, 
		there exists $C_2 > 0$ such that  
		$\E \left[\mathcal{R}_1(H) \right] = C_2n^{4/3}$
		with probability at least $99/100$.
		Then 
		\begin{align*}
			\mathcal{R}_1(X_{t+1}) + \mathcal{R}_1(Y_{t+1})
			&= \sum\nolimits_{\mathcal{C} \in D_t'} |\mathcal{C}|^2 + \sum\nolimits_{\mathcal{C} \in M_t'} |\mathcal{C}|^2
			+ \mathcal{R}_1(H) \\
		& \le a_1 \left(\mathcal{R}_1(X_{t}) + \mathcal{R}_1(Y_{t})\right) + C_2n^{4/3} 
		\le a_1 C_1 n^{4/3} + C_2n^{4/3} 
		\le C_1 n^{4/3}, 
		\end{align*}
		for large enough $C_1$. 
		A union bound implies that all four properties hold with at least constant probability 
		$\varrho > 0$. 
		
		Thus, the probability that at each step of update all four properties can be maintained throughout Phase 1 is at least
		$$\varrho^{T_1} = \varrho^{-12\log_{a_1} (\log \log \log n)} = \left(\log \log \log n\right)^{-12\log_\varrho (a_1)}.$$
		
		
		If property 2 holds at the end of Phase 1, we have
		$$Z_{T_1} = O\left( \frac{n^{4/3}}{h_1(n)} \cdot a_1^{T_1} \right)
		= O\left( \frac{n^{4/3}}{h_1(n)} \cdot a_1^{ -12 \log_{a_1}(\log \log \log n)} \right) 
		= O\left( \frac{n^{4/3}}{(\log \log \log n)^{12}} \right).$$
	
		To facilitate discussions in Phase 2,
		we show that the two copies of chains satisfy one additional property at the end of Phase 1.
		In particular, there exist a lower bound for the number of components in a different set of intervals.
		We consider the last percolation step in Phase 1. 
		Then, Corollary~\ref{lemma:critical:q<2:maintain-trees} 
		with $g_2(n) := \left( \log\log\log n \cdot \log\log\log\log n \right)^2$ implies
		$\hat{N}_{k}(T_1, g_2) = \Omega(g_2(n)^{3 \cdot 2^{k-1}})$ for all $k \ge 1$ 
		such that $n^{2/3}g_2(n)^{-2^k}\rightarrow \infty$, with high probability.
		
		
		
		Recall $S(X)$ and $Q(X)$ defined at the beginning of Section~\ref{BigProof}.
		In Phase 2, 3 and 4, a new element of the argument is to also control the behavior $S(X_t)$ and $Q(X_t)$. 
		We provide a general result that will be used in the analysis of the three phases:
		\begin{claim}
			\label{longClaim}
			Given positive increasing functions $T, g, h$ and $r$ that tend to infinity and satisfy
			\begin{enumerate}
				\item $g(n) = o(n^{1/6})$;
				\item $T = o(g)$; 
				\item $r(n) = o(n^{1/15})$;
				\item $T(n)^2 \cdot r(n)^2 \leq g(n)^2/\log g(n)$;
				\item $T(n) \leq r(n) / \log r(n)$;
				\item $ g(n) \log g(n) \leq h(n)^{1/3}$.
			\end{enumerate} 
			and random-cluster configurations $X_0, Y_0$ satisfying
			\begin{enumerate}
				\item $Z_0 = O\left( \frac{n^{4/3}}{h(n)} \right)$;
				\item $|S(X_0)|, |S(Y_0)| \le  n^{2/3} r(n) $;
				\item $Q(X_0), Q(Y_0) =  O(n^{4/3}r(n)) $;
				\item $I(X_{0}), I(Y_{0}) = \Omega(n)$;
				\item $\hat{N}_{k}(0, g) = \Omega\left( g(n)^{3 \cdot 2^{k-1}}\right)$ for all $k \ge 1$ such that $n^{2/3}g(n)^{-2^k}\rightarrow \infty$.
			\end{enumerate}
			There exists a coupling of CM steps such that after $T=T(n)$ steps, with $\Omega(1)$ probability,
			\begin{enumerate}
				\item $Z_T = O\left( \frac{n^{4/3}}{a^{T(n)}} \right)$;
				\item $|S(X_T)|, |S(Y_T)| =  O(n^{2/3} r(n) T(n)) $;
				\item $Q(X_T), Q(Y_T) =  O(n^{4/3}r(n)T(n)) $;
				\item $I(X_{T}), I(Y_{T}) = \Omega(n)$;
				\item If a function $g'$ satisfies $g'\geq g$  and $g'(n) = o(n^{1/3})$, 
				then $\hat{N}_{k}(T, g') = \Omega\left( g'(n)^{3 \cdot 2^{k-1}}\right)$ for all $k \ge 1$ such that $n^{2/3}g'(n)^{-2^k}\rightarrow \infty$.
			\end{enumerate}
		\end{claim}
		
		Proof of this claim is provided later in Section~\ref{subsec:aux:rw}. 
	
		\textbf{Phase 2.}
		Let $a= q/(q-1)$.
		For Phase 2, we set $g_2(n) = \left( \log\log\log n \cdot \log\log\log\log n \right)^2$,
		$g_3(n) = (\log\log n \cdot \log \log \log n )^2$,
		$h_2(n) = \left(\log \log \log n\right)^{12} $, 
		$r_2(n) = 13 \log_a \log \log n \cdot \log \log_a \log \log n$
		and  $T_2 = T_1 + 12 \log_a \log \log n$.
		Notice these functions satisfy the conditions of \ref{longClaim}: 
		\begin{enumerate}
			\item $g_2(n) = o(n^{1/6})$;
			\item $T_2 - T_1 = o(g_2(n))$;
			\item $r_2(n) = o(n^{1/15})$;
			\item $(T_2 - T_1)^2 r_2(n)^2 \le 10^6(\log_a \log \log n)^4 ( \log \log_a \log \log n)^2 \le g_2(n)^2/\log g_2(n)$;
			\item $T_2 - T_1  = 12 \log_a \log \log n 
			\le r_2(n)/\log r_2(n)$;
			\item $g_2(n)/\log g_2(n) 
			\le \left(\log \log \log n\right)^{4} = h_2(n)^{1/3}$.
		\end{enumerate}
		Suppose that we have all the desired properties from Phase 1, so at the beginning of Phase 2 we have: 
		\begin{enumerate}
			\item 	$Z_{T_1} = O\left( \frac{n^{4/3}}{(\log \log \log n)^{12}} \right) = O\left( \frac{n^{4/3}}{h_2(n)}\right)$;
			\item 	$S(X_{T_1}) \le \sqrt{ R_1(X_{T_1})} \le n^{2/3}r_2(n)$, $S(Y_{T_1}) = \sqrt{ R_1(Y_{T_1})} \le n^{2/3}r_2(n)$;
			\item 	$I(X_{T_1}) = \Omega(n), I(Y_{T_1}) = \Omega(n)$;
			\item $ Q(X_{T_1}) \le R_1(X_{T_1}) = O(n^{4/3})$, $Q(Y_{T_1}) \le R_1(Y_{T_1}) = O(n^{4/3})$;
			\item $\hat{N}_{k}(T_1, g_2) =\Omega\left( g_2(n)^{3 \cdot 2^{k-1}}\right)$
			for all $k \ge 1$ such that $n^{2/3}g_2(n)^{-2^k}\rightarrow \infty$.
		\end{enumerate}   
		Claim~\ref{longClaim} implies there exists a coupling such that with $\Omega(1)$ probability
		\begin{enumerate}
			\item $Z_{T_2} = O\left( \frac{n^{4/3}}{(\log \log n)^{12}}\right)$;
			\item $\lvert S(X_{T_2}) \rvert \le  n^{2/3} r_2(n) \log_a \log \log n$, $\lvert S(Y_{T_2}) \rvert \le n^{2/3} r_2(n) \log_a \log \log n$;
			\item $ Q(Y_{T_2}) = O(n^{4/3} r_2(n) \log_a \log \log n)$, $Q(X_{T_2}) = O(n^{4/3} r_2(n) \log_a \log \log n)$;
			\item $I(X_{T_2}) = \Omega(n), I(Y_{T_2}) = \Omega(n)$;
			\item  $\hat{N}_{k}(T_2, g_3) = \Omega( g_3(n)^{3 \cdot 2^{k-1}})$
			for all $k \ge 1$ such that $n^{2/3}g_3(n)^{-2^k}\rightarrow \infty$.
		\end{enumerate}
		
		\textbf{Phase 3.}
		Suppose the coupling in Phase 2 succeeds. 
	
		For Phase 3, we set the functions as
		$g_3(n) = \left( \log\log n \cdot \log\log\log n \right)^2$,
		$g_4(n) = (\log n \cdot \log \log n )^2$,
		$h_3(n) = \left(\log \log n\right)^{12} $, 
		$r_3(n) = 20\log_a \log n \cdot \log \log_a \log n$
		and  $T_3 = T_2 + 10 \log_a \log n$.
		Claim~\ref{longClaim} implies there exists a coupling such that with $\Omega(1)$ probability
		\begin{enumerate}
			\item $Z_{T_3} = O\left(  \frac{n^{4/3}}{(\log n)^{10}}\right)$;
			\item $|S(X_{T_3})| = O(n^{2/3} r_3(n) \log_a \log n ), |S(Y_{T_3})| = O(n^{2/3} r_3(n) \log_a \log n)$;
			\item $Q(X_{T_3}) = O(n^{4/3} r_3(n) \log_a \log n)$, $ Q(Y_{T_3}) = O(n^{2/3} r_3(n) \log_a \log n)$; 
			\item $I(X_{T_3}) = \Omega(n), I(Y_{T_3}) = \Omega(n)$,
			\item  $\hat{N}_{k}(T_3, g_4) = \Omega( g_4(n)^{3 \cdot 2^{k-1}})$ for all $k \ge 1$ such that $n^{2/3}g_4(n)^{-2^k}\rightarrow \infty$. 
		\end{enumerate}
		
		\textbf{Phase 4.}
		Suppose the coupling in Phase 3 succeeds.
		Let $C_2$ be a constant greater than $4/3$. 
		We set $g_4(n) = \left( \log n \cdot \log\log n \right)^2$,
		$h_4(n) = \left(\log n\right)^{10} $, 
		$r_4(n) = 2C_2 \log_a n \cdot \log \log_a n$
		and  $T_4 = T_3 + C_2 \log_a n$.
		Claim~\ref{longClaim} implies there exists a coupling such that with $\Omega(1)$ probability
		$Z_{T_4} < 1$.
		Since $Z_{T_4}$ is a non-negative integer-value random variable,
		$\Pr[Z_{T_4} < 1] = \Pr[Z_{T_4} = 0]$. 
		When $Z_{T_4} = 0$, $X_{T_4}$ and $Y_{T_4}$ have the same component structure. 
		
		Therefore, if the coupling in every phase succeeds, 
		$X_{T_4}$ and $Y_{T_4}$ have the same component structure. 
		The probability that coupling in Phase 1 succeeds is $\left(\log \log \log n\right)^{- O(\log_\varrho (1/a_1))}$.
		Conditional on the success of their previous phases, 
		couplings in Phase 2, 3 and 4 succeed respectively with at least constant probability.   
		Thus, the entire coupling succeeds with probability
		\[
		\left(\log \log \log n\right)^{- O(\log_\varrho (1/a_1))}	
		\cdot \Omega(1) \cdot \Omega(1) \cdot \Omega(1) 
		= \left(\frac{1}{\log \log \log n}\right)^{\beta},\]  
		where $\beta$ is a positive constant.
	\end{proof}

	\subsubsection{Proof of lemmas used in Section~\ref{BigProof}}
	\label{subsec:aux:rw}
	
	\begin{proof} [Proof of Claim~\ref{longClaim}]
		We will show that given the following properties at any time $t \le T(n)$,
		we can maintain them at time $t+1$ with probability at least
		$1 - O(g(n)^{-1}) - O(r(n)^{-1})$:
		
		\begin{enumerate}
			\item  $Z_t = O\left( \frac{n^{4/3}}{h(n)} \right)$;
			\item $|S(X_t)|, |S(Y_t)| \le C_3 t n^{2/3} r(n) $ for a constant $C_3 > 0$;
			\item $Q(X_t), Q(Y_t) \le t n^{4/3}r(n) + O(n^{4/3})$;
			\item $I(X_{t}), I(Y_{t}) = \Omega(n)$;
			\item $\hat{N}_{k}(t, g) = \Omega( g(n)^{3 \cdot 2^{k-1}})$ for all $k \ge 1$ such that $n^{2/3}g(n)^{-2^k}\rightarrow \infty$.
		\end{enumerate}
		
		By assumption, $t \le T(n) \le r(n)/\log r(n)$.
		According to Lemma~\ref{lemma:critical:q<2:maintain-var},
		$X_{t+1}$ and $Y_{t+1}$ retain properties 2 and 3 with probability at least $1 - O(r(n)^{-1})$. 
		
		Given properties 1, 4 and 5, 
		Lemma~\ref{lemma:critical:q<2:iterative-process} (with $\eta = \log g(n)/2$) implies
		that there exist a constant $\delta > 0$ and a coupling for the activation sub-step of $X_t$ and $Y_t$ 
		\begin{align*}
		\Pr \left[A(X_t) = A(Y_t) \right] 
		& \ge 1- 4e^{-\log g(n)}- \sqrt{\frac{g(n)\log g(n)}{2h(n)}} - \frac{\delta}{g(n)} \\ 
		& = 1 - O\left(\frac{1}{h(n)^{1/3}} \right) - O\left(\frac{1}{g(n)}\right)
		= 1 - O\left(\frac{1}{g(n)}\right).
		\end{align*}
		Note that condition $ g(n) \log g(n) \leq h(n)^{1/3}$ is used to deduce the inequality above. 
		Suppose $A(X_t)=A(Y_t)$;
		we couple components generated in the percolation step and preclude the growth of $Z_t$. 
		Hence, $Z_{t+1} \le Z_t = O\left( \frac{n^{4/3}}{h(n)} \right)$,
		and property 1 holds immediately. 
		
		Recall that $\mathcal{R}_1(X) = Q(X) + \lvert S(X) \rvert^2$.
		Properties 2 and 3 imply that $\mathcal{R}_1(X_t) = O(t^2 n^{4/3} r(n)^2)$ and $\mathcal{R}_1(Y_t) = O(t^2  n^{4/3} r(n)^2)$.
		Since $t < T(n)$ and $T(n)^2 \cdot r(n)^2 \leq g(n)^2/\log g(n)$, we can upper bound $\mathcal{R}_1(X_t)$ and $\mathcal{R}_1(Y_t) $  by 
		$$O\left(n^{4/3} \left(T(n) \cdot r(n) \right)^2 \right) = O\left( \frac{n^{4/3} g(n)^2 }{\log g(n)} \right).$$
		
		We establish properties 4 and 5 with a similar argument as the one used in Phase 1.
	
		Let $H_t\sim G(A(X_t), n/q)$. Due to Lemma~\ref{lemma:critical:q<2:estimate-lambda} (with $f = g$) and
		Corollary~\ref{lemma:critical:q<2:maintain-trees}  with probability at least 
		$1 - O(g(n)^{-1})$, $N_{k}(H_t,g) = \Omega( g(n)^{3 \cdot 2^{k-1}})$
		for all $k \ge 1$ such that $n^{2/3}g(n)^{-2^k}\rightarrow \infty$. In addition, $I(H_t)=\Omega(n)$ with probability $1-O(n^{-1})$ by Lemma~\ref{lemma:rg:isolated}. 
		Since the coupling adds components in $H_t$ to both $X_{t+1}$ and $Y_{t+1}$, 
		properties 4 and 5 are maintained at time $t+1$, with probability at least $1 - O(g(n)^{-1})$.
		
		A union bound concludes that at time $t+1$ we can maintain all five properties with probability at least 
		$1  - O(g(n)^{-1})- O(r(n)^{-1})$.
		Hence, the probability that $X_{T(n)}$ and $Y_{T(n)}$ still satisfy the listed 5 properties above is
		\[\left[1 - O\left( \frac{1}{g(n)} \right) - O\left( \frac{1}{r(n)} \right)\right]^{T(n)} = 1 - o(1).\]
		
		It remains for us to show the bound for $Z_T$ and 
		that for a given function $g'$ satisfying $g'\geq g$ and $g'(n) = o(n^{1/3})$, 
		then $\hat{N}_{k}(T, g') = \Omega\left( g'(n)^{3 \cdot 2^{k-1}}\right)$ for all $k \ge 1$ such that $n^{2/3}g'(n)^{-2^k}\rightarrow \infty$.
			
		Conditioned on $A(X_t) = A(Y_t)$ for every activation sub-step in this phase, 
		a bound for $Z_{T}$ can be obtained through a first moment method. 
		On expectation $Z_t$ contract by a factor of $\frac{1}{a} =  1 - \frac{1}{q}$ each step.
		Thus, we can recursively compute the expectation of $\E[Z_{T}]$:
		\begin{align}
		\label{eq:eofzt1}
		\E[Z_{T}] &=  \E[\E[Z_{T} \mid Z_{T - 1}]] = \frac{1}{a} \cdot \E[Z_{T - 1} ] 
		= ...= \left( \frac{1}{a}\right)^{T} \E[Z_0] 
		= O\left(\left( \frac{1}{a}\right)^{T} \cdot \frac{n^{4/3}}{h(n)} \right).
		\end{align}
		
		It follows from Markov's inequality that with at least constant probability 
		$$Z_{T} = O\left( \frac{n^{4/3}}{a^{T(n)}} \right).$$
		
		Finally, in the last percolation step in this phase, 
		Corollary~\ref{lemma:critical:q<2:maintain-trees} guarantees that with high probability
		$\hat{N}_{k}(T, g') $ $ =\Omega( g'(n)^{3 \cdot 2^{k-1}})$ for all $k \ge 1$ such that $n^{2/3}g'(n)^{-2^k}\rightarrow \infty$.
		The claim follows from a union bound. 
	\end{proof}
	
	\begin{proof}[Proof of Lemma \ref{lemma:critical:q<2:maintain-var}]
		We establish first the bound for $|S(X_{t+1})|$.
		Suppose $s$ vertices are activated from $S(X_t)$. By assumption 
		$$Q(X_t) \le  t n^{4/3} r(n) + O(n^{4/3}) \le \frac{2 n^{4/3} r(n)^2}{\log r(n)},$$
		for sufficiently large $n$. Hence, Hoeffding's inequality implies that 
		$$
		A(X_t) \le s +\frac{n-|S(X_t)|}{q} + n^{2/3} r(n)  \le \frac{n}{q} +\frac{(q-1)s}{q} + n^{2/3} r(n),
		$$
		with probability at least $1-O(r(n)^{-1})$. 
		
		We consider two cases. First suppose that $ \delta(q-1)s/q \ge n^{2/3} r(n)$, where $\delta > 0$ is a sufficiently small constant we choose later.
		Then,
		$$
		A(X_t) \le \frac{n}{q} +\frac{(1+\delta)(q-1)s}{q} =: M.
		$$
		The largest new component corresponds to the largest component of a $G(A(X_t),q/n)$ random graph.
		Let $N$ be the size of that component, and  
		let $N_M$ be the size of the largest component of a $G\left(M,\frac{1+\varepsilon}{M}\right)$ random graph,
		where $\varepsilon = qM/n - 1$.
		By Fact~\ref{fact:GraphMonotone}, $N$ is stochastically dominated by $N_M$.
		Then by Corollary \ref{lemma:rg:critical:giant-concentration} there exists a constant $c>0$ such that
		\begin{equation}
		\label{eq:critical-concentration}
		\Pr\left[N > (2 + \rho)\varepsilon M\right] \le \Pr[N_M > (2 + \rho)\varepsilon M]  = O({\e}^{-c\varepsilon^3 M}),
		\end{equation} 
		for any $\rho < 1/10$.
		Now,
		\begin{align}
		\varepsilon M 
		&= \frac{(1+\delta)(q-1)s}{n} \left(\frac{n}{q} +\frac{(1+\delta)(q-1)s}{q}\right) \notag \\
		&= \frac{(1+\delta)(q-1)s}{q}  + O\left(\frac{s^2}{n}\right) \notag \\
		&\le \frac{(1+\delta)(q-1)s}{q}  + O(n^{1/3}r(n)^4), \notag \\
		&\le \frac{(1+2\delta)(q-1)s}{q} , \label{eq:boundofeM}
		\end{align}
		where for the second to last inequality we use that $s \le |S(X_t)| = O(n^{2/3}r(n)^2)$,
		and the last inequality follows from the assumptions $\frac{\delta(q-1)s}{q} \ge n^{2/3} r(n)$ and $r(n) = o\left(n^{1/15}\right)$. 
		Also, since $s =  O(n^{2/3}r(n)^2)$ and $r(n) = o\left(n^{1/15}\right)$, 
		\begin{align}	
		\varepsilon^3 M 
		=  \left[\frac{(1+\delta)(q-1)s}{n}\right]^3 \left(\frac{n}{q} +\frac{(1+\delta)(q-1)s}{q}\right) 
		= \Omega\left(\frac{s^3}{n^2} + \frac{s^4}{n^3}\right) 
		= \Omega(r(n)^3). \label{eq:boundofe3M}
		\end{align}
		Hence, \eqref{eq:critical-concentration}, \eqref{eq:boundofeM} and \eqref{eq:boundofe3M} imply
		$$
		\Pr\left[N \ge \frac{(2 + \rho)(1+2\delta)(q-1)s}{q}\right] = {\e}^{-\Omega(r(n)^3)}.
		$$
		Since $q < 2$, for sufficiently small $\rho$ and $\delta$
		$$
		\frac{(2 + \rho)(1+2\delta)(q-1)}{q} < 1.
		$$
		Therefore, $N \le s$ with probability $1 - \exp(-\Omega(r(n)^3))$.
		If this is the case, then $|S(X_{t+1})| \le |S(X_{t})|$ and so by a union bound
		$|S(X_{t+1})| \le c(t+1) n^{2/3} r(n)$ with probability at least $1-O(r(n)^{-1})$.
		
		For the second case we assume $ \frac{\delta(q-1)s}{q} < n^{2/3} r(n)$ and proceed in similar fashion. 
		In this case, Hoeffding's inequality implies with probability at least $1-O(r(n)^{-1})$,
		$$
		A(X_t) \le \frac{n}{q} +(1+1/\delta)n^{2/3} r(n) =: M'.
		$$
		The size of the largest new component, denoted $N'$, is stochastically dominated by the size of the largest component of a $G(M',\frac{1+\varepsilon'}{M'})$ random graph,
		with $\varepsilon' = qM'/n - 1$.
		Now, since we assume $r(n) = o\left(n^{1/15}\right)$, 
		\begin{align*}
		\varepsilon' M' 
		&\le \frac{q(1+1/\delta)r(n)}{n^{1/3}}\left[\frac{n}{q} +(1+1/\delta)n^{2/3} r(n)\right] \\
		&= (1+1/\delta) n^{2/3} r(n) + O(n^{1/3}r(n)^2) \le \frac{c}{3} n^{2/3} r(n),
		\end{align*}
		where the last inequality holds  for large $n$ and a sufficiently large constant $c$. 
		Moreover, 
		\begin{align*}	
		\left({\varepsilon'}\right)^3 M' 
		=  \Omega\left(\frac{r(n)^3}{n}\left[\frac{n}{q} + n^{2/3} r(n) \right]\right) = \Omega(r(n)^3).
		\end{align*}
		Hence, 
		$$
		\Pr\left[N' \ge c n^{2/3} r(n)\right] 
		\le \Pr\left[N' \ge \frac{(2 + \rho)c n^{2/3} r(n)}{3}\right] 
		\le \Pr\left[N' \ge (2 + \rho) \epsilon'M' \right],
		$$
		where $\rho < 1/10$, and by
		Corollary \ref{lemma:rg:critical:giant-concentration} 
		\begin{align*}
			\Pr\left[N' \ge (2 + \rho) \epsilon'M' \right]
			= {\e}^{-\Omega(\left({\varepsilon'}\right)^3 M' )} 
			={\e}^{-\Omega(r(n)^3)}.
		\end{align*}
		Since, 
		$|S(X_{t+1})| \le |S(X_t)| + N'$, 
		a union bound implies that $|S(X_{t+1})| \le c(t+1)n^{2/3} r(n)$ 
		with probability at least $1-O(r(n)^{-1})$ as desired.
		
		Finally, to bound $Q(X_{t+1})$ we observe that
		if $C_1,\dots,C_k$ are all the new components in order of their sizes,
		then by Lemma \ref{lemma:rg:critical:exp-crude-bound} and Markov's inequality:
		$$
		\Pr\left[\sum_{j \ge 2} |C_j|^2 \ge n^{4/3} r(n)\right] = O(r(n)^{-1}).
		$$
		Thus,
		$Q(X_{t+1}) \le Q(X_t) + n^{4/3} r(n) \le (t+1)n^{4/3} r(n)$
		with probability at least $1-O(r(n)^{-1})$ as claimed.
		The lemma follows from a union bound. 
	\end{proof}
	
	\begin{proof}[Proof of Lemma~\ref{lemma:critical:q<2:estimate-lambda}]
		Since $\mathcal{R}_1(X_t) = O\left(n^{4/3} f(n)^2 (\log f(n))^{-1}\right)$, by Hoeffding's inequality
		$$
		A(X_t) \in \left[\frac{n- n^{2/3}f(n)}{q} ,\frac{n+ n^{2/3}f(n)}{q} \right] =: J,
		$$
		with probability at least $1 - O(f(n)^{-1})$.
		The new connected components in $X_{t+1}$ correspond to those of a $G(A(X_t),\frac{1+\varepsilon}{A(X_t)})$ random graph, where
		$\varepsilon = A(X_t) q/n -1$. If $A(X_t) \in J$, then
		\begin{equation}
		\label{eq:boundone}
		-n^{-1/3}f(n) \le \varepsilon \le n^{-1/3}f(n).
		\end{equation}
		
		Since $A(X_t) \in J$ we can also define $m := A(X_t) = \theta n$ for $\theta \in (1/2q,1)$, and $\lambda := \varepsilon m^{1/3}$, so we may rewrite (\ref{eq:boundone}) as
		$$
		-f(n) \le - \theta^{1/3} f(n) \le \lambda \le \theta^{1/3} f(n) \le f(n),
		$$
		and the lemma follows. 
	\end{proof}
	
	An important tool used in the proof of Lemma~\ref{lemma:critical:q<2:iterative-process} is the following 
	coupling on a (lazy) symmetric random walk on $\mathbb{Z}$; its proof is given in Appendix~\ref{Appendix RW}.
	
	\begin{theorem}
		\label{thm:prelim:rw-coupling}
		Let $A > 0$ and let 
		$A \le c_1,c_2,\dots,c_m \le 2A$ be positive integers.
		Let $r \in (0,1/2]$ and consider the sequences of random variables $X_1,\dots,X_m$ and $Y_1,\dots,Y_m$ where
		for each $i = 1,\dots,m$:
		$X_i = c_i$ with probability $r$; $X_i = -c_i$ with probability $r$;  $X_i = 0$ otherwise and $Y_i$ has the same distribution as $X_i$.
		Let $X = \sum_{i=1}^m X_i$ and $Y = \sum_{i=1}^mY_i$. Then for any $d > 0$, there exist a constant $\delta := \delta(r) > 0$ and a coupling of $X$ and $Y$ such that
		\[\Pr[d + 2A \ge X - Y \ge d]  \ge 1- \frac{\delta (d+A)}{A\sqrt{m}}.\]
	\end{theorem}
	
	We note that Theorem~\ref{thm:prelim:rw-coupling} is a generalization of the following more standard fact which will also be useful to us.
	
	\begin{lemma} [\cite{ABthesis}, Lemma 2.18] 
		\label{lemma:binomialcoupling}
		Let $X$ and $Y$ be binomial random variables with parameters $m$ and $r$, where $r \in (0,1)$
		is a constant. Then, for any integer $ y>0$, there exists a coupling $(X, Y)$ such that for a suitable
		constant $\gamma = \gamma(r) > 0$, 
		\[\Pr[X-Y = y] \ge 1 - \frac{\gamma y}{\sqrt{m}}.\]
	\end{lemma}
	
	\begin{proof}[Proof of Lemma \ref{lemma:critical:q<2:iterative-process}]
		
		For ease of notation let $\mathcal{I}_k = \mathcal{I}_{k}(g)$ and $\hat{N}_{k} = \hat{N}_{k}(t,g)$ for each $k \ge 1$.
		Also recall the notations $W_t$, $M(X)$ and $D(X)$ defined in Section~\ref{subsec:couple}.
		Let $\hat{I} (X_t)$ and $\hat{I} (Y_t)$ be the isolated vertices in $W_t$ from $X_t$ and $Y_t$, respectively.
		
		Let $k^* := \min_k \{ k \in \mathbb{Z}  : g(n)^{2^{k}} \ge \vartheta n^{1/3} \}$. 
		The activation of the non-trivial components in $M(X_t)$ and $M(Y_t)$
		whose sizes are not in $\{1\} \cup \mathcal{I}_1 \cup \dots \cup \mathcal{I}_{k^*} $ 
		is coupled using the matching $W_t$. That is, $c \in M(X_t)$ and $W_t(c) \in M(Y_t)$ are activated simultaneously with probability $1/q$. 
		The components in $D(X_t)$ and $D(Y_t)$ are activated independently. 
		After independently activating these components, the  number of active vertices from each copy is not necessarily the same.
		The idea is to couple the activation of the remaining components in $M(X_t)$ and $M(Y_t)$ in way that corrects this difference. 
		
		Let $A_0(X_t)$ and $A_0(Y_t)$ be number of active vertices from $X_t$ and $Y_t$, respectively,
		after the activation of the components from $D(X_t)$ and $D(Y_t)$. 
		Observe that $\E[A_0(X_t)] = \E[A_0(Y_t)] =: \mu$ and that by
		Hoeffding's inequality, for any $\eta(n) > 0$
		\[\Pr \left[  \lvert A_0(X_t) - \mu \rvert \ge \sqrt{\eta(n) Z_t } \right] \le 2e^{-2 \eta(n)}.\]
		
		Recall $Z_t \le \frac{C n^{4/3}}{h(n)}$. Hence, with probability at least $1 - 4\exp\left(-2\eta(n)\right)$,
		\[d_0 := \left|A_0(X_t) - A_0(Y_t)\right|  \le 2 \sqrt{\eta(n) Z_t } \le \frac{2{\sqrt{C \eta(n)}  n^{2/3}}}{\sqrt{h(n)}}.\]
		
		We first couple the activation of the components in $\mathcal{I}_1$, then in $\mathcal{I}_2$ and so on up to $\mathcal{I}_{k^*}$. 
		Without loss of generality, suppose that $d_0 = A_0(Y_t) - A_0(X_t)$.
		If $d_0 \le \frac{\vartheta n^{2/3}}{g(n)^2}$,
		we simply couple the components with sizes in $\mathcal{I}_1$ using the matching $W_t$. 
		Suppose otherwise that $ d_0 > \frac{\vartheta n^{2/3}}{g(n)^2}$.
		Let $A_1(X_t)$ and $A_1(Y_t)$ be random variables corresponding to the numbers of active vertices from $M(X_t)$ and $M(Y_t)$
		with sizes in $\mathcal{I}_1$ respectively.
		By assumption $\hat{N}_1 \ge b g(n)^{3}$. Hence,
		Theorem~\ref{thm:prelim:rw-coupling} implies that
		for $\delta = \delta(q) > 0$,
		there exists a coupling for the activation of the components in $M(X_t)$ and $M(Y_t)$ with sizes in $\mathcal{I}_1$
		such that 
		$$
		d_0 \ge A_1(X_t) - A_1(Y_t) \ge d_0 - \frac{\vartheta n^{2/3}}{g(n)^2} 
		$$
		with probability at least 
		\[1 - \frac{\delta \left( d_0 - \frac{\vartheta n^{2/3}}{2g(n)^2}\right)}{\frac{\vartheta n^{2/3}}{2g(n)^2} \sqrt{b g(n)^{3}}} 
		\ge 1 - \frac{\delta d_0}{\frac{\vartheta n^{2/3}}{2g(n)^2} \sqrt{b g(n)^{3}}}  
		\ge 1 - \frac{4\delta \sqrt{C \eta(n) g(n)} }{\vartheta  \sqrt{b h(n)} } 
		\ge 1 - \sqrt{\frac{\eta(n)g(n)}{ h(n)}} ,\]
		where the last inequality holds for $\vartheta $ large enough. 
		Let $d_1 := \left(A_0(Y_t) - A_0(X_t)\right) + \left(A_1(Y_t)  - A_1(X_t)\right)$. 
		If the coupling succeeds,  
		we have
		$0 \le d_1 \le \frac{\vartheta n^{2/3}}{g(n)^2}$. 
		Thus, we have shown that $d_1 \le \frac{\vartheta n^{2/3}}{g(n)^2}$ with probability at least
		$$\left(1 - 4e^{-2\eta(n)}\right)\left(1 - \sqrt{\frac{\eta(n)g(n)}{ h(n)}}\right) 
		\ge 1 - 4e^{-2\eta(n)} - \sqrt{\frac{\eta(n)g(n)}{ h(n)}}.$$
		
		Now, let $d_k$ be the difference in the number of active vertices
		after activating the components in $\mathcal{I}_k$.
		Suppose that $d_k \le \frac{\vartheta  n^{2/3}}{g(n)^{2^{k}}}$, for $k \le k^*$.
		By assumption, $\hat{N}_{k+1} \ge b g(n)^{3 \cdot 2^{k}}$.
		Thus, 
		using Theorem~\ref{thm:prelim:rw-coupling} again we get that there exists a coupling for the activation of the components in $\mathcal{I}_{k+1}$
		such that
		\[\Pr\left[d_{k+1} 	\le \frac{\vartheta n^{2/3}}{g(n)^{2^{k+1}}} \,\middle\vert\, d_k 	\le \frac{\vartheta  n^{2/3}}{g(n)^{2^{k}}} \right] 
		\ge 1 - \frac{\delta d_k}{\frac{\vartheta n^{2/3}}{2g(n)^{2^{k+1}}} \sqrt{b g(n)^{3 \cdot 2^{k}}}}  
		\ge 1 - \frac{2\delta}{\sqrt{b}g(n)^{2^{k-1}}}.\]
		
		Therefore, there is a coupling of the activation components in $\mathcal{I}_2, \mathcal{I}_3, \dots, \mathcal{I}_{k^*}$ such that
		\[\Pr\left[ d_{k^*} \le n^{1/3} \,\middle\vert\, d_1 \le \frac{\vartheta n^{2/3}}{g(n)^2} \right] \ge \prod_{k = 2}^{k^*}  \left(1 - \frac{\delta'}{g(n)^{2^{k-1}}} \right),\]
		where $\delta' = 2\delta/\sqrt{b}$.
		Note that for a suitable constant $\delta'' > 0$, we have
		\begin{equation}
		\prod_{k \ge 2}^{k^*}  \left(1 - \frac{\delta'}{g(n)^{2^{k-1}}} \right) 
		= \exp\left( \sum_{k \ge 1}^{k^*} \ln\left(1 - \frac{\delta'}{g(n)^{2^{k}}}\right)\right) 
		\ge \exp\left(-\delta'' \sum_{k \ge 1}^{k^*} \frac{1}{g(n)^{2^{k}}}\right),  \notag
		\end{equation}
		and since
		\[\sum_{k \ge 1}^{k^*} \frac{1}{g(n)^{2^{k}}} \le
		\sum_{k \ge 1}^\infty \frac{1}{g(n)^{2^{k}}} \le \sum_{k \ge 1}^\infty  \frac{1}{g(n)^{{k}}}  \le \frac{1}{g(n)^2-g(n)}, \] 
		we get
		\[\prod_{k = 2}^{k^*}  \left(1 - \frac{\delta'}{g(n)^{2^{k-1}}} \right)  \ge \exp\left( -\frac{\delta''}{g(n)^2-g(n)} \right) \ge 1 - \frac{\delta''}{g(n)^2-g(n)}.\]
		
		Finally, we couple $\hat{I}(X_t)$ and $\hat{I}(Y_t)$ to fix $d_{k^*}$. 
		By assumption $I(X_t), I(Y_t) = \Omega(n)$, so $m := |\hat{I}(X_t)| = |\hat{I}(Y_t)| = \Omega(n)$.
		Let $A_I(X_t)$ and $A_I(Y_t)$ denote the total number of activated isolated vertices from $\hat{I}(X_t)$ and $\hat{I}(Y_t)$ respectively.
		We activate all isolated vertices independently, so $A_I(X_t)$ and $A_I(Y_t)$ can be seen as two binomial random variables with the same parameters $m$ and $1/q$.
		Lemma~\ref{lemma:binomialcoupling} gives a coupling for binomial random variables such that for $r \le n^{1/3}$,
		\[\Pr\left[ A_I(X_t) - A_I(Y_t) = r \right] \ge 1 - O\left(\frac{1}{n^{1/6}}\right) = 1 - o\left(\frac{1}{g(n)}\right).\]
		
		Therefore,
		\[\Pr \left[A(X_t) = A(Y_t)\right]  \ge 1 - 4e^{-2\eta(n)} - \sqrt{\frac{\eta(n)g(n)}{ h(n)}} - O\left(\frac{1}{g(n)}\right),\]
		as claimed.
	\end{proof}
	
	\section{New mixing time for the Glauber dynamics via comparison}
	\label{sec:GD}
	
	In this section, we establish a comparison inequality between the mixing times of the CM dynamics and of the \emph{heat-bath Glauber dynamics} 
	for the random-cluster model for a general graph $G = (V,E)$.
	The Glauber dynamics is defined as follows.
	Given a random-cluster configuration $A_t$, one step of this chain is given by:
	\begin{enumerate}[(i)]
		\item pick an edge $e \in E$ uniformly at random; 
		\item replace the current configuration $A_t$ by $A_t\cup \{e\}$ with probability
		$$
		\frac{
			\mu_{G,p,q} (A_t \cup \{e\})
		}
		{
			\mu_{G,p,q} (A_t \cup \{e\}) + 
			\mu_{G,p,q} (A_t \setminus \{e\})
		};
		$$ 
		\item else replace $A_t$ by $A_t \setminus \{e\}$.	
	\end{enumerate} 
	It is immediate from its definition this chain is reversible with respect to $\mu = \mu_{G,p,q}$ and thus converges to it.
	
	The following comparison inequality was proved in \cite{BSmf}:
	\begin{equation}
		\label{eq:CM-GD}
		\gsm(\mathrm{GD}) \le O(m\log m) \cdot \gsm(\mathrm{CM}),
	\end{equation}
	where $m$ denotes the number of edges in $G$, 
	and $\gs(\mathrm{CM})$, $\gs(\mathrm{GD})$ the spectral gaps of the transition matrices of the CM and Glauber dynamics, respectively.
	The standard connection between the spectral gap and the mixing time (see, e.g., Theorem 12.3 in \cite{LPW}) yields
	\begin{equation}
	\label{eq:comp:inequality}
	\taumixGD \le O(m\log m) \cdot \taumixCM \cdot \log \mu_{\mathrm{min}}^{-1},
	\end{equation}
	where $\mu_{\mathrm{min}} = \min_{A \in \Omega} \mu(A)$ with $\Omega$ denoting the set of random-cluster configurations on $G$. 
	In some cases, such as in the mean-field model with $p = \Theta(n^{-1})$,
	$\log \mu_{\mathrm{min}}^{-1} = \Omega(m \log m)$, 
	and a factor of $O(m^2 (\log m)^2)$ is thus lost in the comparison.
	We provide here an improved version of this inequality.
	
	\begin{theorem}
		\label{thm:GDCM-comparison}
		For any $q > 1$ and any $p\in (0,1)$,
		the mixing time of Glauber dynamics for the random-cluster model on a graph $G$ with $n$ vertices and $m$ edges satisfies
		$$
		\taumixGD \le O\left(m n \log n +   p m^2 \log n\cdot\log \frac{1}{\min\{p, 1-p\}}\right)\cdot \taumixCM.
		$$
	\end{theorem}
	
	We note that in the mean-field model, where $m = \Theta(n^2)$ and we take $p = \la/n$ with $\la = O(1)$, this theorem yields that 
	$\taumixGD = O(n^3 (\log n)^2) \cdot \taumixCM$, which establishes Theorem~\ref{thm:intro:comparison} from the introduction and improves by a factor of $O(n)$ the best previously known bound for the Glauber dynamics on the complete graph.
	
	To prove Theorem~\ref{thm:GDCM-comparison} we use the following standard fact.
	
	\begin{theorem}
		\label{thm:mt-spectral-gap}
		Let $P$ be a Markov chain on state space $\Gamma$ with stationary distribution $\pi$.
		Suppose there exist a subset of states $\Gamma_0 \subseteq \Gamma$ and a time $T$, such that
		for any $t\ge T$ and any $x \in \Gamma$ we have
		$
			P^t(x, \Gamma\setminus\Gamma_0) \le \frac{1}{16}.
		$
		Then 
		\begin{equation}
			\label{def:pi0}
			\taumix^P = O\left(T + \gsm(P) \log (8\pi_0^{-1})\right), 
		\end{equation}
		where $\pi_0 := \min_{\omega\in \Gamma_0} \pi(\omega)$.
	\end{theorem}
	Note that $\pi_0$ is the minimum probability of any configuration on $\Gamma_0$. Without the additional assumptions in the theorem, the best possible bound involves a factor of $\pi_{\mathrm{min}} = \min_{A\in \Gamma} \pi(A)$ instead.
	We remark that there are related conditions under which~\eqref{def:pi0} holds; we choose the condition that $P^t(x, \Gamma\setminus\Gamma_0) \le \frac{1}{16}$ for every $x$ and every $t \ge T$ for convenience.
	
	
	
	We can now provide the proof of Theorem~\ref{thm:GDCM-comparison}.

	\begin{proof}[Proof of Theorem \ref{thm:GDCM-comparison}]
		First note that if $p=\Omega(1)$, it 
		suffices to prove that 
		$$\taumixGD = O\left(m n \log n + m^2 \log n \log \frac{1}{\min\{p,1-p\}}\right) \cdot \taumixCM.$$
		This follows from~\eqref{eq:comp:inequality} and the fact that 
		$$
		\mu_{\mathrm{min}} \ge \frac{\min\{p,1-p\}^{m}}{q^{n-1}}
		$$
		since the partition function for the random-cluster model on $G$ satisfies $Z_{G} \le q^n$ (see, e.g., Theorem 3.60 in \cite{Grimmett}).
		
		Thus, we may assume $p \le 1/100$.
		From~\eqref{eq:CM-GD} and the standard relationship between the spectral gap and the mixing time (see, e.g., Theorem 12.4 in \cite{LPW}) we obtain:
		\begin{equation}
			\label{eq:GD-spectral}
			\gsm(\mathrm{GD}) \le \taumixCM \cdot O( m\log n). 
		\end{equation}
		Let $P$ denote the transition matrix of the Glauber dynamics.
		In order to apply Theorem~\ref{thm:mt-spectral-gap}, we have to find a suitable subset of states $\Omega_0 \subseteq \Omega$ and a suitable time $T$ so that 
		$P^t(A, \Omega\setminus\Omega_0) \le \frac{1}{16}$, for every $A \in \Omega$ and every $t \ge T$.
		
		We let $\Omega_0 = \{A \subseteq E:  |A| \le 100mp\}$
		and $T = C m \log m$ for a sufficiently large constant $C > 0$.		
		When an edge is selected for update by the Glauber dynamics, 
		it is set to be open with probability $p/(p+q(1-p))$ if it is a ``cut edge'' 
		or with probability $p$ if it is not; recall that we say an edge $e$ is {open} if the edge is present in the random-cluster configuration.
		Therefore, since $p \ge p/(p+q(1-p))$ when $q > 1$, after every edge has been updated at least once
		the number of open edges in any configuration is stochastically dominated by the number of edges in a $G(n,p)$ random graph.
		By the coupon collector bound, every edge has been updated at least once at time $T$ w.h.p.\ for large enough $C$. Moreover, if all edges are indeed updated by time $T$, the number of open edges in $X_t$ at any time $t \ge T$ is at most $100 m p$
		with probability at least $19/20$ by Markov's inequality.
		Therefore, the Glauber dynamics satisfies condition in Theorem~\ref{thm:mt-spectral-gap} for these choices of $T$ and $\Omega_0$. 
	
		It remains for us to estimate $\pi_0$. 
		Let $\pi_{m}$ denote the probability of the configuration where all the edges are open; then, 
		\begin{equation}
			\label{eq:pim}
		\pi_m = \frac{p^{m} q}{Z_{G}} \ge \frac{p^m}{q^{n-1}},
		\end{equation}
		where the inequality follows from the fact that $Z_G \le q^n$.
		Moreover,
		since $1-p>p$ when $p \le 1/100$, then $\pi_0 \ge qp^{100mp}(1-p)^{m-100mp}/Z_{G}$ and so
		\begin{equation}
			\label{eq:pi0}
			\frac{\pi_0}{\pi_m} \ge \frac{qp^{100mp}(1-p)^{m-100mp}}{p^m q} 
			= \left(\frac{1-p}{p}\right)^{m-100mp}.
		\end{equation}
		Using \eqref{eq:pim}, \eqref{eq:pi0} and the fact that $p \le 1/100$, we obtain:
		\begin{align*}
			\log \frac{1}{\pi_0} &= \log\frac{1}{\pi_m} + \log\frac{\pi_m}{\pi_0} 
			\le (n - 1)\log q + m \log p^{-1} - (m - 100mp) \log \frac{1-p}{p} \\
			& = 100mp\log p^{-1} + m(1- 100p)\log  \frac{1}{1-p}  + O(n)\\
			&\le 100mp\log p^{-1} + \frac{mp(1-100p)}{1-p} + O(n) = O(n + mp(\log p^{-1})).
		\end{align*}
	
		Therefore, from~\eqref{eq:GD-spectral} and Theorem~\ref{thm:mt-spectral-gap} we obtain:
		$$
		\taumixGD \le O(m\log m) + \taumixCM \cdot O( m\log n) \cdot O(n + mp(\log p^{-1}))
		=  O(m \log n \cdot(n+mp\log p^{-1})) \cdot \taumixCM,
		$$
		as claimed.
	\end{proof}
	
	
	For the sake of completeness, we conclude this section with a proof of Theorem~\ref{thm:mt-spectral-gap}.

	\begin{proof} [Proof of Theorem~\ref{thm:mt-spectral-gap}]
	
	For $x \in \Gamma$ and $t \ge T$, we have 
		\begin{align*}
			{\|P^t(x,\cdot)-\pi(\cdot)\|}_{\textsc{\tiny TV}} &= \sum_{y\in \Gamma: P^t(x,y) > \pi(y)} P^t(x,y) -\pi(y)  \\
			&\le \sum_{y \in\Gamma_0: P^t(x,y) > \pi(y)} \pi(y) \left|1 - \frac{P^t(x,y)}{\pi(y)} \right|
			+\sum_{y \notin \Gamma_0: P^t(x,y) > \pi(y)} P^t(x,y) \left|1 - \frac{\pi(y)}{P^t(x,y)} \right| \\
			&\le \pi(\Gamma_0) \max_{y\in \Gamma_0}\left|1 - \frac{P^t(x,y)}{\pi(y)} \right| 
			+  P^t(x, \Gamma \setminus \Gamma_0) \\
			&\le \max_{y\in \Gamma_0}\left|1 - \frac{P^t(x,y)}{\pi(y)} \right|  + \frac{1}{16},
		\end{align*}
		where the last inequality follows from the theorem assumption for $t \ge T$. 
		
		For any $y \in \Gamma$, we have
		$$
		\left|1 - \frac{P^t(x,y)}{\pi(y)} \right|  \le \frac{e^{-\gs(P) \cdot t} }{\sqrt{\pi(x)\pi(y)}};
		$$
		see inequality (12.11) in \cite{LPW}. Hence, for any $x  \in \Gamma_0$ we have
		\begin{equation}
		\label{eq:lb:eq1}
		{\|P^t(x,\cdot)-\pi(\cdot)\|}_{\textsc{\tiny TV}} \le \max_{y\in \Gamma_0} \frac{e^{-\gs(P) \cdot t} }{\sqrt{\pi(x)\pi(y)}} + \frac{1}{16} \le \frac{e^{-\gs(P) \cdot t} }{\pi_0} + \frac{1}{16}.
		\end{equation}
	%
	Letting $\taumix^P(x) =  \min \left\{{t \ge 0 : \|P^t(x,\cdot)-\pi(\cdot)\|}_{\textsc{\tiny TV}}  \le 1/4 \right\}$, 
	we deduce from~\eqref{eq:lb:eq1} that for $x \in \Gamma_0$
	\begin{equation}
	\label{eq:mt:bound}
	\taumix^P(x) \le \max \left\{T, \gsm(P) \log \frac{8}{\pi_0}\right\}.
	\end{equation}
	
	Since $\taumix^P = \max_{x \in \Gamma} \taumix(x)$, it remains for us to provide a bound for $\taumix(x)$ when $x \in \Gamma\setminus\Gamma_0$.
	Consider two copies $\{X_t\}$, $\{Y_t\}$ of the chain $P$.
	For $t>T$ let $\mathbb P$ be the coupling of $X_t$, $Y_t$ such that
	the two copies evolve independently up to time $T$
	and if $ X_T = x'$ and $Y_T = y'$ for some $x', y'\in \Gamma_0$ then the optimal coupling is used so that
	$$
	\mathbb P[X_t \neq Y_t \mid X_T = x', Y_T = y'] = {\|P^{t-T}(x',\cdot)-P^{t-T}(y',\cdot)\|}_{\textsc{\tiny TV}};
	$$
	recall that the existence of an optimal coupling is guaranteed by the coupling lemma (see, e.g., Proposition 4.7 in \cite{LPW}).
	Then, for any $x,y \in \Gamma$
	\begin{align*}
	\mathbb P[X_t \neq Y_t &\mid X_0 = x, Y_0 = y ] \\&\le 
	\mathbb P[X_T \notin \Gamma_0 \mid X_0=x] + \mathbb P[Y_T \notin \Gamma_0\mid Y_0=y]  + \max_{x',y'\in \Gamma_0} \mathbb P[X_t \neq Y_t \mid X_T = x', Y_T = y']  \\
	&\le \max_{x',y'\in \Gamma_0} \mathbb P[X_t \neq Y_t \mid X_T = x', Y_T = y'] + \frac{1}{8}\\
	&\le \max_{x',y'\in \Gamma_0}{\|P^{t-T}(x',\cdot)-P^{t-T}(y',\cdot)\|}_{\textsc{\tiny TV}} + \frac{1}{8}\\
	&\le 2 \max_{x' \in \Gamma_0}{\|P^{t-T}(x',\cdot)-\pi(\cdot)\|}_{\textsc{\tiny TV}} + \frac{1}{8},
	\end{align*}
	where the last inequality follows from the triangle inequality.
	Now, 
	\begin{align*}
	\max_{x \in \Gamma} {\|P^{t}(x,\cdot)-\pi(\cdot)\|}_{\textsc{\tiny TV}}
	&\le \max_{x,y \in \Gamma} \mathbb P [X_t \neq Y_t \mid X_0 = x, Y_0 = y] \le 2 \max_{x' \in \Gamma_0}{\|P^{t-T}(x',\cdot)-\pi(\cdot)\|}_{\textsc{\tiny TV}} + \frac{1}{8} \le \frac 58.
	\end{align*}
	provided $t \ge T + \max_{z \in \Gamma_0} \taumix^P(z)$.
	Using a standard boosting argument (see~(4.36) in~\cite{LPW}) and \eqref{eq:mt:bound} we deduce that $\taumix^P  = O(T + \gsm(P) \log \frac{8}{\pi_0})$ as claimed.	
\end{proof}

\bibliographystyle{plain}
\bibliography{ver2}

\appendix

\section{Proof of the local limit theorem}
\label{app:llt}

In this appendix, we prove Theorem~\ref{thm:prelim:llt-cor}.
First, we introduce some notation.
For a random variable $X$ and $d \in \R$, let $H(X,d) = \E[\langle X^* d \rangle^2]$, where 
$\langle \cdot \rangle$ denotes distance to the closest integer and
$X^*$ is a \textit{symmetrized} version of $X$; i.e., $X^*= X - X'$ where $X'$ is an i.i.d.\ copy of $X$.
Let
$H_m  = \inf_{d \in [\frac{1}{4},\frac{1}{2}]} \,\, \sum_{i=1}^m H(X_i,d)$.
The following local limit theorem is due to Mukhin \cite{Muk} (all limits are taken as $m \rightarrow \infty$).

\begin{theorem}[\cite{Muk}, Theorem 1]
	\label{thm:prelim:llt-main}
	Suppose that the sequence $\frac{S_m-\mu_m}{\sigma_m}$ converges in distribution to a standard normal random variable and that $\sigma_m \rightarrow \infty$.	
	If $H_m \rightarrow \infty$ and there exists $\alpha > 0$ such that $\forall u \in [H_m^{1/4},\sigma_m]$ we have
	$\sum_{i:c_i \le u} c_i^2 \ge \alpha u \sigma_m ,$
	then the local limit theorem holds.
\end{theorem}

Next, we show how to derive Theorem~\ref{thm:prelim:llt-cor} from Theorem~\ref{thm:prelim:llt-main}. The proof involves the following two lemmas.

\begin{lemma}
	\label{lemma:clt1}
	For the random variables satisfying the conditions from Theorem~\ref{thm:prelim:llt-cor},
	$\sigma_m \rightarrow \infty$ and
	$\frac{S_m-\mu_m}{\sigma_m}$ converges in distribution to a standard normal random variable.
\end{lemma}

\begin{proof}
	Observe that 
	$$\sigma_m^2 = r(1-r) \sum_{i=1}^m c_i^2 \ge r(1-r) \sum_{i: c_i \in I_1} c_i^2
	= \Omega\left(\frac{m^{4/3}}{g(m)^4} \cdot g(m)^3\right) 
	= \Omega\left(\frac{m^{4/3}}{g(m)}\right)
	\rightarrow \infty,$$
	and also
	\begin{align*}
		\frac{1}{\sigma_m^3} \sum_{i=1}^m \E[|X_i - \E[X_i]|^3] &= \frac{1}{\sigma_m^3} \sum_{i=1}^m r(1-r) c_i^3
		\le \frac{\sigma^2_m c_m}{\sigma^3_m}	
		= O\left(\frac{c_m}{\sigma_m}\right) \\
		&= O\left( \frac{m^{2/3}g(m)^{-1}}{m^{2/3}g(m)^{-1/2}} \right) 
		= O\left(g(m)^{-1/2}\right) \rightarrow 0.
	\end{align*}
	Hence, the random variables $\{X_i\}$ satisfy Lyapunov's central limit theorem conditions (see, e.g., \cite{D}), 
	and so
	$\frac{S_m-\mu_m}{\sigma_m}$ converges in distribution to a standard normal random variable.	
\end{proof}

\begin{lemma}
	\label{lemma:clt2}
	Suppose $c_1, \dots, c_m$ satisfy the conditions from Theorem~\ref{thm:prelim:llt-cor}.
 	For any $u$ satisfying $\sigma_m \ge u\ge 1$, 
	$\sum_{j:c_j \le u} c_j^2 \ge u \sigma_m / r(1-r)$.
\end{lemma}

\begin{proof}
	We have
	$\sigma_m^2 = r(1-r)\sum_{i=1}^m c_i^2  = O\left( \frac{m^{4/3}}{\sqrt{g(m)}} \right).$ We consider three cases. First,
	if $m^{1/4} \le u \le c_m = O\left(m^{2/3}g(m)^{-1}\right)$, 
	there exists a largest integer $k \in [0, \ell)$ such that $u = O\left( \frac{ \vartheta m^{2/3}}{g(m)^{2^k}} \right)$,
	where $\ell > 0$ is the smallest integer such that $ m^{2/3}g(m)^{-2^\ell} = o(m^{1/4})$.
	Then,
	\[\sum_{i:c_i \le u} c_i^2 \ge \sum_{i:c_i \in I_{k+1}} c_i^2  \ge \frac{\vartheta^2 m^{4/3}}{4 g(m)^{2^{k+2}}} g(m)^{3\cdot 2^{k}} = \frac{\vartheta^2 m^{4/3}}{4 g(m)^{2^{k}}} \gg u \sigma_m;\]
	by $\gg$ we mean that $u \sigma_m$ is of lower order with respect to $\frac{\vartheta^2 m^{4/3}}{4 g(m)^{2^{k}}}$.
	Now, when $\sigma_m \ge u \ge c_m$, we have
	\[\sum_{i:c_i \le u} c_i^2 = \sum_{i=1}^m c_i^2 = \frac{\sigma_m^2}{r(1-r)} \ge \frac{u \sigma_m}{r(1-r)}.\]
	Finally, if $1 \le u \le m^{1/4}$, $u\sigma_m$ is sublinear and so
	$$
	\sum_{i:c_i \le u} c_i^2 \ge \sum_{i=1}^{\rho m} c_i^2 = \rho m \gg m^{1/4} \sigma_m \ge u\sigma_m,
	$$
	as claimed.
\end{proof}

\begin{proof}[Proof of Theorem~\ref{thm:prelim:llt-cor}]
	We check that the $X_i$'s satisfy the conditions from Theorem \ref{thm:prelim:llt-main}. 
	Lemma~\ref{lemma:clt1} implies 
	$\sigma_m \rightarrow \infty$ and $\frac{S_m - \mu_m}{\sigma_m} \rightarrow N(0,1)$;
	by Lemma~\ref{lemma:clt2} we also have that
	for any $u$ satisfying $\sigma_m \ge u\ge 1$, $\sum_{j:c_j \le u} c_j^2 \ge u \sigma_m / r(1-r)$.
	It remains to show that $H_m \rightarrow \infty$.
	
	Now, for $i \le \rho m$, 
	observe that the value of 
	$X^*_i$ equals to $1$ with probability $r(1-r)$, 
	$-1$ with probability $r(1-r)$,
	and $0$ otherwise.
	Then for $1/4 \le d \le 1/2$, 
	$\langle X_i^* d \rangle^2$ evaluates to $d^2$ with probability $2r(1-r)$, and $0$ otherwise.
	Therefore, for $i \le \rho m$ and $1/4 \le d \le 1/2$ we have that 
	$
	\E[\langle X_i^* d \rangle^2] = 2 r(1-r)d^2.
	$
	Thus,
	\[H_m  = \inf_{\frac{1}{4} \le d \le\frac{1}{2}} \sum_{i=1}^{m} H(X_i,d) 
	\ge \inf_{\frac{1}{4} \le d \le\frac{1}{2}} \sum_{i=1}^{\lfloor\rho m \rfloor} H(X_i,d) 
	= \inf_{\frac{1}{4} \le d \le\frac{1}{2}} \sum_{i=1}^{\lfloor\rho m \rfloor} 2r(1-r)d^2  
	= \Omega(m) \rightarrow \infty.\]
	
	
	Since we have shown that the $X_i$'s satisfy all the conditions from Theorem \ref{thm:prelim:llt-main}, the result follows.
\end{proof}

For completeness, we also derive Theorem~\ref{thm:prelim:llt-cor} from first principles (i.e. without using Mukhin's result \cite{Muk}) in Appendix~\ref{app:llt2}. 

\section{Proofs of random walk couplings}
\label{Appendix RW}

Another important tool in our proofs are couplings based on the evolutions of certain random walks.
In this section we consider a (lazy) symmetric random walk $(S_k)$ on $\mathbb{Z}$ with bounded step size, and the first result we present is 
an estimate on $M_k = \max\{S_1, \dots, S_k\}$ which is based on the well-known
reflection principle (see, e.g., Chapter 2.7 in \cite{levin2017markov}).

\begin{lemma}
	\label{lemma:prelim:rw-maximum}
	Let $A > 0$ and let
	$A \le c_1,c_2,\dots,c_n \le 4A$ be positive integers.
	Let $r \in (0,1/2]$ and consider the sequence of random variables $X_1,\dots,X_n$ where
	for each $i = 1,\dots,n$:
	$X_i = c_i$ with probability $r$; $X_i = -c_i$ with probability $r$; and $X_i = 0$ otherwise. 
	Let $S_k = \sum_{i=1}^k X_i $ and $M_k = \max\{S_1,\dots,S_k\}$.
	Then, for any $y \ge 0$
	\[\Pr[M_n \ge y] \ge 2 \Pr[S_n \ge y + 8A + 1].\]	
\end{lemma}

\begin{proof} 
	First, note that
	\begin{align}
	\Pr[M_n \ge y] 
	&= \sum_{k=y}^{4An}	\Pr[M_n \ge y,S_n=k] + \sum_{k=-4An}^{y-1}	\Pr[M_n \ge y,S_n=k] \notag\\
	&= \Pr[S_n \ge y]  + \sum_{k=-4An}^{y-1}	\Pr[M_n \ge y,S_n=k]. \label{eq:prelim:rw-1}		
	\end{align}
	If $M_n \ge y$,
	let $W_n$ be the value of the random walk $\{S_i\}$ the first time
	its value was at least $y$. Then,
	\begin{align}
	\Pr[M_n \ge y,S_n=k]
	&= \sum_{b = y}^{y+4A-1} \Pr[M_n \ge y,S_n=k,W_n=b] \notag\\
	&= \sum_{b = y}^{y+4A-1} \Pr[M_n \ge y,S_n=2b-k,W_n=b] \notag\\
	&= \sum_{b = y}^{y+4A-1} \Pr[S_n=2b-k,W_n=b],\notag		
	\end{align}
	where in the second equality we used the fact that the random walk is symmetric and the last one follows from 
	the fact that $2b - k \ge y$. Plugging this into (\ref{eq:prelim:rw-1}), we get
	\begin{align}
	\Pr[M_n \ge y] 
	&= \Pr[S_n \ge y]  + \sum_{b = y}^{y+4A-1}\sum_{k=-4An}^{y-1} \Pr[S_n=2b-k,W_n=b] \notag\\
	&= \Pr[S_n \ge y]  + \sum_{b = y}^{y+4A-1}\sum_{k=2b-y+1}^{4An} \Pr[S_n = k,W_n=b] \notag\\
	&= \Pr[S_n \ge y]  + \sum_{b = y}^{y+4A-1} \Pr[S_n \ge 2b-y+1,W_n=b] \notag\\
	&\ge \Pr[S_n \ge y]  + \sum_{b = y}^{y+4A-1} \Pr[S_n \ge y+8A+1,W_n=b] \notag,	
	\end{align}
	since $b < y + 4A$. Finally, observe that
	\[\sum_{b = y}^{y+4A-1} \Pr[S_n \ge y+8A+1,W_n=b] = \Pr[S_n \ge y+8A+1]\]
	and so
	\[
	\Pr[M_n \ge y] 
	\ge \Pr[S_n \ge y]  + \Pr[S_n \ge y+8A+1] \ge 2 \Pr[S_n \ge y+8A+1],	
	\]
	as desired.
\end{proof}	


We can now prove Theorem~\ref{thm:prelim:rw-coupling}.

\begin{proof} [Proof of Theorem~\ref{thm:prelim:rw-coupling}]
	Set $\delta = \frac{10}{\sqrt{r}}$.
	Let $D_k = \sum_{i=1}^k (X_i - Y_i)$ for each $k \in \{1,\dots,m\}$. We construct a coupling for $(X,Y)$ by coupling each $(X_k,Y_k)$ as follows:
	\begin{enumerate}
		\item If $D_k < d$, sample $X_{k+1}$ and $Y_{k+1}$ independently.
		\item If $D_k \ge d$, set $X_{k+1} = Y_{k+1}$.
	\end{enumerate}
	Observe that if $D_k \ge d$ for any $k \le m$, then $d + 2A \ge X - Y \ge d$. Therefore,
	\[\Pr[d+2A \ge X-Y \ge d] \ge \Pr[M_m \ge d]\]
	 where $M_m = \max\{D_0,...,D_m\}$. 
	Note that $\{D_k\}$ behaves like a (lazy) symmetric random walk
	until the first time $\tau$ it is at least $d$; after that $\{D_k\}$ stays put.

	Let $\{D'_k\}$ denote such random walk which does not stop after $\tau$, and $M'_m:=\max\{D'_0,...,D'_m\}$.
	Notice 
	$$
	\Pr[M_m \ge d] = \Pr[M'_m \ge d].
	$$
	Since the step size of $\{D'_k\}$ is at least $A$ and at most $4A$,
	by Lemma \ref{lemma:prelim:rw-maximum} for any $d \ge 0$
	\[\Pr[M'_m \ge d] \ge 2 \Pr[D'_m \ge d + 8A + 1].\]
	Let $\sigma^2 = \sum_{i=1}^m \E[(X_i-Y_i)^2] = 4r\sum_{i=1}^m c_i^2$ and $\rho =  \sum_{i=1}^m \E[|X_i - Y_i|^3] = 4r(1+2r)\sum_{i=1}^m c_i^3$.
	By the Berry-Ess\'{e}en theorem for independent (but not necessarily identical) random variables (see, e.g. \cite{Berry1941TheAO}), we get that 
	for any $y \in \R$
	\[\left| \Pr[D'_m > y\sigma] - \Pr[N > y] \right| \le \frac{c \rho}{\sigma^{3}} \le \frac{2cA}{\sigma}.\]
	 where $N$ is a standard normal random variable, 
	 and $c\in [0.4, 0.6]$ is an absolute constant. Then,
	\begin{equation}
	\label{eq:normal}
	\Pr[D'_m >  y \sigma] \ge \Pr[N>y]  - \frac{2cA}{\sigma}.
	\end{equation}
	
	Notice $\sigma \ge 2A \sqrt{rm}$.
	If $d + 8A \ge \sigma$,
	the theorem holds vacuously since
	$$ 1- \frac{\delta (d+A)}{A\sqrt{m}} = 1 - \frac{10(d+A)}{A\sqrt{rm}} 
	< 1 - \frac{d+8A}{A\sqrt{rm}} \le 1 - \frac{\sigma}{A\sqrt{rm}} 
	\le 1-2 <0.$$

	If $d + 8A < \sigma$, since it can be checked via a Taylor's expansion that $2 \Pr[N > y] \ge 1 - \sqrt{\frac{2}{\pi}}y$  for $y < 1$,
	we get from~\eqref{eq:normal}
	\begin{align}
		\Pr[M_m \ge d] \ge		2 \Pr[D'_m > d + 8A] 
		& \ge 2\Pr\left[N>\frac{d + 8A}{\sigma}\right]  - \frac{4cA}{\sigma} \notag\\
	 & \ge 1 - \frac{\sqrt{2/\pi}(d+8A)}{\sigma} - \frac{4cA}{\sigma} \notag\\
	 &\ge 1 -  \frac{9 (d+A)}{\sigma} \notag \\
	 & \ge 1- \frac{\delta (d+A)}{A\sqrt{m}}, \notag
	\end{align}
	as claimed.
\end{proof}



\section{Random graphs estimates}
\label{Appendix RG}

In this section, 
we provide proofs of lemmas which do not appear in the literature.



Recall $G \sim G(n,\frac{1+\lambda n^{-1/3}}{n})$, where $\lambda = \lambda(n)$ may depend on~$n$. 
Both of Lemmas \ref{lemma:prelim:small-cmpt-var} and \ref{lemma:prelim:cmpt-count} are proved using the following precise estimates on the moments of the number of trees of a given size in $G$. 
We note that similar estimates can be found in the literature (see, e.g., \cite{Pittel,PM}); a proof is included for completeness.

\begin{claim}
	\label{claim:prelim:trees}
	Let $t_k$ be the number of trees of size $k$ in $G$.
	Suppose there exists a positive increasing function $g$ such that $g(n) \rightarrow \infty$,
	$|\lambda| \le g(n)$ and
	$i,j,k \le \frac{n^{2/3}}{g(n)^2}$. 
	If $i,j,k\rightarrow \infty$ as $n \rightarrow \infty$, then:
	\begin{enumerate}[(i)]
		\item 	
		$\E[t_k] 
		= \Theta\left(\frac{n}{k^{5/2}}\right)$;
		\label{claim:prelim:trees:exp}
		\item $\var(t_k) \le \E[t_k] + \frac{(1+o(1)) \lambda n^{2/3}}{2\pi k^{3}}$; \label{claim:prelim:trees:var}
		\item For $i \neq j$, $\cov(t_i,t_j) \le \frac{(1+o(1))\lambda n^{2/3}}{2\pi i^{3/2} j^{3/2}}$. \label{claim:prelim:trees:cov}
	\end{enumerate}	
\end{claim}

To prove Lemma~\ref{lemma:prelim:small-cmpt-var}, we also use the following result.

\begin{lemma}
	\label{T7}
	Suppose $\epsilon^3 n \rightarrow \infty$ and $\epsilon = o(1)$.
	Then w.h.p.\ the largest component of $G \sim G\left(n, \frac{1 + \epsilon}{n}\right)$ is the only component of $G$ which contains more than one cycle.
	Also, w.h.p.\ the number of vertices contained in the unicyclic components of $G$ is less than $g(n) \epsilon^{-2}$
	for any function $g(n) \rightarrow \infty$.
\end{lemma}

\begin{proof}
	An equivalent result was established in~\cite{LuczakCycles}  for the $G(n, M)$ model, in which exactly $M$ edges
	are chosen independently at random from the set of all $\binom{n}{2}$ possible edges (see Theorem 7 in \cite{LuczakCycles}).
	The result follows from
	the asymptotic equivalence between the $G(n, p)$ and $G(n, M)$ models when $M = \binom{n}{2} p$ (see, e.g., Proposition 1.12 in \cite{JLR}).	
\end{proof}

\begin{proof}[Proof of Lemma~\ref{lemma:prelim:small-cmpt-var}]
	Let us fix $\alpha > 0$ and consider first the case when $|\lambda|$ is large.  
	If $\lambda < 0$ and $|\lambda| = \Omega(h(n)^{1/2})$, then Lemma \ref{T512} implies that
	$$\E\left[\sum\nolimits_{j:L_j(G) \le B_h} L_j(G)^2\right] \le \E[\mathcal{R}_1(G)] 
	= O\left(\frac{n}{\lambda n^{-1/3}} \right) = O\left(\frac{n^{4/3}}{h(n)^{1/2}} \right).$$
	Similarly, if $\lambda > 0$ and $\lambda = \Omega(h(n)^{1/2})$, then Lemma \ref{T513} implies that 
	$\E[\mathcal{R}_2(G)] = O(n^{4/3}h(n)^{-1/2})$. 
	We may assume $L_1(G) \le B_h$ since otherwise the size of the largest component does not contribute to the sum. 
	Then,
	$$\E\left[\sum\nolimits_{j:L_j(G) \le B_h} L_j(G)^2\right] \le \E[\mathcal{R}_2(G)] + B_h^2 = O\left(\frac{n^{4/3}}{h(n)^{1/2}} \right).$$
	Hence, if $|\lambda| = \Omega(h(n)^{-1/2})$, the result follows from Markov's inequality.
	
	Suppose next $|\lambda| \le \sqrt{h(n)}$.
	Let $t_k$ be the number of trees of size $k$ in $G$ and
	let $\mathcal{T}_{B_h}$ be the set of trees of size at most $B_h$ in $G$. 
	By Claim~\ref{claim:prelim:trees}.\ref{claim:prelim:trees:exp},
	\begin{align}
	\label{eq:prelim:exp-sum-sq-tree}
	\E\left[\sum\nolimits_{\tau \in \mathcal{T}_{B_h}} |\tau|^2\right] 
	&= \sum_{k=1}^{B_h} k^2\E[t_k] = O(n h(n)^2)  + \sum_{k=\lfloor h(n) \rfloor}^{B_h} k^2\E[t_k] \notag\\
	&= O(n h(n)^2) + O(n) \sum_{k=\lfloor h(n) \rfloor}^{B_h} \frac{1}{k^{1/2}}  
	= O\left(\frac{n^{4/3}}{h(n)^{1/2}}\right).
	\end{align}
	By Markov's inequality, we get that $\sum\nolimits_{\tau \in \mathcal{T}_{B_h}} |\tau|^2 \le  An^{4/3}h(n)^{-1/2}$ with probability at least $\gamma$, for any desired $\gamma > 0$ for a suitable constant $A = A(\gamma) > 0$.
	
	All that is left to prove is that the contribution from complex  (non-tree) components is small.
	When $|\lambda| = O(1)$, this follows immediately from the fact that 
	the expected number of complex components is $O(1)$ (see, e.g., Lemma 2.1 in \cite{LuPiWi}). 
	Then, if $\mathcal{C}_{B_h}$ is the set of complex components in $G$ of size at most $B_h$, we have
	\begin{equation}
	\label{eq:prelim:exp-sum-sq-complex}
	\E\left[\sum\nolimits_{C \in \mathcal{C}_{B_h}} |C|^2\right] 
	= O\left(\frac{n^{4/3}}{h(n)^{2}} \right) \E\left[\left\lvert\mathcal{C}_{B_h}\right\rvert\right] 
	= O\left(\frac{n^{4/3}}{h(n)^{2}} \right),\notag
	\end{equation}
	and the result follows again from Markov's inequality and a union bound.
	
	Finally, when $\sqrt{h(n)} \ge |\lambda|  \rightarrow \infty$, Lemma \ref{T7} implies that w.h.p.\ there is no multicyclic component except the largest component
	and that the number of vertices in unicyclic components is bounded by $n^{2/3} g(n)/\lambda^2$, for any function $g(n) \rightarrow \infty$. 
	Hence, w.h.p.,
	$$\sum\nolimits_{C \in \mathcal{C}_{B_h}} |C| \le \frac{n^{2/3} g(n)}{\lambda^2}+ B_h.$$ 
	Setting $g(n)=\lambda^2$, it follows that w.h.p.\
	$$		\sum\nolimits_{C \in \mathcal{C}_{B_h}} |C|^2  \le B_h \left(\frac{n^{2/3} g(n)}{\lambda^2}+ B_h\right) \le \frac{n^{4/3}}{h(n)}.
	$$
	This, combined with~\eqref{eq:prelim:exp-sum-sq-tree}, Markov's inequality
	and a union bound yields the result.
\end{proof}







\begin{proof}[Proof of Lemma \ref{lemma:prelim:cmpt-count}]
	Let $T_B$ be number of trees in $G$ with size in the interval $[B,2B]$; then $|S_B| \ge T_B$. By Chebyshev's inequality for $a>0$:
	\[\Pr[T_B \le \E[T_B] - a \sigma] \le \frac{1}{a^2},\]
	where $\sigma^2 = \var(T_B)$. By Claim~\ref{claim:prelim:trees}.\ref{claim:prelim:trees:exp},
	\[\E[T_B] = \sum_{k=B}^{2B} \E[T_k] \ge \frac{c_1 n}{B^{3/2}}\]
	for a suitable constant $c_1 > 0$. 
	Now,
	\[\var(T_B) = \sum_{k=B}^{2B} \var(t_k) + \sum_{j \neq i: j,i \in [B,2B]} \cov(t_i,t_j).\]
	By Claim~\ref{claim:prelim:trees}.\ref{claim:prelim:trees:exp} and Claim~\ref{claim:prelim:trees}.\ref{claim:prelim:trees:var},
	\[\sum_{k=B}^{2B} \var(t_k) \le \sum_{k=B}^{2B} \E[t_k] +  \sum_{k=B}^{2B} \frac{(1+o(1)) \lambda n^{2/3}}{2\pi k^{3}} = O\left(\frac{n}{B^{3/2}}\right) + O\left(\frac{|\lambda| n^{2/3}}{B^2}\right) = O\left(\frac{n}{B^{3/2}}\right),\]
	where in the last equality we used the assumption that $\lambda = o(n^{1/3})$.
	Similarly, by Claim~\ref{claim:prelim:trees}.\ref{claim:prelim:trees:cov}
	\[\sum_{j \neq i: j,i \in [B,2B]} \cov(t_i,t_j) \le \sum_{j \neq i: j,i \in [B,2B]} \frac{(1+o(1))\lambda n^{2/3}}{2\pi i^{3/2} j^{3/2}}  \le \frac{(1+o(1))|\lambda| n^{2/3}}{2\pi B} =  O\left(\frac{n}{B^{3/2}}\right),\]
	where the last inequality follows from the assumption that $B \le \frac{n^{2/3}}{g(n)^2}$. Hence, for a suitable constant $c_2 > 0$
	\[\var(T_B) \le \frac{c_2 n}{B^{3/2}}\]
	and taking $a = \frac{c_1 n}{2 B^{3/2} \sigma}$ we get
	\[\Pr\left[|S_B| \le \frac{c_1 n}{2B^{3/2}}\right] \le \Pr\left[T_B \le \frac{c_1 n}{2B^{3/2}}\right] 
	\le \left(\frac{2 B^{3/2} \sigma}{c_1 n}\right)^2 \le 
	\frac{4c_2B^{3/2}}{c_1^2 n},\]
	as desired.	\end{proof}

	\begin{proof}[Proof Corollary~\ref{lemma:critical:q<2:maintain-trees}]
	
		Lemma \ref{lemma:prelim:cmpt-count} 
		implies that for a suitable constant $b>0$
		$$
		\Pr\left[N_{k}(X_{t+1},g) < b g(n)^{3 \cdot 2^{k-1}}\right] = O(g(n)^{-3 \cdot 2^{k-1}}),
		$$ 
		for any $k \ge 1$ such that $g(n)^{2^k} = o(m^{2/3})$.
		Observe that
		$$
		\sum_{k \ge 1} \frac{1}{g(n)^{3 \cdot 2^{k-1}}} 
		\le \sum_{i \ge 1} \frac{1}{g(n)^{3i}}
		= O(g(n)^{-3}).
		$$
		Hence,
		a union bound over $k$,  i.e., over the intervals $\mathcal{I}_k(g)$, implies that, with probability at least  $1-O(g(n)^{-3})$,
		$N_{k}(X_{t+1},g) \ge b g(n)^{3 \cdot 2^{k-1}}$
		for all $k \ge 1$ such that $n^{2/3}g(n)^{-2^k}\rightarrow \infty$, as claimed.
	\end{proof}

\begin{proof}[Proof of Claim \ref{claim:prelim:trees}]
	Let $c = 1+\lambda n^{-1/3}$. The following combinatorial identity follows immediately from the fact that there are exactly $k^{k-2}$ trees of size $k$.
	\[\E[t_k] =  \binom{n}{k} k^{k-2} \left(\frac{c}{n}\right)^{k-1}\left(1-\frac{c}{n}\right)^{k(n-k)+ \binom{k}{2}-k+1}.\]
	Using the Taylor expansion for $\ln(1-x)$ and the fact that $k = o(n^{2/3})$, we get
	\begin{align}
	\frac{n!}{(n-k)!} &= n^k \, \prod_{i=1}^{k-1}  \left(1-\frac{i}{n}\right) 
	=n^k \exp\left(-\frac{k^2}{2n}-\frac{k^3}{6n^2}+o(1)\right).
	\end{align}
	Similarly,
	\begin{align}
	\left(\frac{c}{n}\right)^{k-1} &= \frac{1}{n^{k-1}} \exp\left(\frac{\lambda k}{n^{1/3}} - \frac{\lambda^2 k}{2 n^{2/3}}+o(1)\right), \notag\\
	\left(1-\frac{c}{n}\right)^{k(n-k)+ \binom{k}{2}-k+1} &= \exp\left(-k-\frac{\lambda k}{n^{1/3}}+\frac{k^2}{2n}+\frac{\lambda k^2}{2n^{4/3}}+o(1)\right).\notag
	\end{align}
	Since $k \rightarrow \infty$, Stirling's approximation gives
	\begin{equation}
	\frac{k^{k-2}}{k!} = \frac{(1+o(1)) e^k}{\sqrt{2\pi}k^{5/2}}. \label{eq:prelim:stirling}
	\end{equation}
	Putting all these bounds together, we get
	\begin{align}
	\E[t_k] 
	= \frac{(1+o(1))n}{\sqrt{2\pi} k^{5/2}} \exp\left(-\frac{\lambda^2 k}{2n^{2/3}} + \frac{\lambda k^2}{2n^{4/3}} - \frac{k^3}{6n^2}\right)
	= \Theta\left(\frac{n}{k^{5/2}}\right),
	\label{eq:prelim:tree-exp-approx}
	\end{align}
	where in the last inequality we used the assumptions that $|\lambda| \le g(n)$ and $k \le \frac{n^{2/3}}{g(n)^2}$.
	This establishes part (i).
	
	For part (ii) we proceed in similar fashion, starting instead from the following combinatorial identity:
	\[\E[t_k(t_k-1)] = \frac{n!}{k!k!(n-2k)!} (k^{k-2})^2 \left(\frac{c}{n}\right)^{2k-2} \left(1-\frac{c}{n}\right)^{m},\] 
	where $m = 2\binom{k}{2}-2(k-1)+k^2+2k(n-2k)$ (see, e.g., \cite{Pittel}). Using the Taylor expansion for $\ln (1-x)$, we get
	\begin{align}
	\frac{n!}{(n-2k)!} &= n^{2k} \exp\left(-\frac{2k^2}{n}-\frac{4k^3}{3n^2}+o(1)\right), \notag\\
	\left(\frac{c}{n}\right)^{2k-2} &= \frac{1}{n^{2k-2}} \exp\left(\frac{2\lambda k}{n^{1/3}}-\frac{\lambda^2 k}{n^{2/3}}+o(1)\right),\notag\\
	\left(1-\frac{c}{n}\right)^{m} &= \exp\left(-2k+\frac{2k^2}{n}-\frac{2\lambda k}{n^{1/3}}+\frac{2\lambda k^2}{n^{4/3}}+o(1)\right).\notag
	\end{align}
	These three bounds together with (\ref{eq:prelim:stirling}) imply
	\[\E[t_k(t_k-1)] = \frac{(1+o(1))n^2}{2\pi k^5} \exp\left(-\frac{4k^3}{3n^2}-\frac{\lambda^2k}{n^{2/3}}+\frac{2\lambda k^2}{n^{4/3}}\right).\]
	From (\ref{eq:prelim:tree-exp-approx}), we get 
	\[\E[t_k]^2 = \frac{(1+o(1))n^2}{2\pi k^5} \exp\left(-\frac{\lambda^2 k}{n^{2/3}} + \frac{\lambda k^2}{n^{4/3}} - \frac{k^3}{3n^2} \right).\]
	Hence,
	\begin{align}
	\var(t_k) 
	&= \E[t_k]+\frac{(1+o(1))n^2}{2\pi k^5} \exp\left(-\frac{\lambda^2 k}{n^{2/3}} + \frac{\lambda k^2}{n^{4/3}} - \frac{k^3}{3n^2}\right) \left[\exp\left(\frac{\lambda k^2}{n^{4/3}}-\frac{k^3}{n^2}\right) - 1\right] \notag\\
	&= \E[t_k]+\frac{(1+o(1))n^2}{2\pi k^5}\left[\exp\left(\frac{\lambda k^2}{n^{4/3}}-\frac{k^3}{n^2}\right) - 1\right] \notag\\
	&\le \E[t_k]+\frac{(1+o(1))n^2}{2\pi k^5}\left[\exp\left(\frac{\lambda k^2}{n^{4/3}}\right) - 1\right] \notag\\
	&\le \E[t_k]+\frac{(1+o(1)) \lambda n^{2/3}}{2\pi k^3} \notag
	\end{align}
	where in the second equality we used the assumptions that 
	$|\lambda| \le g(n)$ and 
	$k \le \frac{n^{2/3}}{g(n)^2}$ and for the last inequality we used the Taylor expansion for $e^x$. This completes the proof of part (ii).
	
	For part (iii), let $\ell = i+j$.
	When $i \neq j$ we have
	the following combinatorial identity (see, e.g., \cite{Pittel}):
	\[\E[t_it_j] = \frac{n!}{i!j!(n-\ell)!} i^{i-2} j^{j-2} \left(\frac{c}{n}\right)^{\ell-2} \left(1-\frac{c}{n}\right)^{m'},\] 
	where $m' = \binom{i}{2}-(i-1)+\binom{j}{2}-(j-1)+ij+\ell(n-\ell)$. 
	Using Taylor expansions and Stirling's approximation as in the previous two parts, we get
	\[\E[t_it_j] = \frac{(1+o(1))n^2}{2\pi i^{5/2}j^{5/2}}\exp\left(-\frac{\ell^3}{6n^2}-\frac{\lambda^2 \ell}{2n^{2/3}} + \frac{\lambda \ell^2}{2 n^{4/3}}\right).\]
	Moreover, from (\ref{eq:prelim:tree-exp-approx}) we have
	\[\E[t_i]\E[t_j] = \frac{(1+o(1))n^2}{2\pi i^{5/2} j^{5/2}} \exp\left(-\frac{\lambda^2 \ell}{2n^{2/3}} + \frac{\lambda (i^2+j^2)}{2n^{4/3}} - \frac{i^3+j^3}{6n^2} + o(1)\right),\]
	and so
	\begin{align}
	\cov(t_i,t_j) 
	&= \E[t_it_j]-\E[t_i]\E[t_j] \notag\\
	&= \frac{(1+o(1))n^2}{2\pi i^{5/2} j^{5/2}} \exp\left(-\frac{\ell^3}{6n^2}-\frac{\lambda^2 \ell}{2n^{2/3}} + \frac{\lambda \ell^2}{2 n^{4/3}}\right) 
	\left[1-\exp\left(-\frac{\lambda ij}{n^{4/3}}+\frac{ij\ell}{2n^2}\right)\right] \notag\\
	&= \frac{(1+o(1))n^2}{2\pi i^{5/2} j^{5/2}} \left[1-\exp\left(-\frac{\lambda ij}{n^{4/3}}+\frac{ij(i+j)}{2n^2}\right)\right] \notag\\
	&\le \frac{(1+o(1)) \lambda n^{2/3}}{2\pi i^{3/2} j^{3/2}} \notag
	\end{align}
	where in the third equality we used the assumptions that 
	$|\lambda| \le g(n)$ and 
	$i,j \le \frac{n^{2/3}}{g(n)^2}$
	and the last inequality follows from the Taylor expansion for $e^x$.
\end{proof}

\section{The second proof of the local limit theorem}
\label{app:llt2}
In this appendix, we provide an alternative proof of Theorem~\ref{thm:prelim:llt-cor} that does not use Theorem~\ref{thm:prelim:llt-main}.

\begin{proof}[Proof of Theorem~\ref{thm:prelim:llt-cor}]
	Let $\Phi(\cdot)$ denote the probability density function of a standard normal distribution. 
	We will show for any fixed $a \in \R$,
	\begin{equation}
		\label{eq:llt2}
		\left\lvert \Pr\left[ \frac{S_m - \mu_m }{\sigma_m }= a\right] - \frac{\Phi(a)}{\sigma_m} \right\rvert = o\left(\frac{1}{\sigma_m}\right),
	\end{equation}
	which is equivalent to \eqref{eq:prelim:llt}.

	Let $\phi(t)$ denote the characteristic function for the random variable $(S_m - \mu_m )/\sigma_m$. 
	By applying the inversion formula (see Theorem 3.3.14 and Exercise 3.3.2 in \cite{D}), 
	$$
	\Phi(a) = \frac{1}{2\pi} \int_{-\infty}^\infty e^{-ita} e^{-t^2/2} dt,
	$$
	and
	$$
	\Pr\left[ \frac{S_m - \mu_m }{\sigma_m }= a\right] = \frac{1}{2\pi \sigma_m} \int_{-\pi \sigma_m}^{\pi \sigma_m} e^{-ita} \phi(t) dt.
	$$
	Hence, the left hand side of \eqref{eq:llt2} can be bounded from above by
	$$
	\frac{1}{2\pi \sigma_m} \left[ \int_{-\pi \sigma_m}^{\pi \sigma_m} \left\lvert e^{-ita} \left(\phi(t) -e^{-\frac{t^2}{2}} \right) \right\rvert dt 
	+ 2 \int_{\pi \sigma_m}^\infty e^{-\frac{t^2}{2}} dt \right].
	$$
	Since $|e^{-ita}| \le 1$, it suffices to show that
	for all $\epsilon > 0$ there exists $M>0$ such that if $m >M$ then
	\begin{equation}
		\label{eq:llt3}
		\int_{-\pi \sigma_m}^{\pi \sigma_m} \left\lvert  \phi(t) -e^{-\frac{t^2}{2}} \right\rvert dt 
	+ 2 \int_{\pi \sigma_m}^\infty e^{-\frac{t^2}{2}}  dt \le \epsilon.
	\end{equation}
	We can bound from above the left hand side of \eqref{eq:llt3} by:
	\begin{equation}
		\label{eq:llt:destruction}
		\int_{-A}^{A} \left\lvert  \phi(t) -e^{-\frac{t^2}{2}} \right\rvert dt + 
		2\int_{A}^{ \sigma_m/2} \left\lvert  \phi(t)  \right\rvert dt + 
		2\int_{ \sigma_m /2}^{\pi \sigma_m} \left\lvert  \phi(t)  \right\rvert dt +
		2\int_{A}^{\infty} e^{-\frac{t^2}{2}}  dt.
	\end{equation}
	The division depends on some constant $A$ that we will choose soon. 
	We proceed to bound integral terms in~\eqref{eq:llt:destruction} independently. 

	Lemma~\ref{lemma:clt1} implies that $\frac{S_m-\mu_m}{\sigma_m}$ converges 
	in distribution to a standard normal. Combined with the continuity theorem (see Theorem 3.3.17 in \cite{D}), we get that $\phi(t) \rightarrow e^{-\frac{t^2}{2}}$ 
	as $m \rightarrow \infty$.
	The dominated convergence theorem (see Theorem 1.5.8 in \cite{D}) hence implies that for any $A < \infty$ the first integral of \eqref{eq:llt:destruction} converges to 0.
	We select $M$ large enough so that the integral is less than $\epsilon/4$.

	The last integral of \eqref{eq:llt:destruction} is the standard normal tail that goes to $0$ exponentially fast as $A$ increases (see e.g. Proposition 2.1.2 in \cite{Vershynin}).
	Therefore, we are able to select $A$ large enough so that each tail has probability mass less than $\epsilon/8$.

	To bound the remaining two terms we use the properties of the characteristic function $\phi(t)$.
	By definition and the independence between $X_i$'s, 
	$$\phi(t) = \E \left[\exp\left(it\cdot \frac{S_m - \mu_m}{\sigma_m}\right)\right] = \exp\left(- \frac{it\mu_m}{\sigma_m}\right) \prod_{j=1}^m \phi_j(t),$$ 
	where $\phi_j(t)$ denotes the characteristic function of $X_j/\sigma_m$.
	Since $\exp(-\frac{it \mu_m}{\sigma_m}) $ always has modulo 1, $\lvert \phi(t) \rvert \le \prod_{j=1}^m \lvert \phi_j(t)\rvert $.

	We proceed to bound the third integral of \eqref{eq:llt:destruction}. Note that $\lvert \phi_j(t)\rvert \le 1$ for all $j$ and $t$. Therefore,
	$$
	\lvert \phi(t) \rvert \le \prod_{j=1}^m |\phi_j(t)| \le \prod_{j\le\rho m} |\phi_j(t)|.
	$$
	Notice that the $X_j$'s for $j \le \rho m$ are Bernoulli random variables.
	By periodicity (see Theorem 3.5.2 in \cite{D}),
	$|\phi_j(t)|$ equals to 1 only when $t$ equals to the multiples of $2\pi \sigma_m$.
	For $t \in [\sigma_m/2, \sigma_m\pi]$, $|\phi_j(t)|$ is bounded away from $1$, and
	there exists a constant $\eta<1$ such that  $\lvert \phi_j(t) \rvert \le \eta$. 
	Hence, $\lvert \phi(t) \rvert \le \eta^{\rho m}$. By choosing $M$ to be sufficiently large, we may bound the integral for $m > M$:
	$$
	\int_{ \sigma_m /2}^{\pi \sigma_m} \left\lvert  \phi(t)  \right\rvert dt \le 
	\int_{ \sigma_m /2}^{\pi \sigma_m}  \eta^{\rho m} dt \le 
	\pi \sigma_m \eta^{\rho m} \le m  \eta^{\rho m} \le \frac{\epsilon}{8}.
	$$

	Finally, we bound the second integral of \eqref{eq:llt:destruction}. By the definition of $X_j$, we have
	$$
		\phi_j(t) = r e^{it \cdot \frac{c_j}{\sigma_m}} + (1-r)
		= r\cdot \left(\cos \frac{c_j t}{\sigma_m} + i \cdot\sin \frac{c_j t}{\sigma_m} \right) + 1 - r,
	$$
	where the last identity uses Euler' formula.
	Take the modulo of both sides, 
	\begin{align*}
		\lvert \phi_j(t) \rvert &= \sqrt{r^2 \sin^2 \frac{c_j t}{\sigma_m} + \left(r \frac{c_j t}{\sigma_m} + 1 -r \right)^2}\\
		&= \sqrt{r^2 \sin^2 \frac{c_j t}{\sigma_m} + r^2 \cos^2 \frac{c_j t}{\sigma_m} + (1-r)^2 + 2r(1-r) \cos \frac{c_j t}{\sigma_m}}\\
		&= \sqrt{r^2 +  (1-r)^2 + 2r(1-r) \cos \frac{c_j t}{\sigma_m}} \\
		& = \sqrt{1 - 2r(1-r)\left(1-\cos \frac{c_j t}{\sigma_m}\right)} \\
		& = 1 - r(1-r)\left(1-\cos \frac{c_j t}{\sigma_m}\right) - \frac{1}{2}r^2(1-r)^2\left(1-\cos \frac{c_j t}{\sigma_m}\right)^2 - \dots, 
	\end{align*}
	where the last equality corresponds to the Taylor expansion for $\sqrt{1+y}$ when $\lvert y \rvert \le 1$. 
	We can also Taylor expand $\cos \frac{c_j t}{\sigma_m}$ as 
	$$
	1 -  \frac{c_j^2 t^2}{2\sigma_m^2} + \frac{c_j^4 t^4}{4!\sigma_m^4} - \frac{c_j^6 t^6}{6!\sigma_m^6} + \dots
	$$
	Observe that if $\frac{c_j t}{\sigma_m} < 1$, then we can bound $\cos \frac{c_j t}{\sigma_m}$ from above by $1 -  \frac{c_j^2 t^2}{4\sigma_m^2}$.
	Furthermore, if we keep only the first order term from the expansion for  $\lvert \phi_j(t) \rvert$, we have
	\begin{equation}
		\label{eq:phijt}
		\lvert \phi_j(t) \rvert \le 1 - r(1-r)  \frac{c_j^2 t^2}{4\sigma_m^2}
		\le \exp \left(- r(1-r)  \frac{c_j^2 t^2}{4\sigma_m^2} \right).
	\end{equation}
	
	Note that \eqref{eq:phijt} only holds if $c_j t< \sigma_m$. 
	However, if we fix $\sigma_m$ then for every $t \in [A, \sigma_m/2]$, 
	there always exists a real number $u(t)\ge 1$ such that $u(t) t < \sigma_m \le 2u(t)t$, 
	which implies for all $c_j \le u(t)$, \eqref{eq:phijt} can be established. 
	Now we aggregate all $\lvert \phi_j(t) \rvert$ for which we could claim \eqref{eq:phijt}.
	\begin{equation}
		\label{eq:phit}
		\lvert \phi(t) \rvert 
		\le \prod_{j:c_j \le u(t)} \lvert \phi_j(t) \rvert 
		\le \exp \left(- r(1-r)  \sum_{j:c_j \le u(t)} \frac{c_j^2 t^2}{4\sigma_m^2} \right).
	\end{equation}

	Without loss of generality, we assume $A>1$. Consequently, $u(t) \le \sigma_m$. 
	Lemma~\ref{lemma:clt2} implies for $\sigma_m\ge u(t) \ge 1$, $\sum_{j:c_j \le u(t)} c_j^2 \ge u(t) \sigma_m / r(1-r)$.
	Plugging this inequality into \eqref{eq:phit}, we obtain
	$$
	\lvert \phi(t) \rvert 
	\le \exp\left(-  \frac{ u(t) \sigma_m t^2}{4\sigma_m^2} \right) = 
	\exp \left(-\frac{ t}{8}\cdot  \frac{2u(t)t}{\sigma_m} \right) \le \exp \left( - \frac{ t}{8} \right).
	$$ 
	Therefore, for sufficiently large $A$,
	$$
	\int_{A}^{ \sigma_m/2} \left\lvert  \phi(t)  \right\rvert dt 
	\le \int_{A}^{ \infty} \exp \left( - \frac{ t}{8} \right) dt
	\le \frac{e^{- A/8}}{8}  \le \frac{\epsilon}{8}.
	$$
	Thus we established \eqref{eq:llt3} and the proof is complete.
\end{proof}

\end{document}